\documentclass[12pt]{amsart}
\def\elystyle{0}
\usepackage{amsmath,amssymb,amsthm,amsfonts,enumitem,graphicx,comment,xparse,tikz,pgfmath,dsfont,tikz,tikz-cd,verbatim,subcaption,hyperref,mathrsfs,stmaryrd,filecontents,slashed}
 \ifx\elystyle\undef
 \usepackage[margin=1in]{geometry} 
 \else
 \if\elystyle0
  \usepackage[margin=1in]{geometry} 
 \else
 \fi
 \fi
 \usepackage[style=numeric,
sorting=nyt,
date=year,
isbn=false,
doi=false,
url=false,
maxbibnames=99
]{biblatex}
\AtEveryBibitem{\clearfield{pages}}

\renewcommand{\phi}{\varphi}
  
  \setenumerate[0]{wide,label=\textnormal{(\alph*)}}

 \newcommand{\Z}{\mathbb{Z}}

 \newcommand{\R}{\mathbb{R}}

 \newcommand{\F}{\mathbb{F}}

  \newcommand{\E}{\mathbb{E}}
  \newcommand{\B}{\mathbb{B}}

  \newcommand{\calF}{\mathcal{F}}
  \newcommand{\calG}{\mathcal{G}}

  \newcommand{\ul}{\underline}

\newcommand{\Lap}{\Delta}
  \newcommand{\p}{\partial}

\DeclareMathOperator{\id}{id}

\renewcommand{\div}{\operatorname{div}}
\newcommand{\grad}{\operatorname{grad}}

\ifx\elystyle\undefined

 \else
 \if\elystyle0
  \newtheorem{thm}{Theorem}[section]
  \newtheorem{lem}[thm]{Lemma}
    \newtheorem{defin}[thm]{Definition}

 \fi
 \fi

\newcommand{\g}{\gamma}
\newcommand{\omg}{\omega}
\renewcommand{\o}{\omega}
\renewcommand{\O}{\Omega}

\newcommand{\al}{\alpha}
\newcommand{\be}{\beta}
\newcommand{\ep}{\epsilon}

\newcommand{\de}{\delta}
\newcommand{\si}{\sigma}
\newcommand{\s}{\sigma}
\newcommand{\G}{\Gamma}
\newcommand{\la}{\lambda}
\renewcommand{\th}{\theta}

\newcommand{\rt}{\tilde{\rho}}
\newcommand{\ut}{\tilde{u}}
\newcommand{\ot}{\tilde{\omega}}
\newcommand{\wt}{\tilde{w}}
\newcommand{\vt}{\tilde{v}}
\newcommand{\pt}{\tilde{p}}
\newcommand{\ct}{\tilde{c}}
\newcommand{\Ft}{\tilde{\F}}
\newcommand{\Et}{\tilde{\E}}
\newcommand{\Ht}{\tilde{H}}
\newcommand{\Ct}{\tilde{C}}
\newcommand{\St}{\tilde{S}}
\newcommand{\psit}{\tilde{\psi}}
\newcommand{\chit}{\tilde{\chi}}
\newcommand{\At}{\tilde{A}}
\newcommand{\Bt}{\tilde{B}}

\newcommand{\zt}{\tilde{z}}
\newcommand{\Qt}{\tilde{Q}}

\newcommand{\wb}{\bar{w}}

\newcommand{\ub}{\bar{u}}
\newcommand{\chib}{\bar{\chi}}
\newcommand{\Hb}{\bar{H}}
\newcommand{\Sb}{\bar{S}}
\newcommand{\Pb}{\bar{P}}

\renewcommand{\E}{E}
\renewcommand{\F}{F}
\renewcommand{\a}{a}
\renewcommand{\b}{b}
\newcommand{\pr}{\mathring{p}}
\newcommand{\fr}{\mathring{f}}
\newcommand{\omgr}{\mathring{\o}}
\newcommand{\yr}{\mathring{y}}
\newcommand{\sr}{\mathring{\si}}
\newcommand{\yf}{y_f}

\newcommand{\Wr}{\mathring{W}}
\newcommand{\Qr}{\mathring{Q}}
\newcommand{\Vr}{\mathring{V}}
\newcommand{\Ar}{\mathring{A}}
\newcommand{\rhor}{\mathring{\rho}}

\newcommand{\pul}{\underline{p}}
\newcommand{\oul}{\underline{\o}}
\newcommand{\Eul}{\underline{\E}}
\newcommand{\Ful}{\underline{\F}}
\newcommand{\Sul}{\underline{S}}

\newcommand{\ph}{\hat{p}}
\newcommand{\oh}{\hat{\o}}
\newcommand{\Eh}{\hat{\E}}
\newcommand{\Fh}{\hat{\F}}
\newcommand{\Sh}{\hat{S}}
\newcommand{\Nh}{\hat{N}}

\newcommand{\rs}{\slashed{\rho}}
\newcommand{\us}{\slashed{u}}
\newcommand{\As}{\slashed{A}}
\newcommand{\Bs}{\slashed{B}}

\newcommand{\rc}{\check{\rho}}
\newcommand{\uc}{\check{u}}
\newcommand{\Ac}{\check{A}}
\newcommand{\Bc}{\check{B}}

\newcommand{\mqty}[1]{\begin{pmatrix}#1\end{pmatrix}}
\newcommand{\abs}[1]{|#1|}
\newcommand{\ev}[1]{\langle#1\rangle}
\newcommand{\norm}[1]{||#1||}
\renewcommand{\d}{\mathrm{d}}
\newcommand{\dv}[2][]{\frac{\d #1}{\d #2}}
\newcommand{\pdv}[2][]{\frac{\partial #1}{\partial #2}}
\newcommand{\dvn}[3][]{\frac{\d^{#2} #1}{\d #3^{#2}}}

\begin{document}
\title[Polytropic Hunter-type Solutions]{Hunter-type Implosion Profiles for Energy-supercritical polytropic equations of state}
\author{Ely Sandine}
\thanks{This material is based upon work supported by the National Science Foundation Graduate Research Fellowship Program under Grant No. DGE 2146752. Any opinions, findings, and conclusions or recommendations expressed in this material are those of the author(s) and do not necessarily reflect the views of the National Science Foundation.}
\address{Ely Sandine, Department of Mathematics, University of California, Berkeley, 94720.}
\email{ely\_sandine@berkeley.edu}
\date{8/7/2025}

\maketitle
\begin{abstract}
  We rigorously construct a family of smooth self-similar solutions to the isentropic gravitational Euler-Poisson system with a polytropic equation of state for polytropic indices lying in the full energy-supercritical range, $1<\g<\frac{6}{5}$. The result is an extension of the author's previous construction of Hunter solutions in the isothermal case, $\g=1$, and complements a construction of Larson-Penston-type solutions by Guo-Had\v zi\'c-Jang-Schrecker for the same system in the full mass-supercritical range, $1<\g<\frac{4}{3}$. As an ingredient in the proof, a general framework is introduced for proving local analyticity of solutions to this system in the vicinity of singular points. This framework could be used for other quasilinear self-similar blow-up constructions.
\end{abstract}
\tableofcontents
\section{Introduction}
\label{sec:intro}
The (gravitational) Euler-Poisson system of partial differential equations models the evolution of a compressible gas. We will work under the assumptions that the gas is supported on the Euclidean space $\R_x^3$ and that there is no vacuum. We will let $I_t$ denote a time interval. We shall also assume that the gas is isentropic, and described completely by its density, $\rho\colon I_t\times \R_x^3\to \R_{>0}$, and velocity vector field, $u\colon I_t\times \R_x^3\to \R^3$. It is also assumed that the gas self-gravitates according to Newtonian gravity, and we shall choose units so that the gravitational constant is $1$. We shall also assume that equation of state is polytropic with polytropic index $\g\geq 1$. We shall choose units so that the pressure is given by $\rho^{\g}$. The laws of local conservation of mass and momentum are equivalent (for smooth solutions) to the following closed equations of motion
\begin{equation}
  \label{eq:epsystem}
  \begin{cases}    
    \p_t \rho+\div(\rho u)=0\\
    \p_t u+u\cdot \grad u+\frac{\grad (\rho^\g)}{\rho}+4\pi \grad(\Lap^{-1}(\rho))=0
  \end{cases}.
\end{equation}
From now on we shall refer to this system as the Euler-Poisson system.

We now motivate the critical polytropic indices. From the first equation it is easy to see that the total mass
\begin{equation*}
  M(\rho,u):=\int_{\R_x^3}\rho dx
\end{equation*}
is conserved. One can also check that the physical energy of this system, given by
\begin{equation*}
  E(\rho,u):=
  \begin{cases}
    \int_{\R_x^3} \frac{1}{2}\rho \abs{u}^2+\frac{1}{\g-1}\rho^{\g}-2\pi\abs{\grad \Delta^{-1}\rho}^2 dx & \g>1\\
    \int_{\R_x^3} \frac{1}{2}\rho \abs{u}^2+\rho \log(\rho)-2\pi\abs{\grad \Delta^{-1}\rho}^2 dx & \g=1
  \end{cases},
\end{equation*}
is also conserved.

In addition, the set of solutions to the equation is invariant under the scaling symmetry ($\la>0$)
\begin{equation}
  \label{eq:scaling_symmetry}
  (\rho(t,x),u(t,x))\mapsto (\frac{1}{\la^2}\rho(\frac{t}{\la},\frac{x}{\la^{2-\g}}),\la^{1-\g}u(\frac{t}{\la},\frac{x}{\la^{2-\g}})).
\end{equation}
We note that under this scaling
\begin{align*}
  M(\frac{1}{\la^2}\rho(\frac{t}{\la},\frac{x}{\la^{2-\g}}),\la^{1-\la}u(\frac{t}{\la},\frac{x}{\la^{2-\g}}))&=\la^{4-3\g}M(\rho(t,x),u(t,x)),\\
  E(\frac{1}{\la^2}\rho(\frac{t}{\la},\frac{x}{\la^{2-\g}}),\la^{1-\la}u(\frac{t}{\la},\frac{x}{\la^{2-\g}}))&=\la^{6-5\g}E(\rho(t,x),u(t,x)).
\end{align*}
In particular, we see that if $\g=\frac{4}{3}$, then the mass is invariant under this scaling. If $1\leq \g<\frac{4}{3}$, then the equation is mass-supercritical, in the sense that we do not expect a norm at the same scaling as mass to control small-scale behavior. Similarly, $\g=\frac{6}{5}$ is the energy-critical exponent, and $1\leq \g<\frac{6}{5}$ is the energy-supercritical regime. 

We also shall assume that the motion is radially symmetric and let $r=\abs{x}$. We will use the notations
\begin{equation*}
  \div_r:=\p_r+\frac{2}{r},\qquad \div_r^{-1}f(r):=\frac{1}{r^2}\int_0^r (r')^2 f(r')dr'.
\end{equation*}
Thus, in radial symmetry, we obtain the following system of $2$ partial differential equations in $2$ unknowns:
\begin{equation}
  \label{eq:ep_radial}
  \begin{cases}
    \p_t \rho+\div_r(\rho u)=0\\
    \p_{t} u+u\p_r u+\g \rho^{\g-2}\p_r \rho+4\pi \div^{-1}(\rho)=0
  \end{cases}.
\end{equation}
We shall also restrict attention to self-similar evolutions, i.e. solutions to \eqref{eq:ep_radial} which are invariant under the symmetry \eqref{eq:scaling_symmetry}. This is done by considering the ansatz ($I_t=(-\infty,0)$)
\begin{equation*}
  \rho=\frac{1}{(-t)^2}\rt(\frac{r}{(-t)^{2-\g}}),\quad u=(-t)^{1-\g}\ut(\frac{r}{(-t)^{2-\g}}).
\end{equation*}
We introduce the self-similar coordinate
\begin{equation*}
  y=\frac{r}{(-t)^{2-\g}}.
\end{equation*}
With these conventions, \eqref{eq:ep_radial} takes the form
\begin{equation}
  \begin{cases}
    2\tilde{\rho}+(2-\gamma)y\p_y \tilde{\rho}+\div_y(\tilde{\rho}\tilde{u})=0\\
    (\gamma-1)\tilde{u}+(2-\gamma)y\p_y \tilde{u}+\tilde{u}\p_y \tilde{u}+\gamma \tilde{\rho}^{\gamma-2}\p_y \tilde{\rho}+4\pi \div_y^{-1}(\tilde{\rho})=0.
  \end{cases}
  \label{selfsimsys}
\end{equation}
Fortunately, for $\g\neq \frac{4}{3}$, this non-local system admits a reduction to a (local) system of ODEs. In particular, the first equation can be rewritten as
\[
  \tilde{\rho}=\frac{1}{4-3\gamma}\div_y\left(\tilde{\rho}(\tilde{u}+(2-\gamma)y)\right).
\]
Substituting this into the non-local term of the second equation yields the main system of ODEs we study
\begin{equation}
  \begin{pmatrix}\ut+(2-\gamma)y & \rt\\ \gamma \rt^{\gamma-2} & \ut+(2-\gamma) y\end{pmatrix}\begin{pmatrix}\rt'\\ \ut'\end{pmatrix}+\begin{pmatrix}\frac{2\rt (\ut+y)}{y} \\ (\gamma-1)\ut+\frac{4\pi}{4-3\gamma}\rt(\ut+(2-\gamma)y)\end{pmatrix}=0.
  \label{polytropicselfsimeq}
\end{equation}

As a comment, requiring smoothness at at the origin imposes the boundary conditions
\begin{equation}
  \label{eq:selfsim_boundary_conds}
  \rt'(0)=0,\qquad \ut(0)=0,\qquad \ut'(0)=-\frac{2}{3}.
\end{equation}

This system has two explicit solutions. We follow the conventions of \cite{ghj}, and use the name ``Friedman solution'' to refer to
\begin{equation*}
  (\rt_F,\ut_F)=(\frac{1}{6\pi},-\frac{2}{3}y).
\end{equation*}
Similarly, we use the label ``far-field solution'' to refer to
\begin{equation}
  \label{eq:farfield_def}
  (\tilde{\rho}_{f},\tilde{u}_f)=\left(k y^{-\frac{2}{2-\gamma}},0\right),
\end{equation}
where the constant $k$ is defined as
\begin{equation}
  \label{eq:kdef}
  k:=(\frac{\gamma(4-3\gamma)}{2\pi(2-\gamma)^2})^{\frac{1}{2-\gamma}}.
\end{equation}
Neither of these models implosion at a point. The Friedman solution models mass-density blowing up at all locations simultaneously. The mass-density of the far-field solution is singular at the origin and does not evolve. In this work we study smooth self-similar solutions to \eqref{polytropicselfsimeq} with the mass density radially decaying, in pursuit of the larger goal of gaining a more detailed understanding of the local structure of the type of singularities which would develop from regular and localized initial data.

\subsection{Review of literature  and statement of the theorem}
We now discuss the existing literature on non-explicit smooth solutions to \eqref{polytropicselfsimeq}. Some sources are omitted, and further discussions can be found, for instance, in \cite{ghj}, \cite{ghjs} and \cite{sandine}.

Historically, the isothermal case, $\g=1$, has receieved the most attention. In this case, the self-similar equations were first derived by Larson (\cite{larson}) and Penston (\cite{penston}) independently. Moreover, through the use of numerical integration they identified a smooth solution, known now as the Larson-Penston solution. Hunter (\cite{hunter}) numerically found an additional discrete family of solutions. Numerical mode stability analyses of the Larson-Penston and Hunter solutions were carried out in \cite{hanawanakayama} and \cite{maedaharada}. These numerics support the hypothesis that the Larson-Penston solution is a stable description of gravitational collapse and the $i$th Hunter solution ($i\in \Z_{\geq 1}$) is codimension $i$ stable. In particular, the first Hunter solution is believed to play a role in critical collapse theory (see, for instance, \cite{haradamaedasemelin} for a discussion).

The first rigorous mathematical result on the isothermal self-similar system is \cite{ghj}, in which the authors proved existence of a solution whose properties match the Larson-Penston solution. We comment that this result was preceeded by \cite{continuedcollapse}, a study of a dust-like, non-self-similar, regime of gravitational collapse for the Euler-Poisson system. In \cite{sandine}, for each $i$ \emph{sufficiently large} a solution whose properties matches the $i$th Hunter solution was constructed. We comment that constructing the $i$th Hunter solutions for small $i$ (in particular $i=1$) remains an open problem.

We now discuss the polytropic, mass-supercritical regime. The first numerical work in this area was by Yahil (\cite{yahil}) who numerically constructed analogs of the Larson-Penston solution for $\frac{6}{5}<\g<\frac{4}{3}$. This work was in part motivated by the observation that for the isothermal case, the Larson-Penston solution has infinite mass and energy, but for $\frac{6}{5}<\g<\frac{4}{3}$ the analog would have finite energy (but still infinite mass).

These Larson-Penston-type solutions studied by Yahil were constructed in the full mass-supercritical range ($1<\g<\frac{4}{3}$) in \cite{ghjs}. A difficulty in the construction of self-similar solutions for $\g>1$ is that the algebra governing the local behavior of the solution near sonic points complicates significantly. The authors of \cite{ghjs} made the novel observation that for $\g>1$ the algebraic Larson-Penston and Hunter sonic-point conditions merge into a single solution branch, termed the Larson-Penston-Hunter-type Taylor expansion. In this work, we use this insight to extend the result of \cite{sandine} to the energy-supercritical range $1<\g<\frac{6}{5}$. In particular, we establish the following theorem (where $\ev{y}=\sqrt{1+y^2}$ denotes the Japanese bracket).
\begin{thm}
  \label{thm:result}
  For some integer $N\gg 1$ there exists real analytic solutions $\{(\rt_i,\ut_i)\}_{i=N}^{\infty}$ to \eqref{selfsimsys} on $[0,\infty)$ subject to the boundary conditions \eqref{eq:selfsim_boundary_conds} and the requirement $\rt_i>0$. In addition, the graph of $\rt_i(y)$ interesects that of the far-field solution $\rt_f(y)$, \eqref{eq:farfield_def}, exactly $i+1$ times. For these solutions, the sonic point condition, $(\ut+(2-\g)y)^2-\g\rt^{\g-1}=0$, is satisfied at exactly one point $y=y_*\in (0,2y_f)$, where $y_f$ is the sonic point for the far-field solution, \eqref{eq:yf_def}. At $y_*$, the solution $\rt,\ut$ admits the Larson-Penston-Hunter-type Taylor expansion (as defined in \cite{ghjs}). Finally, these solutions satisfy the global pointwise bounds $\rt_i\lesssim_i \ev{y}^{-\frac{2}{2-\g}}, \ut_i\lesssim \ev{y}^{-\frac{\g-1}{2-\g}}$.
\end{thm}
This theorem is proved in Subsection \ref{sec:proof}. In contrast to the isothermal case, as far as the author is aware, there are no numerical constructions of these Hunter-like solutions in the existing literature for $\g>1$. It is not hard to generate the first few Hunter-type solutions numerically by straightforwardly generalizing the approach of \cite{hunter} to $1<\g<\frac{6}{5}$.

We also mention that in contrast to \cite{ghjs}, we are not able to rigorously construct solutions for $\frac{6}{5}\leq \g<\frac{4}{3}$. It is unclear whether Hunter-type solutions would exist in this range or not. This point is discussed further in Subsection \ref{sec:nlwcomparison} via analogy with a scalar wave equation.

We also comment that the existence proof in \cite{ghjs} used rigorous interval arithmetic to control the signs of certain quantities. One feature of the approach we use is that due to its more perturbative nature, we do not need to use computer-assisted techniques.

We mention two other classes of self-similar solutions for the Euler-Poisson system in the literature. Firstly, if one considers dynamic entropy, then a new regime of solutions has been found in \cite{ahs}. Secondly, for the mass-critical polytropic index $\g=\frac{4}{3}$, there is a family of self-similar solutions to \eqref{eq:epsystem} found by Goldreich and Weber (\cite{goldreichweber}). These solutions different in character from the type of solutions mentioned so far in that they are semi-explicit and compactly supported. We refer interested readers to \cite{hadzicsurvey} for an introduction to the mathematical literature regarding these solutions and their stability.

We also comment that one source of motivation for studying these types of blow-ups for the Euler-Poisson system is the cosmic censorship conjectures of general relativity originally formulated by Penrose (\cite{penrose}, \cite{penroseb}). Indeed, for the Einstein-Euler system (describing a relativistic self-gravitating gas) a relativistic Larson-Penston solution was numerically constructed by Ori-Piran for an isothermal equation of state and sufficiently small values of the (constant) speed of sound (\cite{op1}, \cite{op2}, \cite{op3}). These authors argued that the solution could be suitably truncated to exhibit a naked singularity. A rigorous construction of the Ori-Piran solution and a dynamically forming naked singularity was established in \cite{ghj2}, which also contains further references on related literature. It is currently unclear if our methods would extend to the Einstein-Euler system to prove existence of relativistic Hunter solutions (or even, as far as the author is aware, if such solutions exist in the first place). For additional discussions on self-similarity, gravitational collapse and related models, we refer readers to the (living) review \cite{gundlach}.

Another related problem is constructing imploding solutions to the compressible Euler equations in the absence of a gravitational force. In this case, the system admits an additional scaling symmetry and the self-similar equations can be expressed as a system of \emph{autonomous} ODEs. Existence of (not necessarily smooth) solutions was first established by Guderley and Sedov (\cite{guderley}, \cite{sedov}). Smooth self-similar solutions were first constructed by Merle-Rapha\"el-Rodnianski-Szeftel (\cite{mrrs1}). The (finite-codimension) stability of these smooth profiles allowed the authors to construct blow-ups for the compressible Navier-Stokes and defocusing Schr\"odinger equations (\cite{mrrs2}, \cite{mrrs3}). Their results on singularity formation have been generalized to additional regimes. These include, in particular, implosions with mass density bounded from below (\cite{buckmastercaolaboragomezserrano}), implosions outside of radial symmetry (\cite{caolaboragomezserranoshistaffilani}, \cite{caolaboragomezserranoshistaffilani2}) and also implosions where the vorticity simultaneously blows up (\cite{ccsv}, \cite{chen}). Moreover, the instability of these blow-ups have been studied numerically in \cite{biasi}. We emphasize that in the non-gravitational problem, the additional scaling symmetry leads to a choice in scaling parameter (labeled $r$ in the above-mentioned works). There is no analog of this choice for the isentropic Euler-Poisson system, and so one must use non-autonomous methods.

We conclude this subsection by mentioning that, even in the isothermal case, the stability of the Larson-Penston and Hunter solutions has not been rigorously established. One would like to show that, in analogy with the numerics, the Larson-Penston solution is stable and that the $i$th Hunter solution is codimension-$i$ unstable. One aspect of this problem that is difficult is that the linearized problem is non-self-adjoint and non-explicit. A speculative comparison of this problem with the related problem for a scalar wave equation is discussed in Subsection \ref{sec:nlwcomparison}.
\subsection{Sonic point boundary conditions and local solvability}
\label{sec:sonicpoints}
We see that the ODE system \eqref{polytropicselfsimeq} is singular when
\begin{equation*}
  \left(\ut+(2-\gamma)y\right)^2-\gamma \rt^{\gamma-1}=0.
\end{equation*}
Following the literature, we refer to a point $y$ where this occurs as a sonic point. We comment that in terms of the physical coordinates $(t,x)$, a sonic point corresponds to an integral curve of the scaling vector-field $t\p_t+(2-\g)\sum_{i=1}^3x^i\p_{x^i}$ being null with respect to the acoustical Lorentzian  metric, $-c_{s}^2dt^2+\sum_{i=1}^3(dx^i-u^i dt)^2$, where $c_{s}=\g^{\frac{1}{2}}\rho^{\frac{\g-1}{2}}$ is the speed of sound. When we transition to the $y$ coordinate, we quotient along this curve introducing a singularity into the equations.

We now discuss the local theory of smooth solutions to \eqref{polytropicselfsimeq} near a singular point. In particular, we shall provide a derivation of the Larson-Penston-Hunter-type expansion. In contrast to previous, more direct, treatments, such as \cite{ghjs} and \cite{sandine}, we will treat \eqref{polytropicselfsimeq} as a special case of a more general class of singular quasilinear ODEs. This treatment is the most significant development from \cite{sandine}, and we emphasize it as we believe the methods could be useful for studying self-similar solutions to other quasilinear systems.

The class of ODEs we consider will be those that can be put into the following normal form. Note that we abuse notation and the $(u_1,u_2)$ below do not refer to velocity.
\begin{defin}
  \label{def:normal_form}
  We say that a singular quasilinear ODE system for $(u_1(t),u_2(t))$ is in normal form if it is of the form
  \begin{equation*}
    \begin{cases}
      A(t,u_1,u_2)\mqty{u_1'\\u_2'}+B(t,u_1,u_2)=0\\
      u_1(0)=u_2(0)=0
    \end{cases}
  \end{equation*}
  for smooth functions $A(z_0,z_1,z_2)$, $B(z_0,z_1,z_2)$ valued in $\R^{2\times 2}$, $\R^2$ respectively with
  \begin{equation*}
    A(0)=\mqty{1 & 0\\ 0 & 0},\qquad B(0)=\mqty{0\\ 0}.
  \end{equation*}
  For such a system, we define four characteristic parameters to be
  \begin{equation*}
    a:=\pdv[A_{22}]{z_0}(0),\qquad b:=\pdv[A_{22}]{z_2}(0),\qquad c:=\pdv[B_2]{z_0}(0),\qquad d=\pdv[B_2]{z_2}(0).
  \end{equation*}
\end{defin}
The author was unable to find a reference for the theory of smooth solutions to ODEs of this form on a neighborhood of $z=0$ (in particular in the quasilinear case $b\neq 0$), so the necessary lemmas are proved in the Appendix using elementary methods. 

To motivate these lemmas, we now suppose that $(u_1(t),u_2(t))$ is a $C^1$ solution to an equation of the form descirbed by in Definition \ref{def:normal_form}. From evaluating the first line of the matrix equation at $t=0$, we see that $u_1'(0)=0$. Now suppose that $u_2'(0)=U$. By expanding the second line to $o(t)$, we see that $U$ must satisfy the quadratic equation
\begin{equation}
  \label{eq:characteristic_quadratic}
  (a+bU)U+c+dU=0.
\end{equation}
Thus, we see that, depending on the parameters $a,b,c,d$, there will generically be at most two possible choices for $U$. Indeed, for our problem one of these choices will correspond to the Larson-Penston-Hunter Taylor expansion.

Going further, supposing now that $(u_1,u_2)$ is $C^{(k)}$ for $k\geq 2$ and $u_2'(0)=U$, if we expand the first equation to order $o(t^{k-1})$ and the second equation to $o(t^k)$, we obtain an equation schematically of the form (suppressing the dependence on the coefficients $A,B$)
\begin{equation*}
  \begin{cases}
    u_{1}^{(k)}(0)=F_{1,k}(u_1^{\leq(k-1)}(0),u_2^{\leq(k-1)}(0))\\
    \left(a+dU+(a+bU)k\right)u_2^{(k)}(0)=F_{2,k}(u_1^{\leq(k)}(0),u_2^{\leq(k-1)}(0))
  \end{cases}.
\end{equation*}
In particular, we can solve the first equation for $u_1^{(k)}(0)$ and substitute it into the second equation. The second equation is a linear equation for $u_{2}^{(k)}(0)$ and it is solvable if
\begin{equation}
  \label{eq:characteristic_nondegeneracies}
  a+dU+(a+bU)k\neq 0.
\end{equation}
Thus, assuming this non-degeneracy condition on the characteristic parameters, the higher derivatives are uniquely determined. Heuristically one should be able to solve the ODE exactly if one has a solution to the quadratic \eqref{eq:characteristic_quadratic} and if for $k\geq 2$ one has the conditions \eqref{eq:characteristic_nondegeneracies}. Justifying these heuristics rigorously is the content of the following lemmas. The proofs of these lemmas (which are elementary in terms of tools) are provided in the Appendix. As a comment, the statements of these lemmas do not give quantative bounds on the intervals of existence, although in principle a closer examination of the proof would yield such bounds, which would depend on the coefficients. The first lemma deals with solutions of regularity $C^1$.
\begin{lem}
  \label{lem:local_C1_wellposedness}
  Consider a singular ODE system for $u(t)=(u_1(t),u_2(t))$ as in Definition \ref{def:normal_form} with characteristic indices $a,b,c,d$. Suppose that the quadratic equation \eqref{eq:characteristic_quadratic} has a root $U$. Suppose that $a+bU\neq 0$ and define $\kappa=\frac{d+bU}{a+bU}$. Suppose that $\kappa>-1$. Then, for $T$ sufficiently small the ODE system subject to the condition $u_2'(0)=U$ admits a unique $C^1$ solution $u(t)$ defined on $(-T,T)$. 
\end{lem}
For the proof, see Subsection \ref{sec:local_C1_lemma}.
The second lemma deals with analytic regularity.
\begin{lem}
  \label{lem:main_analyticity_lemma}
  Consider a singular ODE system for $u(t)=(u_1(t),u_2(t))$ as in Definition \ref{def:normal_form} with characteristic indices $a,b,c,d$. Suppose that the coefficients $A(z),B(z)$ are real analytic. Suppose that the quadratic equation \eqref{eq:characteristic_quadratic}  has a root $U$. Suppose that $a+bU\neq 0$ and define $\kappa=\frac{d+bU}{a+bU}$. Suppose that $\kappa\notin \Z_{\leq -2}$. Then, for $T$ sufficiently small the ODE system subject to the condition $u_2'(0)=U$ admits a unique real-analytic solution $u(t)$ defined on $(-T,T)$. 
\end{lem}
For the proof, see Subsection \ref{sec:appendix_analyticity}.

A lengthy but straightforward computation shows that the system \eqref{polytropicselfsimeq} can be placed into the desired form near a sonic point. We summarize the desired information in the following lemma. 
\begin{lem}
  \label{lem:normalform}
  Consider a solution $(\rt,\ut)$ to \eqref{polytropicselfsimeq} with a sonic point at $y_*$. Define
  \begin{equation*}
    \rho_0:=\rt(y_*),\qquad u_0=\ut(y_*),\qquad \o_0:=\frac{u_0}{y_*}+2-\g.
  \end{equation*}
  Suppose that the following two equations in $(\rho_0,\o_0,y_*)$ are satisfied;
  \begin{align}
    y_*\o_0&=\g^{\frac12}\rho_0^{\frac{\g-1}{2}}    \label{intermedstep0},\\
    \frac{4\pi \rho_0\o_0}{4-3\g}&=2\o_0^2+(\g-1)\o_0+(2-\g)(\g-1). \label{eq:rho0sonic}
  \end{align}
  Then, there exists an invertible change of variables from $(\rt,\ut)$ to $(\rc,\uc)$ so that if $(\rt,\ut)$ is a $C^1$ solution of \eqref{polytropicselfsimeq} then $(\rc,\uc)$ is a $C^1$ solution to an equation in the normal form of Definition \ref{def:normal_form} with characteristic parameters
  \begin{align*}
    a&=\frac{2(2-\g)+(\g-3)(\o_0+(\g-1))}{4y_*\o_0},\\
    b&=-\frac{(\g+1)}{4\rho_0},\\
    c&=\frac{\rho_0}{4(y_*\o_0)^2}\left(-(\g+3)\o_0^2+(-2\g^2+\g+3)\o_0+(2-\g)(\g-1)^2\right)\\
     &\qquad-\frac{\rho_0^2}{4(y_*\o_0)^2}\left(\frac{4\pi}{(4-3\g)}\left(4-3\g-2\o_0\right)\right),\\
    d&=\frac{1}{4y_*\o_0}\left(3(\g-1)+(\g-3)(\o_0+(\g-1))\right).
  \end{align*}
  The conditions $\rc'(0)=0$, $\uc'(0)=U$ are equivalent to the equations for $\rt'(y_*)$ and $\ut'(y_*)$ determined by letting
  \begin{equation*}
    R=\frac{y_* \rt'(y_*)}{\rho_0},\qquad W=\ut'(y_*)-y_*^{-1}u_0
  \end{equation*}
  and solving
  \begin{equation*}
    \begin{cases}
      U-\frac{\rho_0}{y_*\o_0}(\o_0 R+(\o_0+(\g-1)))=0\\
      \omg_0R+W+3\omg_0-(4-3\g)=0
    \end{cases}.
  \end{equation*}
  The Larson-Penston-Hunter branch corresponds to the solution of \eqref{eq:characteristic_quadratic} given by
  \begin{equation}
    U=\frac{-(a+d)+\sqrt{(a+d)^2-4bc}}{2b}.
    \label{eq:lph_branch}
  \end{equation}
  Moreover, the far-field solution obeys the Larson-Penston-Hunter sonic point conditions at
  \begin{equation}
    \label{eq:yf_def}
    y_f:=\frac{\g^{\frac12}}{2-\g}(\frac{4-3\g}{2\pi})^{\frac{\g-1}{2}}=\frac{k^{\frac{(\g-1)(2-\g)}{2}}\g^{\frac{2-\g}{2}}}{(2-\g)^{2-\g}}
    .
  \end{equation}
  Finally, for the far-field solution we have
  \begin{equation*}
    a+bU=\frac{1}{2(2-\g)y_f},\qquad \kappa=\frac{5(\g-1)}{2}.
  \end{equation*}

\end{lem}
The proof of this lemma is deferred to Appendix \ref{sec:normal_form_appendix}. As a consistency check, we mention that if one substitutes $U=\frac{\rho_0}{y_*}(\o_0R+(\o_0+(\g-1)))$ into the quadratic \eqref{eq:characteristic_quadratic} then one obtains the quadratic
\begin{multline}
  \label{eq:Rquadratic}
  -(1+\gamma)\omega_0^2R^2+(9-7\gamma-8\omega_0)\omega_0R\\
  -6\omega_0^2+(8-6\gamma)\omega_0+(-3\gamma^2+9\gamma-6)+\frac{(4-3\gamma)(2-\gamma)(1-\gamma)}{\omega_0}=0,
\end{multline}
which is consistent with \cite{ghjs} (their equation (2.69)).

We finally comment that the proof of local analyticity given by \cite{ghjs} for this system (their Theorem 2.15) is similar in terms of techniques to our proof of Lemma \ref{lem:main_analyticity_lemma} and could in principle be modified for use with Hunter-type solutions. We treat the problem more abstractly to clarify precisely the assumptions needed and to allow future works to ``black box'' this part of the argument.
\subsection{Analogous results and open problems for a focusing scalar wave equation}
\label{sec:nlwcomparison}

In the isothermal case, the author's previous construction of Hunter solutions was motivated by similarities between the Euler-Poisson system and a class of semilinear heat equations (see \cite{sandine} for more details and references). In this section, we will compare the theory of self-similar solutions of \eqref{eq:epsystem} with that of solutions to the following semilinear focusing wave equation (NLW) for $\phi\colon I_t\times \R_x^d\to \R$:
\begin{equation}
  \label{eq:nlw}
  \p_t^2\phi-\Lap \phi-\phi^p=0.
\end{equation}
One advantage of using this equation as a model, instead of a heat equation, is that it is hyperbolic, so the theory of non-linear stability of self-similar solutions to \eqref{eq:nlw} should be a good model for the Euler-Poisson system. As a first step towards deriving the self-similar ODE, we note that this equation is invariant under the scaling transformation ($\la>0$)
\begin{equation*}
  \phi(t,x)\mapsto \la^{-\frac{2}{p-1}}\phi(t/\la,x/\la).
\end{equation*}
If a solution is invariant with respect to this symmetry, its profile will satisfy the ODE
\begin{equation}
  \label{eq:nlw_selfsim}
  (1-y^2)\p_y^2\phi+(\frac{d-1}{y}-\frac{2(p+1)}{p-1}y)\p_y \phi-\frac{2(p+1)}{(p-1)^2}\phi+\phi^p.
\end{equation}
In particular, we see that this equation admits (in our language) a sonic point, but unlike the Euler-Poisson problem, the location of the sonic point is independent of the solution. We also comment that, in the case $d=3$, the energy-supercritical range of $p$ for NLW is $p>5$ (this was $\g<\frac{6}{5}$ for Euler-Poisson) and the conformal-supercritical range of $p$ for NLW is $p>3$ (this corresponds to $\g<\frac{4}{3}$ for Euler-Poisson).

For \eqref{eq:nlw_selfsim} there also exist versions of the Friedman and far-field solutions. The Friedman (or ODE blow-up solution) is given explicitly by
\begin{align*}
  \phi_F=(\frac{2(p+1)}{(p-1)^2})^{\frac{1}{p-1}}.
\end{align*}
The far-field solution (in the case $p>\frac{d}{d-2}$) is given by
\begin{align*}
  \phi_{F}=(\frac{2((p-1)(d-2)-2)}{(p-1)^2})^{\frac{1}{p-1}}\cdot y^{-\frac{2}{p-1}}.
\end{align*}

Uniqueness and stability of the Friedman solution for \eqref{eq:nlw_selfsim} in 1 spatial dimension was shown by Merle  and Zaag (\cite{mz03,mz05,mz07,mz08,mz12a,mz12b}). In higher dimensions the stability problem is more delicate and has been the subject of much recent work (\cite{cd,d10,d17,dr,ds12,ds14,ds16,ds17,liu,mz15,mz16,ostermann,wallauch}).  We comment however, that for the purposes of modeling gravitational collapse from isolated bodies, we expect the relevant profiles to decay as $\abs{x}\to \infty$. Consequentially, we shall focus on the non-Friedman self-similar solutions to \eqref{eq:nlw_selfsim}. 

We now briefly discuss some literature on the existence of such solutions. To rule out the Friedman solution, we will restrict attention to self-similar solutions whose leading order decay rate as $y\to \infty$ is at least that of the far-field solution, namely $y^{-\frac{2}{p-1}}$. This can be viewed as a localization requirement. A numerical study of several exponents is carried out in \cite{bct} in $d=3$. In particular, for the energy-supercritical exponent $p=7$ a family of self-similar solutions was found (with decay rate $y^{-\frac{2}{p-1}}$), including a solution with a single unstable mode which was observed to play a role in describing the threshold between blow-up and dispersion. Self-similar solutions in $d=3$ for $p$ which are odd and energy-supercritical were rigorously constructed in \cite{bmw}. We emphasize the close resemblance of Figure $1$ of \cite{bmw}, depicting their shooting construction of self-similar solutions to NLW, and Figure $2$ of \cite{hunter}, depicting the shooting construction of  Hunter solutions. In higher dimensions, energy-supercritical self-similar blow-up profiles were analyzed numerically in \cite{kycia} and constructed rigorously in \cite{dd}. One surprising feature of NLW is that in higher spatial dimensions, $d\geq 5$, the (energy-supercritical) cubic NLW admits an explicit self-similar solution found and studied by Glogi\'c-Sch\"orkhuber (\cite{gs}). We also remark that it decays faster ($y^{-\frac{2}{p-1}-1}$) than the far-field solution ($y^{-\frac{2}{p-1}}$). This solution's role as a threshold was studied numerically in \cite{gms}. For the quadratic NLW, in the energy-supercritical dimensions $d\geq 7$, there is again an explicit self-similar solution found and analyzed in \cite{cgs}, which differs from the previously discussed solutions in that it is non-positive. This work was extended beyond the past-light cone of the singularity in \cite{cdgms}.

One key difference between NLW and Euler-Poisson is that for NLW, each self-similar solution known to exist has at least one linear instability direction, whereas for the Euler-Poisson system, the Larson-Penston-type solutions are expected to be linearly stable.

We also comment that as far as the author is aware, for $d=3$, there is currently no theoretical or numerical evidence of existence of self-similar blow-up profiles to NLW for the conformal-supercritical, energy-subcritical range $3<p<5$ with the far-field decay rate $\phi=O(y^{-\frac{2}{p-1}})$. This is a bit surprising as for the more complicated Euler-Poisson system, this corresponds to the range of exponents $\frac{6}{5}<\g<\frac{4}{3}$ for which a solution is known to exist (\cite{ghjs} following numerics of \cite{yahil}).

The work on NLW which seems most potentially applicable for understanding the stability of Hunter-type solutions to the Euler-Poisson system is \cite{kim}. In this work (for $d=3$), the author constructed self-similar solutions in the full energy-supercritical range, and showed finite-codimensional stability of the solutions by adapting the framework of \cite{mrrs3}. As our construction of solutions is very similar to the approach taken in this work, one could hope that the author's framework could be used to show the finite-codimension stability of the Hunter-type solutions we contruct. We comment that unlike the preceding work on the non-linear heat equation (\cite{crs}), or the scenerio of the explicit solutions (\cite{gs}, \cite{cgs}), the dimension of instability is not revealed in \cite{kim}.

Indeed, if one is either interested in whether Hunter-type solutions exist for $\frac{6}{5}<\g<\frac{4}{3}$ or interested in if one can show non-linear stability with the expected codimension, the related, probably easier, open problems for the NLW equation \eqref{eq:nlw} would be the following.
\begin{itemize}
\item For $d=3$ and $3<p<5$ do there exist any smooth self-similar blow-up profiles of \eqref{eq:nlw} such that $\phi=O(y^{-\frac{2}{p-1}})$ as $y\to \infty$? If yes, do there exist infinitely many solutions?
\item Can one rigorously determine the exact codimension of instability for the solutions constructed by Kim in \cite{kim}?
\end{itemize}
\subsection{Comparison with the isothermal problem, numerology and proof outline}
At a broad level, the proof of the theorem is an adaptation of the isothermal proof (\cite{sandine}), which was in turn based on techniques from simpler models (see, in particular, \cite{crs}). In fact, the methods we use are largely insensitive to the choice of $\g$ in $[1,\frac{6}{5})$ and the main challenges of the polytropic case are that the algebraic expressions get more complicated. We comment that in \cite{ghjs}, the exponent $\g=\frac{10}{9}$ played a special role. One point in their argument where this exponent appears is when the authors parameterize the curve of triples $(\rho_0,\o_0,y_*)$ which satisfy \eqref{intermedstep0} and \eqref{eq:rho0sonic} by $y_*$. We parameterize this curve by $\omg_0$ and consequentially do not see $\g=\frac{10}{9}$ at this step.

This work contains several technical improvents over \cite{sandine}. The most significant of these is in the treatment of local analyticity, as outlined in Subsection \ref{sec:sonicpoints} and carried out in Subsection \ref{sec:proof}. Moreover, the derivation of difference equations in Subsections \ref{sec:exterior} and \ref{section:int} is streamlined. In particular, in Section \ref{sec:exterior}, our linear lemma now accomodates compatible non-zero Taylor coefficients. Moreover, in Section \ref{section:int} we do not rescale the coordinate $y$. In addition, in this work (through the use of Lemma \ref{lem:implicit_function_theorem}) we perform the matching using only a Lipschitz dependence of the interior and exterior solutions on $\la$, $\ep$, which reduces the number of estimates required. As these improvements are all technical, the paper is organized identically to \cite{sandine}, which the reader is recommended to consult as a guide since the algebraic expressions are much cleaner and there is more commentary in the proof.

We will now present some heuristics which allow relatively quick derivations of the numerology. In particular, this section should illuminate why our methods cannot work for $\g\geq \frac{6}{5}$.

The basic structure of the argument is to consider a fixed, small $y_0$, construct a 1-parameter family of ``exterior solutions'' on $[y_0,\infty)$, construct a 1-parameter family of ``interior'' solutions on $[0,y_0]$ and then match them to obtain an infinite, discrete family. These solutions are a priori $C^1$ and higher regularity is obtained through application of the theory outlined in Subsection \ref{sec:sonicpoints}.

We begin with the (non-self-similar) Euler-Poisson system, \eqref{eq:epsystem}. We shall define the enthalpy to be
\[
  w=\frac{\gamma}{\gamma-1}\rho^{\gamma-1}.
\]
Setting $u=0$, taking the divergence of the second equation and changing variables from $\rho$ to $w$ gives the Lane-Emden equation
\begin{equation}
  \Delta w+4\pi (\frac{\gamma-1}{\gamma})^{\frac{1}{\gamma-1}}w^{\frac{1}{\gamma-1}}=0.
  \label{laneemdeneq}
\end{equation}
We now summarize properties of solutions to this equation that can be found, for instance, in \cite{chandrasekhar}.
We shall define $Q$ to be the unique solution such that
\[
  Q(0)=\frac{\gamma}{\gamma-1}.
\]
The family of regular solutions to \eqref{laneemdeneq} is parameterized by
\[
  \lambda\mapsto Q_{\lambda}=\frac{1}{\lambda^{\frac{2(\gamma-1)}{2-\gamma}}}Q(\cdot/\lambda).
\]
The far-field solution in this case is
\[
  Q_f=\frac{\gamma}{\gamma-1}k^{\gamma-1}r^{-\frac{2(\gamma-1)}{2-\gamma}}
\]
which, if we change back to $\rho$, is
\[
  \rho_f=(\frac{\gamma-1}{\gamma})^{\frac{1}{\gamma-1}}Q_f^{\frac{1}{\gamma-1}}=kr^{-\frac{2}{2-\gamma}}.
\]
We comment that for $\g<\frac{6}{5}$, the Lane-Emden solutions are positive for all $r$ and their spatial decay matches the far-field solution ($r^{-\frac{2}{2-\g}}$). This is not true for $\g\geq \frac{6}{5}$. In contrast to \cite{ghjs}, the Lane-Emden equation and its linearization plays a key role in our analysis, so it would be unable to alter our proof to deal with $\g\geq \frac{6}{5}$. 

We now return to the self-similar system \eqref{polytropicselfsimeq}. We recall that the far-field solution, \eqref{eq:farfield_def}, has a sonic point at $y=y_f$ (defined in \eqref{eq:yf_def}). If we linearize the self-similar system about the far-field solution $(\rt_f,\ut_f)$, we obtain the following linear ODE which is singular at $y\in \{0,y_f\}$,
\begin{multline*}
  \p_y \begin{pmatrix}\rt\\ \ut\end{pmatrix}+\frac{1}{(2-\gamma)^2y^2-k^{\gamma-1}\gamma y^{2-\frac{2}{2-\gamma}}}\cdot\\
  \begin{pmatrix}2(2-\gamma)y-k\frac{4\pi(2-\g)(3-\g)}{4-3\gamma}y^{1-\frac{2}{2-\gamma}} & -3(\gamma-1)k y^{-\frac{2}{2-\gamma}}-\frac{4\pi}{4-3\gamma}k^2y^{-\frac{4}{2-\gamma}}\\ (-2k^{\gamma-2}\gamma +\frac{4\pi(2-\g)^2(3-\g)}{4-3\gamma})y^2 & (2\frac{\gamma-1}{2-\gamma}k^{\gamma-1}\gamma +\frac{4\pi(2-\g)}{4-3\gamma}k)y^{1-\frac{2}{2-\gamma}}+(2-\gamma)(\gamma-1) y\end{pmatrix}\begin{pmatrix}\rt\\ \ut\end{pmatrix}=0.
\end{multline*}
If we then define $\tilde{p}=y^{\frac{2}{2-\gamma}}\tilde{\rho}$, $\tilde{v}=y^{-1}\tilde{u}$, then we get the system
\begin{multline*}
  \p_y \begin{pmatrix}\tilde{p}\\ \tilde{v}\end{pmatrix}+\frac{1}{y}\mqty{-\frac{2}{2-\gamma} & 0\\ 0 & 1}\begin{pmatrix}\tilde{p}\\ \tilde{v}\end{pmatrix}+\frac{1}{(2-\gamma)^2y^2-k^{\gamma-1}\gamma y^{2-\frac{2}{2-\gamma}}}\cdot\\
  \mqty{2(2-\gamma)y-k\frac{4\pi(2-\g)(3-\g)}{4-3\gamma}y^{1-\frac{2}{2-\gamma}} & -3(\gamma-1)k y-\frac{4\pi}{4-3\gamma}k^2y^{1-\frac{2}{2-\gamma}}\\ (-2k^{\gamma-2}\gamma +\frac{4\pi(2-\gamma)^2(3-\gamma)}{4-3\gamma})y^{1-\frac{2}{2-\gamma}} & (2\frac{\gamma-1}{2-\gamma}k^{\gamma-1}\gamma +\frac{4\pi(2-\gamma)k}{4-3\gamma})y^{1-\frac{2}{2-\gamma}}+(2-\gamma)(\gamma-1) y}\mqty{\tilde{p}\\ \tilde{v}}=0.
\end{multline*}
We can simplify this to
\begin{multline*}
  \p_y \mqty{\tilde{p}\\ \tilde{v}}+\frac{1}{(2-\gamma)^2y^2-k^{\gamma-1}\gamma y^{2-\frac{2}{2-\gamma}}}\cdot\\\
  \mqty{-\frac{4\pi k(2-\g)^2}{4-3\g}y^{1-\frac{2}{2-\g}}& -3(\gamma-1)k y-\frac{4\pi}{4-3\gamma}k^2y^{1-\frac{2}{2-\gamma}}\\ \frac{4\pi(2-\g)^3}{4-3\g}y^{1-\frac{2}{2-\gamma}} & (2-\g)y+\frac{2\pi(2-\g)k(3\g-2)}{4-3\g}y^{1-\frac{2}{2-\g}}}\mqty{\rt\\ \vt}=0.
\end{multline*}
We can then simplify the residue at $y=0$ to be
\[
  \mqty{\frac{4\pi k^{2-\g}(2-\g)^2}{\g(4-3\g)} & \frac{4\pi k^2}{(4-3\gamma)k^{\gamma-1}\gamma}\\ 2k-\frac{4\pi(2-\gamma)^2(3-\gamma)}{(4-3\gamma)k^{\gamma-1}\gamma} & \frac{-2\pi (2-\g)k^{2-\g}(3\g-2)}{2(4-3\g)}}=\mqty{2 & \frac{2k}{(2-\gamma)^2}\\ -\frac{2(2-\gamma)}{k} & -\frac{3\gamma-2}{2-\gamma}}.
\]
We define 
\begin{equation}
  \label{eq:munudef}
  \mu:=\frac{6-5\gamma}{2(2-\gamma)},\qquad \nu:=\frac{\sqrt{-\gamma^2-20\gamma+28}}{2(2-\gamma)}.
\end{equation}
We may then diagonalize (the negative of) this residue to be
\begin{multline*}
  \mqty{-2 & -\frac{2k}{(2-\gamma)^2}\\ \frac{2(2-\gamma)}{k} & \frac{3\gamma-2}{2-\gamma}}=\mqty{1 & 1\\ \frac{(2-\gamma)^{\frac{3}{2}}}{k}e^{i\theta_0} & \frac{(2-\gamma)^{\frac{3}{2}}}{k}e^{-i\theta_0}}\\
  \cdot\mqty{-\mu+i\nu & 0\\ 0 &-\mu-i\nu }\mqty{1 & 1\\ \frac{(2-\gamma)^{\frac{3}{2}}}{k}e^{i\theta_0} & \frac{(2-\gamma)^{\frac{3}{2}}}{k}e^{-i\theta_0}}^{-1}.
\end{multline*}
where we define
\begin{equation}
  \label{eq:angledef}
  \theta_0=\arctan(\frac{\sqrt{-\gamma^2-20\gamma+28}}{2+\gamma})+\pi.
\end{equation}
We comment that this angle satisfies
\[
  \sin(\theta_0)=-\frac{\sqrt{-\gamma^2-20\gamma+28}}{4\sqrt{2-\gamma}}\neq 0.
\]
From this, we may read off the asymptotics of the homogenous solution to the linearized equations to be
\begin{align*}
  p_{hom}&=\frac{c_1}{y^{\mu}}\sin(\nu \log(y)+d_1)+O(y^{-\mu+\frac{2}{2-\g}}),\\
  \omega_{hom}&=\frac{c_1}{y^{\mu}}\frac{(2-\gamma)^{\frac{3}{2}}}{k}\sin(\nu \log(y)+d_1+\theta_0)+O(y^{-\mu+\frac{2}{2-\g}}).
\end{align*}
The exterior solutions we construct (for small $\ep$) will then be of the form
\begin{equation}
  \label{eq:exterior_heuristics}
  \begin{cases}
    \tilde{\rho}_{ext}=\frac{k}{y^{\frac{2}{2-\gamma}}}+\frac{\epsilon }{y^{\frac{2}{2-\gamma}}}p_{hom}+\cdots\\
    \tilde{u}_{ext}=\epsilon y \omega_{hom}+\cdots
  \end{cases}.
\end{equation}
We comment that in order for the $O(\ep)$ terms in the expansion for $\rt_{ext}$ to have faster decay than the principal term, we need $\mu>0$ which is equivalent to $\g<\frac{6}{5}$. We also comment that when $\g=1$, $(\mu,\nu)=\frac{1}{2},\frac{\sqrt{7}}{2}$, which recovers the numerology observed in the isothermal problem.

Turning now to the interior problem, we shall define a universal profile $u_*$ by solving the linear equation
\begin{equation}
  \label{eq:ustardef}
  \begin{cases}
    ((2-\gamma)y\p_y+2)(\frac{\gamma-1}{\gamma}Q)^{\frac{1}{\gamma-1}}+\div((\frac{\gamma-1}{\gamma}Q)^{\frac{1}{\gamma-1}} u_*)=0\\
    u_*(0)=0,\qquad u_*'(0)=-\frac{2}{3}
  \end{cases}.
\end{equation}
where $Q$ is the solution to the Lane-Emden Equation \eqref{laneemdeneq} with $Q(0)=\frac{\gamma}{\gamma-1}$. When we linearize \eqref{laneemdeneq} about the far-field solution, the linearized operator is (after flipping signs)
\[
  \Ht:=-(\Delta+4\pi(\frac{\gamma-1}{\gamma})^{\frac{1}{\gamma-1}}\frac{1}{\gamma-1}Q_f^{\frac{2-\gamma}{\gamma-1}}),
\]
which is simplified to
\begin{equation}
  \label{eq:Htdef}
  \tilde{H}:=-(\Delta+\frac{2(4-3\gamma)}{(2-\gamma)^2}\frac{1}{y^2}).
\end{equation}
In this case, the fundamental solutions of $\tilde{H}$ are given by $y^{-\frac{1}{2}\pm i\nu}$. This is used to justify the asymptotics
\[
  Q=\frac{\gamma}{\gamma-1}\frac{k^{\gamma-1}}{y^{\frac{2(\gamma-1)}{2-\gamma}}}+\frac{\gamma}{k^{2-\gamma}}\frac{c_2 \sin(\nu \log(y)+d_2)}{y^{\frac{2(\g-1)}{2-\g}+\mu}}+O(\frac{1}{y^{\frac{2(\g-1)}{2-\g}+2\mu}}).
\]
We comment again, that here the condition $\g<\frac{6}{5}$ plays a role to ensure that the first term is in fact of leading order as $y\to \infty$. 

This is in terms of enthalpy variables, converting to density yields
\[
  (\frac{\gamma-1}{\gamma})^{\frac{1}{\gamma-1}}Q^{\frac{1}{\gamma-1}}=\frac{k}{y^{\frac{2}{2-\gamma}}}+\frac{c_2 \sin(\nu \log(y)+d_2)}{y^{\frac{2}{2-\g}+\mu}}+O(\frac{1}{y^{\frac{2}{2-\gamma}+2\mu}})
\]
To solve for $u_*$ it is useful to use the Lane-Emden equation to rewrite $(\frac{\gamma-1}{\gamma}Q)^{\frac{1}{\gamma-1}}=-\frac{1}{4\pi}\Delta Q$ and then to use the identity
\[
  ((2-\gamma)y\p_y+2)\circ \Delta=\Delta\circ \left((2-\gamma)y\p_y+2(\gamma-1)\right)
\]
This allows us to solve for $u_*$ directly to be
\[
  u_*=(\frac{\gamma-1}{\gamma}Q)^{-\frac{1}{\gamma-1}}\frac{1}{4\pi}\p_y((2-\gamma)y\p_y Q+2(\gamma-1)Q).
\]
From this, we see that as $y\to \infty$, we have that
\begin{equation}
  u_*(y)=k^{-1}(2-\gamma)^{\frac{3}{2}}c_2\sin(\nu \log(y)+d_2+\theta_0)y^{1-\mu}+O(y^{1-2\mu}).
  \label{ustarasymptot}
\end{equation}
For completeness, we also mention that from the Lane-Emden equation and $Q(0)=\frac{\gamma}{\gamma-1}$, we see that as $y\to 0$, $Q(y)=\frac{\gamma}{\gamma-1}-\frac{4\pi}{6}y^2+o(y^2)$ which gives $u_*'(0)=-\frac{2}{3}$.

For $\la$ sufficiently small, we will construct an interior solution of the form
\begin{equation}
  \label{eq:interior_heuristics}
  \begin{cases}
    \tilde{\rho}_{int}=\frac{1}{\lambda^{\frac{2}{2-\gamma}}}(\frac{\gamma-1}{\gamma}Q)^{\frac{1}{\gamma-1}}(\cdot/\lambda)+\cdots\\
    \tilde{u}_{int}=\lambda u_*(\cdot/\lambda)+\cdots
  \end{cases}.
\end{equation}

Turning our attention to the matching, when we evaluate the exterior solutions, \eqref{eq:exterior_heuristics}, at $y_0$ we will obtain
\[
  \begin{cases}
    \rt_{ext}(y_0)-\frac{k}{y_0^{\frac{2}{2-\gamma}}}&=\epsilon \frac{1}{y_0^{\frac{2}{2-\g}+\mu}}c_1\sin(\nu \log(y_0)+d_1)+O(\ep y_0^{-\mu}) \\
    \frac{\ut_{ext}(y_0)}{y_0^{\frac{4-\gamma}{2-\gamma}}}&=\frac{\epsilon}{y_0^{\frac{2}{2-\g}+\mu}}\frac{(2-\gamma)^{\frac{3}{2}}}{k}c_1 \sin(\nu \log(y_0)+d_1+\theta_0)+O(\ep y_0^{-\mu})
  \end{cases}.
\]
Similarly, evaluating the interior solutions, \eqref{eq:interior_heuristics}, at $y_0$ will give
\[
  \begin{cases}
    \tilde{\rho}_{int}(y_0)-\frac{k}{y_0^{\frac{2}{2-\gamma}}}&=\frac{\lambda^{\mu}}{y_0^{\frac{2}{2-\g}+\mu}}c_2 \sin(\nu\log(y_0/\lambda)+d_2)+O(\la^{\mu} y_0^{-\mu})\\
    \frac{\tilde{u}_{int}(y_0)}{y_0^{\frac{4-\gamma}{2-\gamma}}}&=\frac{\lambda^{\mu}}{y_0^{\frac{2}{2-\g}+\mu}}\frac{c_2(2-\gamma)^{\frac{3}{2}}}{k}\sin(\nu \log(y_0/\la)+d_2+\theta_0)+O(\la^{\mu} y_0^{-\mu})
  \end{cases}.
\]
For simplicity, by the flexibility in choosing $y_0$, we may impose
\begin{equation}
  \label{eq:wigglecond}
  \nu\log(y_0)+d_1=\frac{\pi}{2}\mod 2\pi.
\end{equation}
Matching the interior and exterior solutions will give the equations
\[
  \begin{cases}
    \epsilon c_1+O(\ep y_0^{\frac{2}{2-\g}})&=c_2\lambda^{\mu}\cos(\nu\log(\lambda)+d_1-d_2)+O(\la^{\mu}y_0^{\frac{2}{2-\g}})\\
    \epsilon c_1\cos(\theta_0)+O(\ep y_0^{\frac{2}{2-\g}})&=c_2\lambda^{\mu}\cos(\nu\log(\lambda)+d_1-d_2-\theta_0)+O(\la^{\mu} y_0^{\frac{2}{2-\g}})
  \end{cases}.
\]
Solving the first equation for $\ep$ as a function of $\la$, and then substituting it into the second equation will yield
\[
  0=\sin(\nu \log(\lambda)+d_1-d_2)+O(y_0^{\frac{2}{2-\g}}).
\]
By choosing $y_0$ sufficiently small, we see that this equation will have an infinite discrete family of solutions which will correspond to the Hunter-type solutions. Finally, non-trivial task of verifying analyticity of the solutions is accomplished through application of the framework described in Subsection \ref{sec:sonicpoints}.

\subsection{Organization and notations}
We now summarize the structure of the paper (which is parallel with \cite{sandine}). Section \ref{sec:exterior} details the construction of the exterior family of solutions. Section \ref{section:int} contains the construction of the interior family of solutions. The matching is performed in Section \ref{sec:matching} culminating in the proof of the main theorem. Sections \ref{sec:ext_proofs}, \ref{sec:interior_proofs} contain proofs defered from Sections \ref{sec:exterior} and \ref{section:int} respectively. Finally, Appendix \ref{sec:appendix} contains the relatively self-contained proofs of lemmas used in Section \ref{sec:matching}.

For the reader's convenience we also provide a partial dictionary of notational conventions used in the paper. The meanings of some notations are depend on the section of the paper in which they occur.

Selected section-independent notations:
\begin{itemize}
\item $\g$: polytropic index for the equation of state.
\item $\rt,\ut$: self-similar mass density and velocity.
\item $y$: self-similar coordinate.
\item $\ev{y}$: $\sqrt{1+y^2}$.
\item $\rt_F,\ut_F$: Friedman solution.
\item $\rt_f,\ut_f$: far-field solution.
\item $k$: an explicit constant, see \eqref{eq:kdef}.
\item $y_*$: a sonic point.
\item $y_f$: the sonic point for the far-field solution, see \eqref{eq:yf_def}.
\item $Q$: Lane-Emden solution (at the level of enthalpy).
\item $Q_f$: far-field enthalpy.
\item $u_*$: a leading order velocity derived from the Lane-Emden solution, see \eqref{eq:ustardef}.
\item $\mu,\nu$: exponents arising from linearized analysis, see \eqref{eq:munudef}.
\item $\theta_0$: angle arising from linearized analysis, see \eqref{eq:angledef}.
\item $y_0$: self-similar coordinate at which the matching occurs.
\item $\ep$: smallness parameter for the exterior region.
\item $\la$: used for scaling, used as a smallness parameter for the interior region.
\end{itemize}

Selected notations used in Section \ref{sec:intro}:
\begin{itemize}
\item $(t,x)$: physical time and space coordinates.
\item $(\rho,u)$: the physical (i.e. non-self-similar) mass density and velocity.

\item $r=\abs{x}$: physical radial coordinate.
\item $(u_1,u_2)$: variables for the ODE normal form definition.
\item $z_0,z_1,z_2$: dummy variables for $t,u_1,u_2$ for the ODE normal form definition. 
\item $a,b,c,d$: characteristic parameters for the ODE normal form definition.
\item $U$: solution to a quadratic equation associated to the ODE normal form.
\item $\kappa$: parameter for the ODE lemmas.   
\item $(\rc,\uc)$: variables for putting the self-similar Euler-Poisson system in normal form.
\end{itemize}

Selected notations used in Sections \ref{sec:exterior}, \ref{sec:ext_proofs}:
\begin{itemize}
\item $(\pt,\ot)$: convenient variables for the exterior analysis.
\item $\alpha,\beta$: parameters for changing coordinates to $z$.
\item $z$: choice of coordinates which fixes the location of the sonic point.
\item $(\pul,\oul)$: a convenient choice of variables in $z$ coordinates.
\item $(\ph,\oh)$: variables for the difference equation.
\item $(\Et,\Ft)$, $(\Eul,\Ful)$ and $(\Eh,\Fh)$: coefficient functions for the equations involving $(\pt,\ot)$, $(\pul,\oul)$ and $(\ph,\oh)$ respectively.
\item $(a,b)$ dummy variables for $(\pt,\ot)$, $(\pul,\oul)$ or $(\ph,\oh)$.
\item $\pr_0,\omgr_0,\yr_*,\pr_1,\omgr_1$: smooth functions of $\ep$ that describe the sonic point conditions.
\item $\pr_{0,d}$, $\omgr_{0,d}$, $\pr_{1,d}$, $\omgr_{1,d}$, difference quotients associated to the above functions.
\item $\pr_{0,e}$, $\omgr_{0,e}$, $\pr_{1,e}$, $\omgr_{1,e}$, further ``error'' difference quotients associated to the above functions.
\item $\sr_1,\sr_2,\sr_3$: smooth functions of $\ep$ that describe compatibility conditions.
\item $E_{lin}$, $F_{lin}$: linearizations of $\Eh$, $\Fh$.
\item $L$ and $L_1,L_2$, the linearized operator and its components.
\item $\Nh$: the non-linearity in the difference equation.
\item $\s_1,\s_2,\s_3$: variables used for describing combatibility conditions.
\item $(p_{hom},\o_{hom})$: homogeneous solution for $L$.
\item $U_{0}$, $U_{\infty}$: fundamental matrices for $L$ on $(0,1)$ and $(1,\infty)$.
\item $R$: left-inverse to $L$ constructed in Lemma \ref{ext_fun_matrix_lemma}.
\item $X_{y_0},N_{y_0}$: Banach spaces used for the iteration.
\item $X_{y_0}[\ep],N_{y_0}[\ep]$, $X_{y_0}[\ep_1,\ep_2]$, $N_{y_0}[\ep_1,\ep_2]$: affine subspaces of $X_{y_0},N_{y_0}$.
\item $(p,\o)$: final difference variables.
\end{itemize}

Selected notations used in Sections \ref{section:int}, \ref{sec:interior_proofs}:
\begin{itemize}
\item $\ot$: self-similar enthalpy.
\item $H$: linearization of Lane-Emden equation about $Q$.
\item $\Ht$: linearization of Lane-Emden equation about $Q_f$.
\item $\wb,\ub$: a change of variables used for a heuristic.
\item $J,T,L,S$: linear operators arising from the asymptotic expansion.
\item $J_{\la},L_{\la},H_{\la},S_{\la},T_{\la},K_{\la}$: rescaled linear operators.
\item $F_{\la},G_{\la,1},G_{\la,2},G_{\la}$: rescaled non-linear operators.
\item $Z,Y,\calG,\calF$: Banach spaces used for the iteration.
\item $(w,u)$: difference variables. 
\end{itemize}

Selected notations used in Section \ref{sec:matching}:
\begin{itemize}
\item $\calF,\calG$: functions describing the matching.
\end{itemize}

Selected notations used in Appendix \ref{sec:appendix}.
\begin{itemize}
\item $\rs,\us$: preliminary variables for putting the equations in normal form.
\end{itemize}
\section{Solving in the exterior region}
\label{sec:exterior}
We begin by choosing convenient variables for the exterior problem. In this case, we define the variables $\pt,\ot$ implicitly via
\[
  \tilde{\rho}=\frac{\pt}{y^{\frac{2}{2-\gamma}}},\qquad \ut=y(\ot-(2-\g)).
\]
In these variables, the self-similar equations \eqref{polytropicselfsimeq} take the form
\begin{equation}
  \begin{cases}
    (\pt \ot y)'+2\frac{1-\g}{2-\g}\pt \ot-(4-3\g)\pt&=0\\
    \ot(y(\ot-(2-\gamma)))'+\frac{\g \pt'}{y^{\frac{\g}{2-\g}\pt^{2-\g}}}+(\g-1)(\ot-(2-\g))&\\
    \qquad+\frac{1}{y^{\frac{2}{2-\g}}}(\frac{4\pi \pt \ot}{4-3\g}-\frac{2\g}{2-\g}\pt^{\g-1})&=0
  \end{cases}.
  \label{pomegaeq}
\end{equation}
In these variables, the far-field solution is
\[
  (\pt,\ot)=(k,2-\g).
\]
We will be considering a perturbation of the far-field solution, and so we note that we may write the system more abstractly as
\begin{equation}
  \Et(y,\pt,\ot)\mqty{\pt'\\ \ot'}+\Ft(y,\pt,\ot)=0.
  \label{pomegaeqabstract}
\end{equation}
where (using dummy variables $\a,\b$ for $\pt,\ot$)
\[
  \Et(y,\a,\b)=\mqty{\b y & \a y\\ \frac{\g}{y^{\frac{\g}{2-\g}}} \frac{1}{\a^{2-\g}}& \b y},
\]
\[
  \Ft(y,\a,\b)=\mqty{\frac{4-3\g}{2-\g}\a(b-(2-\g))\\\b(\b-(2-\g))+(\g-1)(b-(2-\g))+\frac{1}{y^{\frac{2}{2-\g}}}(\frac{4\pi}{4-3\g}\a\b-\frac{2\g}{2-\g} \a^{\g-1})}.
\]
We note that $\Et,\Ft$ are smooth with respect to $a,b$ in a neighbhorhood of $\pt,\ot=(k,2-\g)$ and we have
\[
  \Ft(y,k,2-\g)=0.
\]
We now discuss boundary conditions at the sonic points in these variables. We comment that in terms of $p,\o$, the condition \eqref{intermedstep0} becomes
\begin{equation}
  y_*=\frac{\g^{\frac{2-\g}{2}}\pt(y_*)^{\frac{(\g-1)(2-\g)}{2}}}{\ot(y_*)^{2-\g}}.
  \label{nonlineq1}
\end{equation}
If we substitute this into \eqref{eq:rho0sonic} we get the equation
\begin{equation}
  \pt(y_*)=\frac{\g^{\frac{1}{2-\g}}}{\ot(y_*)^{\frac{2}{2-\g}}}(\frac{4-3\g}{4\pi}\frac{2\ot(y_*)^2+(\g-1)\ot(y_*)+(2-\g)(\g-1)}{\ot(y_*)})^{\frac{1}{2-\g}}.
  \label{nonlineq2}
\end{equation}
Indeed, we see that if $\ot(y_*)$ is determined, then by \eqref{nonlineq2}, $\pt(y_*)$ is determined, and these together determine $y_*$ by \eqref{nonlineq1}.

To compute the Larson-Penston-Hunter-type conditions for the derivatives of $\pt,\ot$, we see that $R$, defined in Lemma \ref{lem:normalform}, takes the form
\[
  R=\frac{y_* \pt'(y_*)}{\pt(y_*)}-\frac{2}{2-\g}
\]
and recall that it has to satisfy a quadratic equation, \eqref{eq:Rquadratic}. Once $R$ is determined, then we can solve for $W=y_*\tilde{\omega}'(y_*)$ by recalling that it must satisfy
\begin{equation}
  \ot(y_*)R+W=4-3\g-3\ot(y_*).
  \label{nonlineq3}
\end{equation}
We comment that from \eqref{eq:yf_def}, we may derive the useful identities
\[
  ky_f^{-\frac{2}{2-\g}}=\frac{4-3\g}{2\pi},\qquad \g k^{\g-2}=\frac{(2-\g)^2 y_f^{\frac{2}{2-\g}}}{k}.
\]
We parameterize the exterior solutions by the value $\ot(y_*)$ which we set to be $\omgr_0[\ep]=\frac{2-\g}{1+\epsilon}$. One then defines $\pr_0[\ep]$ by substituting $\ot(y_*)=\omgr_0[\ep]$ into \eqref{nonlineq2} and solving for $\pt(y_*)$. Similarly, we then define $\yr_*[\ep]$ via \eqref{nonlineq1}. One then defines, for $\epsilon$ close to $0$, that $\mathring{R}[\epsilon]$ is the solution branch near $\mathring{R}[0]=-\frac{2}{2-\g}$ of \eqref{eq:Rquadratic}. Finally, one defines $\mathring{W}[\epsilon]$ via \eqref{nonlineq3}.
We then define
\[
  \pr_1[\ep]:=\frac{\mathring{p}[\epsilon]}{\mathring{y}_*[\epsilon]}\left(\mathring{R}[\epsilon]+\frac{2}{2-\g}\right),\qquad \omgr_1[\ep]:=\frac{\mathring{W}[\epsilon]}{\mathring{y}_*[\epsilon]}.
\]
Thus, in terms of the variables $(\pt,\ot)$, Larson-Penston-Hunter-type expansion conditions at $y_*=\mathring{y}_*[\epsilon]$ become
\begin{equation}
  \pt(y_*)=\pr_0[\epsilon],\qquad \ot(y_*)=\omgr_0[\epsilon],\qquad \pt'(y_*)=\pr_1[\ep],\qquad \ot'(y_*)=\omgr_1[\ep].
  \label{posonicconds}
\end{equation}
As a comment, we see from the above identities that as $\epsilon\to 0$, we have
\begin{equation}
  \omgr_0[\epsilon]=(2-\g)-(2-\g)\epsilon+O(\epsilon^2),
  \label{oepsilonexpansion}
\end{equation}
\begin{equation}
  \pr_0[\epsilon]=k+\frac{3\g-1}{2(2-\g)}k\epsilon+O(\epsilon^2),
  \label{pepsilonexpansion}
\end{equation}
\[
  \mathring{y}_*[\epsilon]=y_f+\frac{3\g^2-8\g+9}{4}y_f\epsilon+O(\epsilon^2),
\]
\[
  \mathring{R}[\epsilon]=-\frac{2}{2-\g}+\frac{-9\g^2+9\g+2}{(5\g-3)(2-\g)}\epsilon+O(\epsilon^2),
\]
\[
  \mathring{W}[\epsilon]=\frac{2(7-3\g)(\g-1)}{5\g-3}\epsilon+O(\epsilon^2).
\]
In particular, these imply that
\begin{equation}
  \pr_{1}[\ep]=\frac{(-9\g^2+9\g+2)}{(5\g-3)(2-\g)}\frac{k}{y_f}\epsilon+O(\epsilon^2),
  \label{pderepsilonexpansion}
\end{equation}
\begin{equation}
  \omgr_1[\ep]=\frac{2(7-3\g)(\g-1)}{(5\g-3)y_f}\epsilon+O(\epsilon^2).
  \label{oderepsilonexpansion}
\end{equation}
We now do an $\epsilon$-dependent change of variables that will move the sonic point from $y_0$ to $y_f$. In particular, we will define
\[
  \alpha:=\frac{y_f-y_0}{\mathring{y}_*-y_0},\qquad \beta=:\frac{\mathring{y}_*-y_f}{\mathring{y}_*-y_0}y_0,\qquad z:=\alpha y+\beta,
\]
\begin{equation}
  \pt:=\ul{p}(\alpha y+\beta),\qquad \ot:=\ul{\o}(\alpha y+\beta).
  \label{ulsub}
\end{equation}
As a comment, $\al,\be$ are chosen precisely so that
\[
  \al y_0+\be=y_0,\qquad \al y_*+\be=y_f.
\]
We note that as $\ep\to 0$, we have
\[
  \al=1+O(\ep),\qquad \be=O(\ep).
\]
This gives the system
\[
  \Eul(z,\pul,\oul,y_0,\ep)\mqty{\pul'\\ \oul'}+\Ful(z,\pul,\oul,y_0,\ep)=0
\]
where (in terms of the dummy variables $a,b$ for $\pt,\ot$)
\[
  \Eul(z,a,b,y_0,\ep):=\al \Et(\frac{z-\beta}{\al},a,b),\qquad \Ful(z,\a,\b,y_0,\ep):=\Ft(\frac{z-\be}{\al},a,b).
\]
We note that in these variables, the sonic point boundary conditions become
\[
  \pul(y_f)=\pr_0[\ep],\qquad \oul(y_f)=\omgr_0[\ep],\qquad \pul'(y_f)=\frac{1}{\al}\pr_0[\ep],\qquad \oul'(y_f)=\frac{1}{\al}\omgr_1[\ep].
\]
We then define $\ph,\oh$ to be first order difference quotients, which are defined implicitly by
\begin{equation}
  \begin{cases}
    \pul=k+\ep \ph\\
    \oul=2-\g+\ep\oh
  \end{cases}.
  \label{hatsub}
\end{equation}
To translate the boundary conditions to these variables, we define
\[
  \begin{cases}
    \pr_{0,d}:=\frac{\pr[\ep]-k}{\ep}\\
    \omgr_{0,d}:=\frac{\omgr[\ep]-(2-\g)}{\ep}\\
    \pr_{1,d}:=\frac{1}{\al \ep}\pr_1[\ep]\\
    \omgr_{1,d}:=\frac{1}{\al \ep}\omgr_1[\ep]
  \end{cases}.
\]
and note that the conditions for $\ph,\oh$ are
\[
  \ph(\yf)=\pr_{0,d}[\ep],\qquad 
  \oh(\yf)=\omgr_{0,d}[\ep],\qquad 
  \ph'(\yf)=\pr_{1,d}[\ep],\qquad
  \oh'(\yf)=\omgr_{1,d}[\ep].
\]
We also note that
\begin{equation}
  \label{eq:sonic_cond_epsilon_a}
  \pr_{0,d}[\ep]=\frac{(3\g-1)}{2(2-\g)}k+O(\ep),\qquad \pr_{1,d}[\ep]=\frac{-9\g^2+9\g+2}{(5\g-3)(2-\g)}\frac{k}{y_f}+O(\ep),
\end{equation}
\begin{equation}
  \label{eq:sonic_cond_epsilon_b}
  \omgr_{0,d}[\ep]=-(2-\g)+O(\ep),\qquad \omgr_{1,d}[\ep]=\frac{2(\g-1)(7-3\g)}{5\g-3}\frac{1}{y_f}+O(\ep).
\end{equation}
As a comment, unlike \cite{sandine}, we do not perform additional manipulations to make these constants independent of $\ep$, since this seems to add additional complications in the case of $\g>1$. Substituting the change of variables \eqref{ulsub} into \eqref{pomegaeqabstract} yields
\begin{equation}
  \Eul(z,\pul,\oul,y_0,\epsilon)\p_z \mqty{\pul\\ \oul}+\Ful(z,\pul,\oul,y_0,\epsilon)=0.
  \label{pouleqabstract}
\end{equation}
where
\[
  \Eul(z,\a,\b,y_0,\ep)=\alpha\cdot \Et(\frac{z-\be}{\alpha},\a,\b),\qquad \Ful(z,\pul,\oul,y_0,\ep)=\Ft(\frac{z-\be}{\alpha},\a,\b).
\]
We then substitute change of variables \eqref{hatsub} into \eqref{pouleqabstract} and obtain the equation
\begin{equation}
  \Eh(z,\ph,\oh,y_0,\ep)\p_{z}\mqty{\ph\\ \oh}+\Fh(z,\ph,\oh,y_0,\ep)=0,
  \label{hateq}
\end{equation}
where
\[
  \Eh(z,\a,\b,y_0,\ep):=\Eul(z,k+\ep \a,2-\g+\ep \b,y_0,\ep),
\]
\[
  \Fh(z,\a,\b,y_0,\ep):=\frac{1}{\ep}\Ful(z,k+\ep \a,2-\g+\ep \b,y_0,\ep).
\]
\subsection{First order difference quotients in $\ep$}
We now perform several Taylor expansions, beginning by expanding the symbols in terms of the dummy variables $\a,\b$. We start by noting that
\[
  \Et(y,\a,\b)=\mqty{(2-\g)y & k y \\ \frac{(2-\g)^2 y_f^{\frac{2}{2-\g}}}{k}\frac{y}{y^{\frac{2}{2-\g}}} & (2-\g)y}+O(\a-k)+O(\b-(2-\g)),
\]
It follows that uniformly for $z\in (\frac12,2)$, we have that
\[
  \Eul(z,a,b,y_0,\ep)=\mqty{(2-\g)z & k z \\ \frac{(2-\g)^2 y_f^{\frac{2}{2-\g}}}{k}\frac{z}{z^{\frac{2}{2-\g}}} & (2-\g)z}+O(\a-k)+O(\b-(2-\g))+O(\ep).
\]
We then have that
\[
  \Eh(z,\a,\b,y_0,\ep)=\mqty{(2-\g)z & k z \\ \frac{(2-\g)^2 y_f^{\frac{2}{2-\g}}}{k}\frac{z}{z^{\frac{2}{2-\g}}} & (2-\g)z}+O(\ep).
\]

We similarly Taylor expand
\begin{multline*}
  \Ft(y,a,b)=\mqty{0 & \frac{4-3\g}{2-\g}k\\ \frac{2(2-\g)^2 y_f^{\frac{2}{2-\g}}}{y^{\frac{2}{2-\g}}k} & 1+\frac{2y_f^{\frac{2}{2-\g}}}{y^{\frac{2}{2-\g}}}}\mqty{a-k\\ b-(2-\g)}\\+O((a-k)^2)+O((a-k)(b-(2-\g)))+O((b-(2-\g))^2).
\end{multline*}
From this it follows that
\begin{multline*}
  \Ful(z,a,b,y_0,\ep)=\mqty{0 & \frac{4-3\g}{2-\g}k\\ \frac{2(2-\g)^2 y_f^{\frac{2}{2-\g}}}{z^{\frac{2}{2-\g}}k} & 1+\frac{2y_f^{\frac{2}{2-\g}}}{z^{\frac{2}{2-\g}}}}\mqty{a-k\\ b-(2-\g)}\\+O((a-k)^2)+O((a-k)(b-(2-\g)))+O((b-(2-\g))^2)+O(\ep).
\end{multline*}
It then follows that
\[
  \Fh(z,a,b,y_0,\ep)=\mqty{0 & \frac{4-3\g}{2-\g}k\\ \frac{2(2-\g)^2 y_f^{\frac{2}{2-\g}}}{z^{\frac{2}{2-\g}} k} & 1+\frac{2y_f^{\frac{2}{2-\g}}}{z^{\frac{2}{2-\g}}}}\mqty{a\\ b}+O(\ep).
\]
Identitfying the leading terms for $\Eh$, $\Fh$, we introduce the notations (note that $F_{lin}$ is a $2\times 2$ matrix while $\hat{F}$ is a vector)
\[
  E_{lin}(z)=\mqty{(2-\g)z & k z \\ \frac{(2-\g)^2 y_f^{\frac{2}{2-\g}}}{k}\frac{z}{z^{\frac{2}{2-\g}}} & (2-\g)z},\qquad F_{lin}(z)=\mqty{0 & \frac{4-3\g}{2-\g}k\\ \frac{2(2-\g)^2 y_f^{\frac{2}{2-\g}}}{z^{\frac{2}{2-\g}}k} & 1+\frac{2y_f^{\frac{2}{2-\g}}}{z^{\frac{2}{2-\g}}}}.
\]
We then define the linear operator
\[
  L(\hat{p},\hat{\o})=\mqty{L_1(\hat{p},\hat{\o})\\ L_{2}(\hat{p},\hat{\o})}:=E_{lin}\dv{z} \mqty{\hat{p}\\ \hat{\o}}+F_{lin}\mqty{\hat{p}\\ \hat{\o}}.
\]
and note formally that as $\ep\to 0$, the linearization of \eqref{hateq} is $L(\hat{p},\hat{\o})=0$.

As a comment, before writing down the difference equations, we will let
\[
  S[\ep]=\mqty{1 & 1\\ -\frac{\pr_0[\ep]}{\omgr_0[\ep]} & \frac{\pr_0[\ep]}{\omgr_0[\ep]}}\mqty{1 & 1\\ -\frac{k}{2-\g} & \frac{k}{2-\g}}^{-1}=\mqty{1 & 0\\ 0 & \frac{2-\g}{k}\frac{\pr_0[\ep]}{\omgr_0[\ep]}}.
\]
and multiply on the left by $S^T$ \eqref{hateq}. This symmetrization is done to fix a loss of derivatives that occurs when doing difference estimates, and we note that to zeroth order in $\ep$, $S$ is the identity matrix. Using this matrix we define
\begin{multline}
  \label{nhdef}
  \Nh(\ep,\ph,\oh)=\mqty{\Nh_1\\ \Nh_2}\\:=-\left(S^T \Eh(y,\ph,\oh,y_0,\ep)-E_{lin}\right)\mqty{\ph'\\ \oh'}-\left(S^T \Fh(y,\ph,\oh,y_0,\ep)-F_{lin}\mqty{\ph\\ \oh}\right).
\end{multline}
We rewrite the non-linear system as
\begin{equation}
  \begin{cases}
    L(\hat{p},\hat{\o}):=\Nh(\ep,\ph,\oh)\\
    \ph(\yf)=\pr_{0,d}[\ep],\qquad 
    \oh(\yf)=\omgr_{0,d}[\ep],\qquad 
    \ph'(\yf)=\pr_{1,d}[\ep],\qquad
    \oh'(\yf)=\omgr_{1,d}[\ep]
  \end{cases}.
  \label{nonlinhatvars}
\end{equation}
\subsection{Properties of the linearized operator}
We consider the linearized equation with an inhomogeneity,
\begin{equation}
  \mqty{(2-\g)z & kz\\ \frac{(2-\g)^2 y_f^{\frac{2}{2-\g}}}{k} z^{1-\frac{2}{2-\g}} & (2-\g)z}\mqty{p'\\ \o'}+\mqty{0 & \frac{4-3\g}{2-\g}k\\ \frac{2(2-\g)^2 y_f^{\frac{2}{2-\g}}}{k z^{\frac{2}{2-\g}}} & 1+\frac{2y_f^{\frac{2}{2-\g}}}{z^{\frac{2}{2-\g}}}}\mqty{p\\ \o}=\mqty{P\\ \Omega}.
  \label{inhomsys}
\end{equation}
We note that at $z=y_f$, we can simplify these matrices to zeroth order as
\[
  \mqty{(2-\g)y_f & k y_f\\ \frac{(2-\g)^2}{k}y_f & (2-\g)y_f},\qquad \mqty{0 & \frac{4-3\g}{2-\g}k\\ \frac{2(2-\g)^2}{k} & 3}
\]
and in particular the coefficient matrix has eigenvalues $\{0,2(2-\g)y_f\}$. Left-eigenvectors corresponding to each of these are
\[
  \mqty{1 & -\frac{k}{2-\g}},\qquad \mqty{1 & \frac{k}{2-\g}}
\]
When we multiply the equation (evaluated at $y=y_f$) on the left by each of these respectively, we obtain the following constraints:
\begin{equation}
  \begin{cases}
    -2(2-\g)p(y_f)-\frac{3\g-1}{2-\g}k\omega(y_f)&=P(y_f)-\frac{k}{2-\g}\Omega(y_f)\\
    2(2-\g)y_f p'(y_f)+2k y_f \omega'(y_f)&\\
    \qquad+2(2-\g)p(y_f)+\frac{7-3\g}{2-\g}k\omega(y_f)&=P(y_f)+\frac{k}{2-\g}\Omega(y_f)
  \end{cases}.
  \label{firstmatchingconds}
\end{equation}
If we start with \eqref{inhomsys}, invert the coefficient matrix, and multiply on the left by $\mqty{1 & -\frac{k}{2-\g}}$, we obain
\[
  -2(2-\g)\frac{p}{z-y_f}-\frac{k(3\g-1)}{2-\g}\frac{\o}{z-y_f}=\frac{(P-\frac{k}{2-\g}\Omega)}{z-y_f}+o(\frac{1}{z-y_f}).
\]
We see that if $(p,\o)\in C^1$ then
\[
  \frac{P-\frac{k}{2-\g}\Omega-(P(y_f)-\frac{k}{2-\g}\Omega(y_f))}{z-y_f}
\]
must be continuous on a neighborhood of $z=y_f$. Multiplying the original equation by $\mqty{1 & -\frac{k}{2-\g}}$ and expanding to $o(z-y_f)$ gives
\begin{multline*}
  2(z-y_f)p'(y_f)+\mqty{-2(2-\g)p(y_f)-\frac{3\g-1}{2-\g} \omega(y_f) }\\
  +\left(-2(2-\g)p'(y_f)-\frac{(3\g-1)k \omega'(y_f)}{2-\g}\right)(z-y_f)+(\frac{4p(y_f)}{y_*}+\frac{4 k \omega(y_f)}{(2-\g)^2y_f})(z-y_f)\\
  =P(y_f)-\frac{k}{2-\g}\Omega(y_f)+(P-\frac{k}{2-\g}\Omega)'(y_f)\cdot (z-y_f)+o(z-y_f).
\end{multline*}
By \eqref{firstmatchingconds}, the equations agree to order $0$. Comparing the terms of order $1$, we get the following constraint
\begin{equation}
  2(\g-1)y_f p'(y_f)-\frac{(3\g-1)}{2-\g}k y_fw'(y_f)+4p(y_f)+\frac{4}{(2-\g)^2}k\omega(y_f)=y_f(P-\frac{k}{2-\g}\Omega)'(y_f).
  \label{secondmatchingconds}
\end{equation}
We now suppose that
\begin{equation}
  P=\sigma_1+o(1),\qquad \O=\sigma_2+o(1),\qquad \frac{P-\frac{k}{2-\g}\O}{z-y_f}=\frac{(\sigma_1-\frac{k}{2-\g}\sigma_2)}{z-y_f}+\sigma_3+o(1).
  \label{rhsconstraints}
\end{equation}
We will also assume
\begin{equation}
  p=p_0+p_1(z-y_f)+o(z-y_f),\qquad \omg=\omg_0+\frac{1}{y_f}\omg_1(z-y_f)+o(z-y_f).
  \label{lhscoes}
\end{equation}
In terms of these notations, the system \eqref{firstmatchingconds}, \eqref{secondmatchingconds} becomes
\[
  \begin{cases}
    -2(2-\g)p_0-\frac{3\g-1}{2-\g}k \omg_0=\sigma_1-\frac{k}{2-\g}\sigma_2\\
    2(2-\g)y_fp_1+2 k y_f \o_1+2(2-\g)p_0+\frac{7-3\g}{2-\g}k \omg_0=\sigma_1+\frac{k}{2-\g}\sigma_2\\
    2(\g-1)y_fp_1-\frac{(3\g-1)}{2-\g}y_f k \o_1+4p_0+\frac{4}{(2-\g)^2}k \omg_0=y_f\sigma_3
  \end{cases}.
\]
We will consider $\o_0$ to be a free variable. From the first equation, we have that
\begin{equation}
  p_0=-\frac{(3\g-1)}{2(2-\g)^2}k \o_0-\frac{\sigma_1}{2(2-\g)}+\frac{k \sigma_2}{2(2-\g)^2}.
  \label{p0affine}
\end{equation}
We may then substitute this into the next two equations to obtain the system
\[
  \begin{cases}
    2(2-\g)y_fp_1+2y_fk\o_1+\frac{2(4-3\g)}{(2-\g)}k\omg_0=2\sigma_1\\
    2(\g-1)y_fp_1-\frac{3\g-1}{2-\g}y_f k\omg_1-\frac{6(\g-1)}{(2-\g)^2}k\omg_0=\frac{2}{2-\g}\sigma_1-\frac{2k\sigma_2}{(2-\g)^2}+y_f\sigma_3
  \end{cases}.
\]
We may multiply the first of these equations by $\frac{\g-1}{2-\g}$, and then subtract the second equation from it to obtain
\begin{equation}
  \o_1=-\frac{2(\g-1)(7-3\g)}{(2-\g)(5\g-3)}\frac{1}{y_f}\o_0-\frac{2(2-\g)}{5\g-3}\frac{1}{y_f k}\sigma_1+\frac{2}{(5\g-3)(2-\g)}\frac{1}{y_f}\sigma_2-\frac{2-\g}{5\g-3}\frac{1}{k}\sigma_3.
  \label{o1affine}
\end{equation}
\begin{multline}
  p_1=-\frac{-9\g^2+9\g+2}{(2-\g)^2(5\g-3)}\frac{k}{y_f}\o_0+\frac{3\g+1}{(2-\g)(5\g-3)}\frac{1}{y_f}\sigma_1\\
  -\frac{2}{(5\g-3)(2-\g)^2}\frac{k}{y_f}\sigma_2+\frac{1}{5\g-3}\sigma_3.
  \label{p1affine}
\end{multline}
We comment that if $\sigma_1=\sigma_2=\sigma_3=0$, then, by plugging in $\o_0=-(2-\g)$, the homogeneous solution has first-order Taylor expansion
\begin{equation}
  \begin{cases}
    p&=\frac{(3\g-1)}{2(2-\g)}k+\frac{-9\g^2+9\g+2}{(5\g-3)(2-\g)}\frac{k}{y_f}(z-y_f)+o(z-y_f)\\
    \omega&=-(2-\g)+\frac{2(\g-1)(7-3\g)}{5\g-3}\frac{1}{y_f}(z-y_f)+o(z-y_f)
  \end{cases}.
  \label{lineartaylorcoes}
\end{equation}
These Taylor coefficients are consistent with those computed in \eqref{eq:sonic_cond_epsilon_a}, \eqref{eq:sonic_cond_epsilon_b}. We now state the main estimate regarding the linearized operator. As a comment, in comparison to the analogous statement in \cite{sandine}, part (e) is now modified, as in this problem it is more convenient to allow for non-zero values of $\sigma_1,\sigma_2,\sigma_3$.
\begin{lem}
  \label{ext_fun_matrix_lemma}
  \begin{enumerate}
  \item
    There exists a unique homogeneous solution $(p_{hom}(z),\o_{hom}(z))$ such that $L(p_{hom},\o_{hom})=0$ and such that its first order Taylor polynomials at $z=y_f$ agree with \eqref{lineartaylorcoes}.
  \item
    There exists analytic fundamental matrices $U_{\infty}(z)$, $U_0(z)$ for $L$ defined on $(1,\infty)$ and $(0,1)$ respectively. Moreover, these can be chosen such that as $z\to \infty$
    \[
      U_{\infty}(z)=
      \mqty{1& -\frac{3(\g-1)k}{2-\g}\frac{1}{z^{\frac{1}{2-\g}}}\\
        0 & \frac{1}{z^{\frac{1}{2-\g}}}}+\mqty{O(\frac{1}{z^{\frac{2}{2-\g}}}) & O(\frac{1}{z^{\frac{3}{2-\g}}})\\ O(\frac{1}{z^{\frac{2}{2-\g}}}) & O(\frac{1}{z^{\frac{3}{2-\g}}})},
    \]
    \[
      U_{\infty}^{-1}=\mqty{1 & \frac{3(\g-1)k}{2-\g}\\ 0 &  z^{\frac{1}{2-\g}}}+\mqty{O(\frac{1}{z^{\frac{2}{2-\g}}}) & O(\frac{1}{z^{\frac{2}{2-\g}}})\\ O(\frac{1}{z^{\frac{1}{2-\g}}}) & O(\frac{1}{z^{\frac{1}{2-\g}}})},
    \]
    and as $z\to 0$,
    \[
      U_0(z)=\frac{1}{z^{\mu}}\mqty{\cos(\nu \log(z)) & \sin(\nu \log(z))\\ \cos(\nu \log(z)+\theta_0) & \sin(\nu \log(z)+\theta_0)}+O(z^{-\mu+\frac{2}{2-\g}}),
    \]
    \[
      U_0^{-1}(z)=\frac{z^{\mu}}{\sin(\theta_0)}\mqty{\sin(\nu \log(z)+\theta_0) & -\sin(\nu \log(z))\\ -\cos(\nu \log(z)+\theta_0) & \cos(\nu \log(z))}+O(z^{\mu+\frac{2}{2-\g}}).
    \]
  \item
    There exists $\mu_3,\mu_4,c_1,d_1$with $\sqrt{\mu_3^2+\mu_4^2}\neq 0$, $c_1>0$, $d_1\in \R/2\pi \Z$ such that on $(1,\infty)$ and $(0,1)$ respectively we have
    \begin{equation}
      \mqty{p_{hom}\\ \o_{hom}}=U_{\infty}\mqty{\mu_3\\ \mu_4},\qquad \mqty{p_{hom}\\ \o_{hom}}=U_0\mqty{c_1 \sin(d_1)\\ c_1\cos(d_1)}.
      \label{hom0asymptots}
    \end{equation}
  \item
    Also, there exists constant invertible matrices $M_0,M_{\infty}$ such that as $z\to 1^{\pm}$, we have
    \begin{multline*}
      U_{\infty}M_{\infty},U_0M_0=\mqty{\frac{(3\g-1)k}{2(2-\g)} &\frac{(3\g-1)k}{2(2-\g)} \abs{z-y_f}^{-\frac{5(\g-1)}{2}}\\ -(2-\g)& -\frac{3\g-1}{2}\abs{z-y_f}^{-\frac{5(\g-1)}{2}}}\\
      +\mqty{\frac{(-9\g^2+9\g+2)}{(2-\g)(5\g-3)}\frac{k}{y_f}(z-y_f) & \frac{(3 \gamma -1) \left(25 \gamma ^3-136 \gamma ^2+191 \gamma -72\right) k}{8 (2-\gamma )^2 (7-5 \gamma ) y_f}(z-y_f)\abs{z-y_f}^{-\frac{5(\g-1)}{2}}\\ \frac{2(\g-1)(7-3\g)}{y_f(5\g-3)}(z-y_f) & -\frac{(\gamma -1) (3 \gamma -1) \left(25 \gamma ^2-111 \gamma +104\right)}{8 (2-\gamma ) (7-5 \gamma ) y_f}(z-y_f)\abs{z-y_f}^{-\frac{5(\g-1)}{2}}}\\
      +\mqty{o(z-y_f) & o(z-y_f)\abs{z-y_f}^{-\frac{5(\g-1)}{2}}\\ o(z-y_f) & o(z-y_f)\abs{z-y_f}^{-\frac{5(\g-1)}{2}}}.
    \end{multline*}
  \item
    If we define the principal matrix solution as
    \[
      S(z,z')=\begin{cases}
        U_0(z)U_0^{-1}(z')\mqty{(2-\g)z' & kz'\\ \frac{2\pi(2-\g)^2}{4-3\g}(z')^{1-\frac{2}{2-\g}} & (2-\g)z'}^{-1},\qquad z,z'\in (0,1)\\
        U_\infty(z)U_\infty^{-1}(z')\mqty{(2-\g)z' & kz'\\ \frac{2\pi(2-\g)^2}{4-3\g}(z')^{1-\frac{2}{2-\g}} & (2-\g)z'}^{-1},\qquad z,z'\in (1,\infty)
      \end{cases},
    \]
    then if $P(z),\Omega(z),\frac{P-\frac{k}{2-\g}\Omega-P(y_f)-\frac{k}{2-\g}\Omega(y_f)}{y-y_f}$ are continuous on $(0,\infty)$ and satisfy \eqref{rhsconstraints} then
    \[
      (p,\o):=R(P(z),\O(z)):=\int_{1}^zS(z,z')\mqty{P(z')\\ \O(z')}dz'
    \]
    defines an element of $C^1((0,\infty))$ which solves $L(p,\o)=(P,\O)$ and the coefficients (defined by \eqref{lhscoes}) satisfy \eqref{p0affine}, \eqref{o1affine}, \eqref{p1affine} and also
    \[
      \o_0=-\frac{\frac{(2-\gamma )}{k} \sigma _1-\sigma _2}{5 (\gamma -1)}.
    \]
  \end{enumerate}
\end{lem}
For the proof see Section \ref{ext_fun_matrix_proof}.
\subsection{Sonic point conditions at the difference equation level}
\label{sec:ext_diffeq}

We now note that if $p,\omg$ are $C^1$ then as $y\to y_*$, we have
\begin{align*}
  \Et(y_*,\pt,\ot)&:=\Et(y_*,\pt(y_*),\ot(y_*))+\Et_1(y_*,\pt(y_*),\ot(y_*),\pt'(y_*),\ot'(y_*))(y-y_*)+o(y-y_*),\\
  \Ft(y_*,\pt,\ot)&:=\Ft(y_*,\pt(y_*),\ot(y_*))+\Ft_1(y_*,\pt(y_*),\ot(y_*),\pt'(y_*),\ot'(y_*))(y-y_*)+o(y-y_*),
\end{align*}
where
\begin{align*}
  \Et_1(y,\a,\b,\a_1,\b_1)&=(\pdv[\Et]{y}+\pdv[\Et]{\a}\a_1+\pdv[\Et]{\b}\b_1)|_{y,\a,\b},\\
  \Ft_1(y,\a,\b,\a_1,\b_1)&=(\pdv[\Ft]{y}+\pdv[\Ft]{\a}\a_1+\pdv[\Ft]{\b}\b_1)|_{y,\a,\b}.
\end{align*}
We now note that by \eqref{nonlineq1}, \eqref{nonlineq2}, \eqref{nonlineq3} and \eqref{eq:Rquadratic} that the following equations hold
\begin{equation}
  \begin{cases}
    \mqty{1& -\frac{\pr_0[\ep]}{\omgr_0[\ep]}}\Et(\yr_*[\ep],\pr_0[\ep],\omgr_0[\ep])&=0\\
    \mqty{1& -\frac{\pr_0[\ep]}{\omgr_0[\ep]}}\Ft(\yr_*[\ep],\pr_0[\ep],\omgr_0[\ep])&=0\\
    \mqty{1& \frac{\pr_0[\ep]}{\omgr_0[\ep]}}\Et(\yr_*[\ep],\pr_0[\ep],\omgr_0[\ep])\cdot \mqty{\pr_1[\ep]\\ \omgr_1[\ep]}+\mqty{1& \frac{\pr_0[\ep]}{\omgr_0[\ep]}}\Ft(\yr_*[\ep],\pr_0[\ep],\omgr_0[\ep])&=0\\
    \mqty{1& -\frac{\pr_0[\ep]}{\omgr_0[\ep]}}\Et_1(\yr_*[\ep],\pr_0[\ep],\omgr_0[\ep],\pr_1[\ep],\omgr_1[\ep])\mqty{\pr_1[\ep]\\ \omgr_1[\ep]}&\\
    \qquad+\mqty{1& -\frac{\pr_0[\ep]}{\omgr_0[\ep]}}\Ft_1(\yr_*[\ep],\pr_0[\ep],\omgr_0[\ep],\pr_1[\ep],\omgr_1[\ep])&=0
  \end{cases}.
  \label{constrainttilde}
\end{equation}
We next transition variables. We define $\Eul_1,\Ful_1,\Eh_1,\Eh_2$ by requiring
\begin{align*}
  \Eul(y_f,\pul,\oul)&=\Eul(y_f,\pul(y_f),\oul(y_f))+\Eul_1(y_f,\pul(y_f),\oul(y_f),\pul'(y_f),\oul'(y_f))(z-y_f)+o(z-y_f),\\
  \Ful(y_f,\pul,\oul)&=\Ful(y_f,\pul(y_f),\oul(y_f))+\Ful_1(y_f,\pul(y_f),\oul(y_f),\pul'(y_f),\oul'(y_f))(z-y_f)+o(z-y_f),\\
  \Eh(y_f,\ph,\oh)&=\Eh(y_f,\ph(y_f),\oh(y_f))+\Eh_1(y_f,\ph(y_f),\oh(y_f),\ph'(y_f),\oh'(y_f))(z-y_f)+o(z-y_f),\\
  \Fh(y_f,\ph,\oh)&=\Fh(y_f,\ph(y_f),\oh(y_f))+\Fh_1(y_f,\ph(y_f),\oh(y_f),\ph'(y_f),\oh'(y_f))(z-y_f)+o(z-y_f)
\end{align*}
and note that these are given in terms of the previous functions by
\begin{align*}
  \Eul_1(z,\a,\b,\a_1,\b_2,y_0,\ep)&=\Et_1(\frac{z-\be}{\al},a,b,\al a_1,\al b_1),\\
  \Ful_1(z,\a,\b,\a_1,\b_2,y_0,\ep)&=\frac{1}{\al}\Ft_1(\frac{z-\be}{\al},a,b,\al a_1,\al b_1),\\
  \Eh_1(z,\a,\b,\a_1,\b_2,y_0,\ep)&=\Eul_1(z,k+\ep a,2-\g+\ep b,\ep a_1,\ep b_1,y_0,\ep),\\
  \Fh_1(z,\a,\b,\a_1,\b_2,y_0,\ep)&=\frac{1}{\ep}\Ful_1(z,k+\ep a,2-\g+\ep b,\ep a_1,\ep b_1,y_0,\ep).
\end{align*}
We then define the ``difference'' equation boundary condition functions
\[
  \begin{cases}
    \pr_{0,d}=\frac{\pr[\ep]-k}{\ep}\\
    \omgr_{0,d}=\frac{\omgr[\ep]-(2-\g)}{\ep}\\
    \pr_{1,d}=\frac{1}{\al \ep}\pr_1[\ep]\\
    \omgr_{1,d}=\frac{1}{\al \ep}\omgr_1[\ep]
  \end{cases}.
\]
From this we obtain the identities
\begin{equation}
  \begin{cases}
    \mqty{1& -\frac{\pr_0[\ep]}{\omgr_0[\ep]}}\Eh(\yf,\pr_{0,d}[\ep],\omgr_{0,d}[\ep])&=0\\
    \mqty{1& -\frac{\pr_0[\ep]}{\omgr_0[\ep]}}\Fh(\yf,\pr_{0,d}[\ep],\omgr_{0,d}[\ep])&=0\\
    \mqty{1& \frac{\pr_0[\ep]}{\omgr_0[\ep]}}\Eh(\yf,\pr_{0,d}[\ep],\omgr_{0,d}[\ep])\cdot \mqty{\pr_{1,d}[\ep]\\ \omgr_{1,d}[\ep]}+\mqty{1& \frac{\pr_0[\ep]}{\omgr_0[\ep]}}\Fh(\yf,\pr_{0,d}[\ep],\omgr_{0,d}[\ep])&=0\\
    \mqty{1& -\frac{\pr_0[\ep]}{\omgr_0[\ep]}}\Eh_1(\yf,\pr_{0,d}[\ep],\omgr_{0,d}[\ep],\pr_{1,d}[\ep],\omgr_{1,d}[\ep])\mqty{\pr_{1,d}[\ep]\\ \omgr_{1,d}[\ep]}&\\
    \qquad+\mqty{1& -\frac{\pr_0[\ep]}{\omgr_0[\ep]}}\Fh_1(\yf,\pr_{0,d}[\ep],\omgr_{0,d}[\ep],\pr_{1,d}[\ep],\omgr_{1,d}[\ep])&=0
  \end{cases}.
  \label{constrainthat}
\end{equation}
In particular, we return now to the difference equation \eqref{nonlinhatvars}. We suppose that the input functions, $(\ph_{in},\oh_{in})$, satisfy
\begin{equation}
  \ph_{in}(\yf)=\pr_{0,d}[\ep],\qquad 
  \oh_{in}(\yf)=\omgr_{0,d}[\ep],\qquad 
  \ph_{in}'(\yf)=\pr_{1,d}[\ep],\qquad
  \oh_{in}'(\yf)=\omgr_{1,d}[\ep].
  \label{inputboundaryconds}
\end{equation}
We may then note that
\begin{align*}
  \mqty{1 & -\frac{k}{2-\g}}\hat{N}(\ep,\ph_{in},\oh_{in})&=-\mqty{1 & -\frac{\pr_0}{\omgr_0}}\Eh\mqty{\ph_{in}'\\ \oh_{in}'}+\mqty{1 & -\frac{k}{2-\g}}E_{lin}\mqty{\ph_{in}'\\ \oh_{in}'}\\
          &\qquad-\mqty{1 & -\frac{\pr_0}{\omgr_0}}\Fh+\mqty{1 & -\frac{k}{2-\g}}F_{lin}\mqty{\ph_{in}\\ \oh_{in}},\\
  \mqty{1 & \frac{k}{2-\g}}\hat{N}(\ep,\ph_{in},\oh_{in})&=-\mqty{1 & \frac{\pr_0}{\omgr_0}}\Eh\mqty{\ph_{in}'\\ \oh_{in}'}+\mqty{1 & \frac{k}{2-\g}}E_{lin}\mqty{\ph_{in}'\\ \oh_{in}'}\\
          &\qquad-\mqty{1 & \frac{\pr_0}{\omgr_0}}\Fh+\mqty{1 & \frac{k}{2-\g}}F_{lin}\mqty{\ph_{in}\\ \oh_{in}}.
\end{align*}
We comment that by \eqref{inputboundaryconds}, \eqref{constrainthat} we have
\begin{align*}
  \mqty{1 & -\frac{\pr_0}{\omgr_0}}\Eh\mqty{\ph_{in}'\\ \oh_{in}'}+\mqty{1 & -\frac{\pr_0}{\omgr_0}}\Fh&=o(z-y_f),\\
  \mqty{1 & -\frac{k}{2-\g}}E_{lin}&=\mqty{2 & 0}(z-y_f)+o(z-y_f),\\
  \mqty{1 & -\frac{k}{2-\g}}F_{lin}&=\mqty{-2(2-\g) & -\frac{(3\g-1)k}{2-\g}}+\mqty{\frac{4}{y_f} & \frac{4k}{(2-\g)^2 y_f}}(z-y_f)+o(z-y_f),\\
  \mqty{1 & \frac{\pr_0}{\omgr_0}}\Eh\mqty{\ph_{in}'\\ \oh_{in}'}+\mqty{1 & \frac{\pr_0}{\omgr_0}}\Fh&=o(1),\\
  \mqty{1 & \frac{k}{2-\g}}E_{lin}&=\mqty{2(2-\g)y_f & 2k y_f}+o(1),\\
  \mqty{1 & \frac{k}{2-\g}}F_{lin}&=\mqty{2(2-\g) & (7-3\g)\frac{k}{2-\g}}+o(1).
\end{align*}
In particular, $\Nh$ satisfies \eqref{rhsconstraints} with coefficients
\[
  \begin{cases}
    \sr_1[\ep]:=(2-\g)y_f\pr_{1,d}+y_f k \omgr_{1,d}+\frac{4-3\g}{(2-\g)}k \omgr_{0,d}\\
    \sr_2[\ep]:=\frac{(2-\g)^2 y_f}{k}\pr_{1,d}+(2-\g)y_f \omgr_{1,d}+\frac{2(2-\g)^2}{k}\pr_{0,d}+3\omgr_{0,d}\\
    \sr_3[\ep]:=2(\g-1)\pr_{1,d}-\frac{(3\g-1)}{2-\g}k\omgr_{1,d}+\frac{4}{y_f}\pr_{0,d}+\frac{4 k }{(2-\g)^2y_f}\omgr_{0,d}
  \end{cases}.
\]
In this case, if one substitutes $\sr_1[\ep],\sr_2[\ep],\sr_3[\ep],\omgr_{0,d}$ along with into the equations \eqref{p0affine}, \eqref{o1affine}, \eqref{p1affine} then they are consistent with
\[
  (p_0,p_1,\omg_1)=(\pr_{0,d},\pr_{1,d},\omgr_{1,d}).
\]
\subsection{Difference equations and iteration scheme}
We now define the ``error'' boundary condition parameterization functions to be
\[
  \pr_{0,e}[\ep]=\pr_{0,d}[\ep]-\frac{(3\g-1)}{2(2-\g)}k,\qquad \omgr_{0,e}[\ep]=\omgr_{0,d}[\ep]+(2-\g)
\]
\[
  \pr_{1,e}[\ep]=\pr_{1,d}[\ep]-\frac{-9\g^2+9\g+2}{(5\g-3)(2-\g)}\frac{k}{y_f},\qquad \omgr_{1,e}=\omgr_{1,d}[\ep]-\frac{2(\g-1)(7-3\g)}{5\g-3}\frac{1}{y_f}.
\]
We note that
\[
  \pr_{0,e},\omgr_{0,e},\pr_{1,e},\omgr_{1,e}=O(\ep).
\]
We also have that in terms of these coefficients
\[
  \begin{cases}
    \sr_1[\ep]=(2-\g)y_f\pr_{1,e}+y_f k \omgr_{1,e}+\frac{4-3\g}{(2-\g)}k \omgr_{0,e}\\
    \sr_2[\ep]=\frac{(2-\g)^2 y_f}{k}\pr_{1,e}+(2-\g)y_f \omgr_{1,e}+\frac{2(2-\g)^2}{k}\pr_{0,d}+3\omgr_{0,e}\\
    \sr_3[\ep]=2(\g-1)\pr_{1,e}-\frac{(3\g-1)}{2-\g}k\omgr_{1,e}+\frac{4}{y_f}\pr_{0,e}+\frac{4 k }{(2-\g)^2y_f}\omgr_{0,e}
  \end{cases},
\]
so in particular,
\[
  \sr_1[\ep],\sr_2[\ep],\sr_{3}[\ep]=O(\ep).
\]
We now consider \eqref{nonlinhatvars} and substitute
\[
  (\hat{p},\hat{\o})=(p_{hom},\o_{hom})+(p,\o)
\]
which gives the following difference equations:
\begin{equation}
  \begin{cases}
    L(p,\o)=N(\epsilon,p,\o)\\
    p(y_f)=\pr_{0,e},\qquad \o(y_f)=\omgr_{0,e},\qquad p'(y_f)=\omgr_{1,e},\qquad \o'(y_f)=\omgr_{1,e}
    \label{ext_diffeq}
  \end{cases},
\end{equation}
where
\[
  N(\epsilon,p,\o):=\hat{N}(\epsilon,p_{hom}+p,\o_{hom}+\o).
\]
We then rewrite this as
\begin{equation}
  (p,\o)=RN(\epsilon,p,\o)+\left(\frac{\frac{2-\g}{k}\sr_1-\sr_2}{5(\g-1)}+\omgr_{0,e}\right)\frac{1}{-(2-\g)}(p_{hom},\o_{hom})
  \label{ext_integral_diffeq}.
\end{equation}
As a comment, to ensure that we are not dividing by zero in the original equation, we need solutions to satisfy
\[
  k+\ep(p_{hom}+p)>0
\]
The integral equation will be solved using a fixed point method. We first introduce
\[
  D_{\infty}f(z)=-z^{1+\frac{1}{2-\g}}\p_zf
\]
\begin{defin}
  We define the ``solution'' norm to be
  \[
    \norm{(p,\o)}_{X_{y_0}}:=\sup_{z\geq y_f}\abs{p}+\abs{z^{\frac{1}{2-\g}}\o}+\abs{D_{\infty}p}+\abs{D_{\infty}z^{\frac{1}{2-\g}}\o}+\sup_{y_0\leq z\leq y_f}z^{\mu}(\abs{p}+\abs{\o})+z^{\mu+1}(\abs{p'}+\abs{\o'}).
  \]
  We define our ``solution'' Banach space to be
  \begin{multline*}
    X_{y_0}=\{(p,\o)\in C^1([y_0,\infty))\colon \lim_{z\to \infty}D_{\infty}(p),D_{\infty}(z^{\frac{1}{2-\g}} \o) \text{\qquad exists},\qquad \norm{(p,\o)}_{X_{y_0}}<\infty\}.
  \end{multline*}
\end{defin}
\begin{defin}
  We define our ``nonlinear'' norm to be
  \begin{multline*}
    \norm{(P,\O)}_{N_{y_0}}:=\sup_{\frac{1}{2}\leq z\leq 2}\abs{\frac{P-\frac{k}{2-\g}\O-(P(0)-\frac{k}{2-\g}\O(0))}{z-y_f}}\\
    +\sup_{z\geq y_f}\abs{z^{\frac{1}{2-\g}} P}+\abs{z^{\frac{2}{2-\g}}\O}+\sup_{y_0\leq z\leq y_f}\abs{z^{2\mu} P}+\abs{z^{2\mu+\frac{2}{2-\g}} \O}.
  \end{multline*}
  We define our ``nonlinear'' Banach space to be
  \begin{multline*}
    N_{y_0}=\{(P,\O)\in C^0([y_0,\infty))\colon \frac{P-\frac{k}{2-\g}\O}{z-y_f}\in C^0([y_0,\infty)),\qquad \lim_{z\to \infty}(z^{\frac{1}{2-\g}}P,z^{\frac{2}{2-\g}}\O)\text{\qquad exists},\\
    \norm{(P,\O)}_{N_{y_0}}<\infty\}.
  \end{multline*}
\end{defin}
Now for each $\ep>0$, we shall define affine subspaces on which the iteration will take place. We define the following $\ep$-dependent affine spaces,
\[
  X_{y_0}[\ep]=\{(p,\o)\in X_{y_0}\colon p(y_f)=\pr_{0,e}[\ep],\qquad \o(y_f)=\omgr_{0,e}[\ep],\qquad p'(y_f)=\pr_{1,e}[\ep],\qquad \o'(y_f)=\omgr_{1,e}[\ep]\},
\]
\[
  N_{y_0}[\ep]=\{(P,\O)\in N_{y_0}\colon P(y_f)=\sr_1[\ep],\qquad \O(y_f)=\sr_2[\ep],\qquad (P-\frac{k}{2-\g}\O)'(y_f)=\sr_3[\ep]\}.
\]
We also define the difference spaces
\begin{multline*}
  X_{y_0}[\ep_1,\ep_2]=\{(p,\o)\in X_{y_0}\colon p(y_f)=\pr_{0,e}[\ep_1]-\pr_{0,e}[\ep_2],\qquad \o(y_f)=\omgr_{0,e}[\ep_1]-\omgr_{0,e}[\ep_2]\\
  \qquad p'(y_f)=\pr_{1,e}[\ep_1]-\pr_{1,e}[\ep_2],\qquad \o'(y_f)=\omgr_{1,e}[\ep_1]-\omgr_{1,e}[\ep_2]\},
\end{multline*}
\begin{multline*}
  N_{y_0}[\ep_1,\ep_2]=\{(P,\O)\in N_{y_0}\colon P(y_f)=\sr_1[\ep_1]-\sr_1[\ep_2],\qquad \O(y_f)=\sr_2[\ep_1]-\sr_2[\ep_2]\\
  \qquad (P-\frac{k}{2-\g}\O)'(y_f)=\sr_3[\ep_1]-\sr_3[\ep_2]\}.
\end{multline*}
\begin{lem}
  \label{ext_linear_lemma}
  For $y_0\leq \frac{y_f}{2}$ there exists $C_1>1$ such that if $(P,\O)\in N_{y_0}[\ep]$ then
  \[
    (p,\o):=R(P,\O)+\left(\frac{\frac{2-\g}{k}\sr_1[\ep]-\sr_2[\ep]}{5(\g-1)}+\omgr_{0,e}[\ep]\right)\frac{1}{-(2-\g)}(p_{hom},\o_{hom})\in X_{y_0}[\ep]
  \]
  and 
  \begin{equation}
    \label{eq:ext_linear_bound}
    \norm{(p,\o)}_{X_{y_0}}\leq \frac{1}{2}C_1 y_0^{-\mu}\norm{(P,\O)}_{N_{y_0}}+\frac{1}{2}C_1 \abs{\ep}.
  \end{equation}
  Moreover, if for some $\ep_1,\ep_2$, we have that $(P_i,\O_i)\in N_{y_0}[\ep_i]$ and we define
  \[
    (p_i,\o_i)=R(P_i,\O_i)+\left(\frac{\frac{2-\g}{k}\sr_1-\sr_2}{5(\g-1)}+\omgr_{0,e}\right)\frac{1}{-(2-\g)}(p_{hom},\o_{hom})
  \]
  then $(p_1-p_2,\o_1-\o_2)\in X_{y_0}[\ep_1,\ep_2]$ and
  \begin{equation}
    \label{eq:ext_linear_lipschitz}
    \norm{(p_1-p_2,\o_1-\o_2)}_{X_{y_0}}\leq \frac{1}{2}C_1 y_0^{-\mu}\norm{P_1-P_2,\O_1-\O_2}_{N_{y_0}}+\frac{1}{2}C_1\abs{\ep_1-\ep_2}.
  \end{equation}
\end{lem}
For a proof see Subsection \ref{ext_linear_proof}.

Also, we have that the non-linearity is bounded as follows.
\begin{lem}
  \label{ext_nonlinear_lemma}
  If $y_0\leq \frac{y_f}{2}$, then for all $C>0$ there exists $C_0(C)\ll 1$ such that if $\epsilon\leq C_0(C)y_0^{\mu}$ then there exists $C_2(C)\geq 1$ such that if $(p,\o)\in X_{y_0}[\ep],\norm{(p,\o)}_{X_{y_0}}\leq C$ then
  \begin{equation}
    \label{eq:ext_nonlinear_bound}
    (P,\O):=N(\epsilon,p,\o)\in N_{y_0}[\ep],\qquad \norm{(P,\O)}_{N_{y_0}}\leq \abs{\epsilon}C_2.
  \end{equation}
  Moreover, if for some $\ep_1,\ep_2$. we have $(p_i,\o_i)\in X_{y_0}[\ep_i]$ then $N(\ep_1,p_1,\o_1)-N(\ep_2,p_2,\o_2)\in N_{y_0}[\ep_1,\ep_2]$ and
  \begin{equation}
    \label{eq:ext_nonlinear_lipschitz}
    \norm{N(\ep_1,p_1,\o_1)-N(\ep_2,p_2,\o_2)}_{N_{y_0}}\leq C_2 \abs{\ep}\norm{(p_1,\o_1)-(p_2,\o_2)}_{X_{y_0}}+C_2\abs{\ep_1-\ep_2}.
  \end{equation}
\end{lem}
For a proof see Subsection \ref{ext_nonlinear_proof}.
\begin{lem}
  \label{ext_contraction_lemma}
  There exists a universal constant $C_3$ so that if one fixes $y_0\leq \frac{y_f}{2}$ then for $\abs{\ep}\ll y_0^{\mu}$ there exists a unique solution $(p[\ep](y),\o[\ep](y))\in X_{y_0}[\ep]$ to \eqref{ext_diffeq} with
  \[
    \norm{(p[\ep],\o[\ep])}_{X_{y_0}}\leq C_3 \abs{\ep}y_{0}^{-\mu}.
  \]
  Moreover, if $\abs{\ep_1},\abs{\ep_2}\ll y_0^{\mu}$, then we have the Lipschitz estimate
  \begin{equation}
    \norm{(p[\ep_1],\o[\ep_1])-(p[\ep_2],\o[\ep_2])}_{X_{y_0}}\leq C_3y_0^{-\mu}\abs{\ep_1-\ep_2}.
    \label{ext_lipschitz_1}
  \end{equation}
\end{lem}
For a proof see Subsection \ref{ext_contraction_proof}.

If we convert this back to the original variables we obtain
\[
  \rt=\frac{k}{y^{\frac{2}{2-\g}}}+\ep \frac{p_{hom}(\al y+\be)}{y^{\frac{2}{2-\g}}}+\ep \frac{p[\ep](\al y+\be)}{y^{\frac{2}{2-\g}}},\qquad \ut=\ep y\left(\o_{hom}(\al y+\be)+\o[\ep](\al y+\be)\right).
\]
We may express this as
\[
  \rt=\frac{k}{y^{\frac{2}{2-\g}}}+\ep \frac{p_{hom}(y)}{y^{\frac{2}{2-\g}}}+\ep \rho_{ext}[\ep](y),\qquad \ut=\ep y \o_{hom}(y)+\ep u_{ext}[\ep](y)
\]
where
\begin{align}
  \rho_{ext}[\ep](y)&:=\frac{p_{hom}(\al y+\be)-p_{hom}(y)}{y^{\frac{2}{2-\g}}}+\frac{p[\ep](\al y+\be)}{y^{\frac{2}{2-\g}}},\label{ext_change_variable_1}\\
  u_{ext}[\ep](y)&:=y(\o_{hom}(\al y+\be)-\o_{hom}(y))+y \o[\ep](\al y+\be)\label{ext_change_variable_2}.
\end{align}
Converting the lemma to these variables gives the following result.
\begin{lem}
  \label{ext_construction_lemma}
  For $y_0\leq 1$, $\ep\ll y_0^{\mu}$, there exists $\rt[\ep],\ut[\ep]\in C^1([y_0,\infty))$ which satisfy \eqref{selfsimsys} and may be decomposed as
  \[
    \begin{cases}
      \rt_{ext}[\ep](y)=\frac{k}{y^{\frac{2}{2-\g}}}+\frac{\ep}{y^{\frac{2}{2-\g}}}p_{hom}(y)+\ep \rho_{ext}[\ep](y)\\
      \ut_{ext}[\ep](y)=\ep y \o_{hom}(y)+\ep u_{ext}[\ep](y)
    \end{cases},
  \]
  where we have the bounds
  \begin{equation}
    \label{eq:ext_construction_bound_1}
    \sup_{y_0\leq y\leq 1}\abs{y^{\mu+\frac{2}{2-\g}} \rho_{ext}},\abs{y^{\mu-1} u_{ext}},\abs{y^{\mu+\frac{2}{2-\g}+1}\rho_{ext}'},\abs{y^{\mu}u_{ext}'}\lesssim \ep y_0^{-\mu}
  \end{equation}
  and
  \begin{equation}
    \label{eq:ext_construction_bound_2}
    \sup_{1\leq y\leq \infty}\abs{y^{\frac{2}{2-\g}} \rho_{ext}},\abs{y^{\frac{\g-1}{2-\g}}u_{ext}},\abs{y^{\frac{2}{2-\g}+1}\rho_{ext}'},\abs{y^{\frac{\g-1}{2-\g}+1}u_{ext}'}\lesssim \ep y_0^{-\mu}.
  \end{equation}
  Also we have that $\ep\mapsto \rho_{ext}[\ep](y_0),u_{ext}[\ep](y_0)$ is Lipschitz on a neighborhood of zero, with
  \begin{equation}
    \rho_{ext}[0](y_0)=u_{ext}[0](y_0)=0,
    \label{extremainder1}
  \end{equation}
  \begin{equation}
    \frac{\abs{\rho_{ext}[\ep_1](y_0)-\rho_{ext}[\ep_2](y_0)}}{\ep_1-\ep_2}\lesssim y_0^{-\frac{8-5\g}{2-\g}},\qquad \abs{\frac{u_{ext}[\ep_1](y_0)-u_{ext}[\ep_2](y_0)}{\ep_1-\ep_2}}\lesssim y_0^{\frac{4(\g-1)}{2-\g}}.
    \label{extremainder2}
  \end{equation}
  Finally, we have the monotonicities
  \begin{equation}
    \label{eq:ext_monotonicities}
    (u+(2-\g)y)'\geq \frac{1}{2}(2-\g),\qquad \rt'\leq \frac{1}{2}\frac{2}{2-\g}\frac{k}{y^{\frac{2}{2-\g}+1}}.
  \end{equation}
\end{lem}
\begin{proof}
  The bounds \eqref{eq:ext_construction_bound_1}-\eqref{extremainder2} follow from Lemma \ref{ext_contraction_lemma}, and then the definitions \eqref{ext_change_variable_1}, \eqref{ext_change_variable_2}.
  We then compute
  \begin{equation*}
    \ut'=\ep y \o_{hom}'+\ep \o_{hom}+\ep u',\qquad (\rt-\frac{k}{y^{\frac{2}{2-g}}})'=\ep \frac{p_{hom}'}{y^{\frac{2}{2-\g}}}-\frac{2}{2-\g}\frac{\ep p_{hom}}{y^{\frac{2}{2-\g}+1}}+\ep \rho'.
  \end{equation*}
  We then use the asymptotics for $p_{hom},\o_{hom}$ from Lemma \ref{ext_linear_lemma} to get
  \begin{equation*}
    \abs{\o_{hom}}\lesssim y^{-\mu}\ev{y}^{\mu-\frac{1}{2-\g}},\qquad \abs{y\o_{hom}'}\lesssim y^{-\mu}\ev{y}^{\mu-\frac{1}{2-\g}},
  \end{equation*}
  \begin{equation*}
    \abs{y^{-\frac{2}{2-\g}-1}p_{hom}}\lesssim y^{-\frac{2}{2-\g}-1-\mu}\ev{y}^{\mu},\qquad \abs{y^{-\frac{2}{2-\g}}p_{hom}'}\lesssim y^{-\frac{2}{2-\g}-1-\mu}\ev{y}^{\mu}.
  \end{equation*}
  We also note that by \eqref{eq:ext_construction_bound_1} and \eqref{eq:ext_construction_bound_2} we have
  \begin{equation*}
    \abs{u'}\lesssim \ep y_0^{-\mu}y^{-\mu}\ev{y}^{\mu-\frac{1}{2-\g}},\qquad \abs{\rho'}\lesssim \ep y_0^{-\mu}y^{-\frac{2}{2-\g}-1-\mu}\ev{y}^{\mu}.
  \end{equation*}
  Combining these gives
  \begin{equation*}
    \abs{\ut}'\lesssim \ep y^{-\mu}\ev{y}^{\mu-\frac{1}{2-\g}},\qquad \abs{y^{\frac{2}{2-\g}+1}(\rt-\frac{k}{y^{\frac{2}{2-\g}}})'}\lesssim \ep y^{-\mu}\ev{y}^{\mu}.
  \end{equation*}
  We then note that as $y\geq y_0$, we have that $\ev{y}\lesssim y y_0^{-1}$, so we have
  \begin{equation*}
    \abs{\ut}'\lesssim \ep y_0^{-\mu}\ev{y}^{-\frac{1}{2-\g}},\qquad \abs{y^{\frac{2}{2-\g}+1}(\rt-\frac{k}{y^{\frac{2}{2-\g}}})'}\lesssim \ep y_0^{-\mu}.
  \end{equation*}
  Since we may without loss of generality choose $\ep y_0^{-\mu}\ll 1$, we may assume that these are bounded by $\frac{1}{2}\frac{2}{2-\g}$ and $\frac{1}{2}\frac{2}{2-\g}k$, which proves \eqref{eq:ext_monotonicities}.
\end{proof}
\section{Solving in the interior region}
\label{section:int}
\subsection{Preliminaries on the Lane-Emden Equation}
We introduce the variable $\wt$ (self-similar enthalpy) by 
\[
  \rt=(\frac{\g-1}{\g}\wt)^{\frac{1}{\g-1}}.
\]
We then take the divergence of the second equation of \eqref{selfsimsys} and obtain
\begin{equation}
  \label{eq:enthalpyeq}   
  \begin{cases}
    \left(2+(2-\g)y\p_y\right)(\frac{\g-1}{\g}\wt)^{\frac{1}{\g-1}}+\div\left((\frac{\g-1}{\g}\wt)^{\frac{1}{\g-1}} \ut\right)=0\\
    \div\left(\left((\g-1)+(\ut+(2-\g) y)\p_y \right)\ut\right)+\Lap \wt+4\pi(\frac{\g-1}{\g})^{\frac{1}{\g-1}}\wt^{\frac{1}{\g-1}}=0\\
    \rt'(0)=0,\qquad \ut(0)=0,\qquad \ut'(0)=-\frac{2}{3}
  \end{cases}.
\end{equation}
We then recall that in this case the far-field solution is given by
\[
  Q_f(y)=\frac{\g}{\g-1}k^{\g-1}y^{-\frac{2(\g-1)}{2-\g}},\qquad (\frac{\g-1}{\g})^{\frac{1}{\g-1}}Q_f^{\frac{1}{\g-1}}=ky^{-\frac{2}{2-\g}}.
\]
The asymptotfics of the ground state (normalized by $Q(0)=\frac{\g}{\g-1}$) as $y\to \infty$ can then be seen to be
\[
  Q=\frac{\g}{\g-1}k^{\g-1}y^{-\frac{2(\g-1)}{2-\g}}(1+O(y^{-\mu})).
\]
In particular, we have the global pointwise bounds
\begin{equation}
  \label{eq:Qpointwise}
  \ev{y}^{-\frac{2(\g-1)}{2-\g}}\lesssim Q\lesssim \ev{y}^{-\frac{2(\g-1)}{2-\g}}.
\end{equation}
At the level of the mass-density, we have
\[
  (\frac{\g-1}{\g}Q)^{\frac{1}{\g-1}}=\frac{k}{y^{\frac{2}{2-\g}}}(1+O(y^{-\mu})).
\]
We now define
\[
  H:=-(\Lap+\frac{4\pi}{\g-1}(\frac{\g-1}{\g})^{\frac{1}{\g-1}}Q^{\frac{2-\g}{\g-1}}).
\]
We may then compute the leading asymptotics of the potential to be
\begin{equation}
  \label{eq:potentialasymptot}
  (\frac{4\pi}{\g-1}(\frac{\g-1}{\g})^{\frac{1}{\g-1}}Q^{\frac{2-\g}{\g-1}})=\frac{2(4-3\g)}{(2-\g)^2}\frac{1}{y^2}+O(\frac{1}{y^{2+\mu}})
\end{equation}
From the scaling symmetry of the Lane Emden equation \eqref{laneemdeneq} we see that one element of the kernel of $H$ is
\[
  v_1:=y\p_y Q+\frac{2(\g-1)}{2-\g}Q.
\]
We then define $v_2$ by imposing that $v_1v_2'-v_1'v_2=y^{-2}$, which gives
\begin{equation}
  \label{eq:v2def}
  v_2:=-v_1(\int_{y}^{Y_0} \frac{1}{(v_1)^2(y')^2}dy').
\end{equation}
We note that $v_1,v_2=O(\frac{1}{y^{\frac12}})$ as $y\to \infty$. The following lemma describes the next order asymptotics.
\begin{lem}
  \label{int_asymptotics_lemma}
  As $y\to \infty$, the functions $v_1,v_2$ satisfy, for some $c_3,c_4>0$ and $d_3,d_4\in \R/2\pi \Z$,
  \[
    (v_1,v_2)=(c_3 \frac{\sin(\nu\log(y)+d_3)}{y^{\frac12}},c_4 \frac{\sin(\nu\log(y)+d_4)}{y^{\frac12}})+O(\frac{1}{y^{\frac{1}{2}+\mu}})
  \]
  and one can differentiate these asymptotics arbitrarily many times. Moreover, although $v_2=O(\frac{1}{y})$ as $y\to 0$, $v_1,yv_2$ are smooth up to $y=0$. As a consequence, we have that on $[0,\infty)$, we have the pointwise bounds that for $k\geq 0$
  \[
    \abs{\p_y^{k}v_1}\lesssim \ev{y}^{-\frac{1}{2}-k},\qquad \abs{v_2}\lesssim \frac{1}{y^{1+k}}\ev{y}^{\frac{1}{2}}.
  \]
\end{lem}
This lemma is proved in Subsection \ref{int_asymptotics_proof}. It then follows that for some $c_2>0$ and $d_2\in \R/2\pi \Z$ we have that as $y\to \infty$,
\[
  Q=\frac{\g}{\g-1}k^{\g-1}\frac{1}{y^{\frac{2(\g-1)}{2-\g}}}+\frac{\g}{k^{2-\g}}\frac{c_2 \sin(\nu \log(y)+d_2)}{y^{\frac12}}+O(\frac{1}{y^{\frac{1}{2}+\mu}})
\]
At the level of the mass density we obtain that as $y\to \infty$
\begin{equation}
  (\frac{\g-1}{\g}Q)^{\frac{1}{\g-1}}=\frac{k}{y^{\frac{2}{2-\g}}}+c_2\frac{\sin(\nu \log(y)+d_2)}{y^{\frac52}}+O(\frac{1}{y^{\frac{8-5\g}{2-\g}}}).
  \label{laneemdenasymptotics}
\end{equation}
\subsection{Motivating asymptotic expansion}
We now assume $\wt(0)=\frac{\g}{\g-1}\la^{-\frac{2(\g-1)}{2-\g}}$ and perform the scaling
\[
  \wt=\la^{-\frac{2(\g-1)}{2-\g}}\wb(\cdot/\la),\qquad \ut=\la \ub(\cdot/\la).
\]
We then substitute this into \eqref{eq:enthalpyeq} to obtain
\begin{equation}
  \begin{cases}
    (2+(2-\g)y\p_y)(\frac{\g-1}{\g}\wb)^{\frac{1}{\g-1}}+\div\left((\frac{\g-1}{\g}\wb)^{\frac{1}{\g-1}}\ub \right)=0\\
    \Delta \wb+4\pi(\frac{\g-1}{\g})^{\frac{1}{\g-1}}\wb^{\frac{1}{\g-1}}=-\la^{\frac{2}{2-\g}}\div\left(\left((\g-1)+(\ub+(2-\g) y)\p_y \right)\ub\right)\\
    \wb(0)=\frac{\g}{\g-1},\qquad \wb'(0)=0,\qquad \ub(0)=0,\qquad \ub'(0)=-\frac{2}{3}.
  \end{cases}.
\end{equation}
In this case one can consider doing a formal expansion in powers of $\lambda^{\frac{2}{2-\g}}$. The zeroth order contribution for $u$ is the equation
\[
  \begin{cases}
    (2+(2-\g)y\p_y)(\frac{\g-1}{\g} Q)^{\frac{1}{\g-1}}+J u_0=0\\
    u_0(0)=0,\qquad u_0'(0)=-\frac{2}{3}
  \end{cases}.
\]
where
\[
  Ju:=\div\left((\frac{\g-1}{\g}Q)^{\frac{1}{\g-1}} u\right).
\]
We then recall that the system
\[
  \begin{cases}
    Ju+f=0\\
    u(0)=0,\qquad u'(0)=-\frac{f(0)}{3}
  \end{cases}
\]
is solved by $u=T(f)$, where
\[
  T(f):=-\frac{1}{y^2(\frac{\g-1}{\g}Q)^{\frac{1}{\g-1}}}\int_0^y f(y')(y')^2 dy'.
\]
One then defines
\[
  u_*:=T((2+(2-\g)y\p_y)(\frac{\g-1}{\g}Q)^{\frac{1}{\g-1}}).
\]
We then recall that from \eqref{ustarasymptot} that as $y\to \infty$
\[
  u_*=k^{-1}(2-\g)^{\frac32}c_2\sin\left(\nu\log(y+d_2+\theta_0)\right)y^{1-\mu}+O(y^{\frac{4(\g-1)}{2-\g}}).
\]
In particular, $u_*$ obeys the global pointwise bound
\begin{equation}
  \label{eq:ustarpointwise}
  \abs{u_*}\lesssim \abs{y}\ev{y}^{-\mu}.
\end{equation}
As a comment, in this case, the next order equations one obtains are
\[
  \begin{cases}
    \div((\frac{\g}{\g-1} Q)^{\frac{1}{\g-1}} u_1)+Lw_1=0\\
    -Hw_1+\div(((\g-1)+(u_*+(2-\g)y)\p_y)u_*)=0\\
    w_1(0)=w_1'(0)=u_1(0)=u_1'(0)=0
  \end{cases},
\]
where
\[
  Lw:=(2+(2-\g)y\p_y)(\frac{\g-1}{\g})^{\frac{1}{\g-1}}\frac{1}{\g-1}(Q^{\frac{2-\g}{\g-1}}w)+\div((\frac{\g-1}{\g})^{\frac{1}{\g-1}}\frac{1}{\g-1} Q^{\frac{2-\g}{\g-1}} u_* w).
\]
We then note that since $S$ will again multiply expansions by $y^2$, we will get that
\[
  w_1=O(y^{2-\mu}).
\]
Since $T$ will multiply expansions by $y^{1+\frac{2}{2-\g}}$, we will get that
\[
  u_1=O(y^{1+\frac{2}{2-\g}-\mu}).
\]
These growth rates motivate the function spaces, described in the next section, in which the iteration will take place.
\subsection{Difference equations and iteration scheme}
As a comment, in contrast to the previous section, in this section we will not rescale space and all the functions will be defined on $[0,y_0]$. This is to make more clear the continuity of the solutions with respect to changing $\la$. We now define
\[
  Q_{\la}=\la^{-\frac{2(\g-1)}{2-\g}}Q(\cdot/\la),\qquad u_{\la}=\la u(\cdot/\la),\qquad v_{i,\la}=v_i(\cdot/\la).
\]
We now show the existence of solutions with $\rt(0)=\la^{-\frac{2}{2-\g}}$ (equivalently $\wt(0)=\frac{\g}{\g-1}\la^{-\frac{2(\g-1)}{2-\g}}$). We define difference variables $(w,u)$ implicitly by
\[
  \wt(y)=Q_{\la}(y)+\la^{-\frac{2(\g-1)}{2-\g}+\frac{2}{2-\g}}w(y),\qquad \ut(y)=u_{\la}(y)+\la^{1+\frac{2}{2-\g}} u(y).
\]
We then have that \eqref{eq:enthalpyeq} with the boundary condition $\wt(0)=\frac{\g}{\g-1}\la^{-\frac{2(\g-1)}{2-\g}}$ is equivalent to the following system on $[0,y_0]$:
\[
  \begin{cases}
    (2+(2-\g)y\p_y)(\frac{\g-1}{\g}(Q_{\la}+\la^{-\frac{2(\g-1)}{2-\g}+\frac{2}{2-\g}} w))^{\frac{1}{\g-1}}\\
    +\div\left((\frac{\g-1}{\g}(Q_{\la}+\la^{-\frac{2(\g-1)}{2-\g}+\frac{2}{2-\g}} w))^{\frac{1}{\g-1}} (u_\la+\la^{1+\frac{2}{2-\g}} u)\right)=0\\
    \div\left(\left((\g-1)+(u_\la+\la^{1+\frac{2}{2-\g}} u+(2-\g)y)\p_y\right)(u_\la+\la^{1+\frac{2}{2-\g}}u)\right)\\
    +\Lap(Q_\la+\la^{-\frac{2(\g-1)}{2-\g}+\frac{2}{2-\g}}w)+4\pi (\frac{\g-1}{\g})^{\frac{1}{\g-1}}(Q_\la+\la^{-\frac{2(\g-1)}{2-\g}+\frac{2}{2-\g}} w)^{\frac{1}{\g-1}}=0\\
    w(0)=w'(0)=u(0)=u'(0)=0
  \end{cases}.
\]
We then define
\begin{align*}
  J_\la u&:=\la^{1+\frac{2}{2-\g}}\div((\frac{\g-1}{\g}Q_\la)^{\frac{1}{\g-1}} u),\\
  L_\la w&:=\la^{-\frac{2(\g-1)}{2-\g}+\frac{2}{2-\g}}\left((2+(2-\g)y\p_y)(\frac{\g-1}{\g})^{\frac{1}{\g-1}}\frac{1}{\g-1}(Q_\la^{\frac{2-\g}{\g-1}}w)+\div((\frac{\g-1}{\g})^{\frac{1}{\g-1}}\frac{1}{\g-1} Q_\la^{\frac{2-\g}{\g-1}} u_\la w)\right),\\
  H_\la&:=-\la^{\frac{-2(\g-1)}{2-\g}+\frac{2}{2-\g}}(\Lap+\frac{4\pi}{\g-1}(\frac{\g-1}{\g})^{\frac{1}{\g-1}}Q_\la^{\frac{2-\g}{\g-1}}),\\
  S_\la f&:=\frac{1}{\la^{1-\frac{2(\g-1)}{2-\g}+\frac{2}{2-\g}}}\left(\left(\int_0^y f v_{2,\la}(y')^2 dy'\right)v_{1,\la}-\left(\int_0^y f v_{1,\la}(y')^2 dy'\right)v_{2,\la}\right),\\
  T_\la f&:=-\frac{1}{\la^{1+\frac{2}{2-\g}} y^2(\frac{\g-1}{\g}Q_\la)^{\frac{1}{\g-1}}}\int_0^y f(y')(y')^2 dy',\\
  K_\la&:=T_\la\circ L_\la\circ S_\la.
\end{align*}
We note that these operators satisfy
\[
  H_\la S_\la=\id,\qquad -J_\la T_\la = \id,\qquad \frac{1}{\la}J_\la u_\la+\la^{\frac{2}{2-\g}}(2+(2-\g)y\p_y)(\frac{\g-1}{\g}Q_\la)^{\frac{1}{\g-1}}=0.
\]
We will use the expansion
\begin{align*}
  &(\frac{\g-1}{\g}(Q_\la+\la^{-\frac{2(\g-1)}{2-\g}+\frac{2}{2-\g}} w))^{\frac{1}{\g-1}}\\
  &=(\frac{\g-1}{\g}Q_\la)^{\frac{1}{\g-1}}+\la^{-\frac{2(\g-1)}{2-\g}+\frac{2}{2-\g}} w(\frac{\g-1}{\g})^{\frac{1}{\g-1}}\frac{1}{\g-1}\int_0^{1}(Q_\la+\la^{-\frac{2(\g-1)}{2-\g}+\frac{2}{2-\g}} w t)^{\frac{2-\g}{\g-1}}dt\\
  &=(\frac{\g-1}{\g}Q_\la)^{\frac{1}{\g-1}}+\la^{-\frac{2(\g-1)}{2-\g}+\frac{2}{2-\g}}(\frac{\g-1}{\g})^{\frac{1}{\g-1}}w\frac{1}{\g-1}Q_\la^{\frac{2-\g}{\g-1}}\\
  &\qquad +\la^{-\frac{4(\g-1)}{2-\g}+\frac{4}{2-\g}}w^2(\frac{\g-1}{\g})^{\frac{1}{\g-1}}\frac{2-\g}{(\g-1)^2}\int_0^1 (1-t)(Q_\la+\la^{-\frac{2(\g-1)}{2-\g}+\frac{2}{2-\g}} w t)^{\frac{3-2\g}{\g-1}}dt
\end{align*}
to define
\begin{align*}
  F_{\la}(w,u)&=\la^{-\frac{4(\g-1)}{2-\g}+\frac{4}{2-\g}}(2+(2-\g)y\p_y)(w^2(\frac{\g-1}{\g})^{\frac{1}{\g-1}}\frac{2-\g}{(\g-1)^2}\int_0^1(1-t)(Q_\la+t\la^{-\frac{2(\g-1)}{2-\g}+\frac{2}{2-\g}} w)^{\frac{3-2\g}{\g-1}}dt)\\
              &\qquad+\la^{1-\frac{2(\g-1)}{2-\g}+\frac{4}{2-\g}}\div\left((\frac{\g-1}{\g})^{\frac{1}{\g-1}} \frac{1}{\g-1} u w\int_0^1(Q_\la+\la^{-\frac{2(\g-1)}{2-\g}+\frac{2}{2-\g}} w t)^{\frac{2-\g}{\g-1}}dt\right)\\
              &\qquad+\la^{-\frac{4(\g-1)}{2-\g}+\frac{4}{2-\g}}\div\left(u_\la w^2(\frac{\g-1}{\g})^{\frac{1}{\g-1}}\frac{2-\g}{(\g-1)^2}\int_0^1 (1-t) (Q_\la+\la^{-\frac{2(\g-1)}{2-\g}+\frac{2}{2-\g}} w t)^{\frac{3-2\g}{\g-1}}dt\right),\\
  G_{\la,1}(w,u)&=\div\left(\left((\g-1)+(u_\la+(2-\g)y)\p_y\right)u_\la\right)\\
              &\qquad +\la^{1+\frac{2}{2-\g}}\div\left(u\p_y u_\la+\left((\g-1)+(u_\la+(2-\g)y)\p_y\right)u+\la^{1+\frac{2}{2-\g}}u\p_y u\right),\\
  G_{\la,2}(w,u)&=4\pi \la^{-\frac{4(\g-1)}{2-\g}+\frac{4}{2-\g}}w^2(\frac{\g-1}{\g})^{\frac{1}{\g-1}}\frac{2-\g}{(\g-1)^2}\int_0^1 (1-t)(Q_\la+\la^{-\frac{2(\g-1)}{2-\g}+\frac{2}{2-\g}} w t)^{\frac{3-2\g}{\g-1}}dt,\\
  G_\la(w,u)&=G_{\la,1}(w,u)+G_{\la,2}(w,u).
\end{align*}
we then obtain the system
\[
  \begin{cases}
    0=L_\la w+J_\la u+F_\la(w,u)\\
    0=-H_\la w+G_\la(w,u)\\
    w(0)=w'(0)=u(0)=u'(0)=0
  \end{cases}
\]
which we can solve using the linear solution operators as
\begin{equation}
  \label{eq:int_integral_formulation}
  \begin{cases}
    u=T_\la(F_\la(w,u))+K_\la(G_\la(w,u))\\
    w=S_\la(G_\la(w,u))
  \end{cases}.
\end{equation}
We will now define the spaces in which our iteration will take place in and the norms which we will use.
\begin{defin}We have the following spaces
  \begin{align*}
    Z&:=\{f\in C^2([0,y_0])\colon f(0)=f'(0)=0\},\\
    Y&:=\{f\in C^2([0,y_0])\colon f(0)=f'(0)=0\},\\
    \calG&:=C^{0}([0,y_0]),\\
    \calF&:=\{f\in C^1([0,y_0]\colon f(0)=0)\}.
  \end{align*}
  On each of these spaces, we define a $\la$-parameterized set of norms as follows.
  \begin{align*}
    \norm{f}_{Z_\la}&:=\sup_{y\in [0,y_0]}\ev{\frac{y}{\la}}^{\mu}(\frac{\la^2}{y^2}\abs{f}+\frac{\la}{y}\la \abs{\p_y f}+\la^2\abs{\p_y^2f}),\\
    \norm{f}_{Y_\la}&:=\sup_{y\in [0,y_0]}\ev{\frac{y}{\la}}^{1+\mu-\frac{2}{2-\g}}(\frac{\la^2}{y^2}\abs{f}+\frac{\la}{y}\la \abs{\p_y f}+\la^2\abs{\p_y^2 f}),\\
    \norm{f}_{\calG_\la}&:=\sup_{y\in [0,y_0]}\ev{\frac{y}{\la}}^{\mu}\abs{f},\\
    \norm{f}_{\calF_\la}&:=\sup_{y\in [0,y_0]}\ev{\frac{y}{\la}}^{1+\mu}(\frac{\la}{y}\abs{f}+\la \abs{\p_y f}).
  \end{align*}
\end{defin}
As a comment, to read off these exponents, one can first motivate the growth rates of $Z_{y_1},Y_{y_1}$ by considering $w_1,u_1$ in the previous section. From these, one can get the rates for $\calG_{y_1},\calF_{y_1}$ by requiring that $S,T$ be bounded.

We have the following lemma regarding the linearized operators.
\begin{lem}
  \label{int_linear_lemma}
  If $y_0\ll 1$, $\la_0\leq \frac{1}{10}y_0$, we have that if $G,F\in \calG,\calF$, then $T_\la(F),K_\la(G)\in Y$, and $S_\la(G)\in Z$. Moreover, we have that there exists a constant $C_1$ independent of $y_0,\la$ such that
  \begin{align*}   
    \norm{T_\la(F)}_{Y_\la}&\leq C_1 \norm{F}_{\calF_{\la}},\\
    \norm{S_\la(G)}_{Z_\la}&\leq C_1 \norm{G}_{\calG_{\la}},\\
    \norm{K_\la(G)}_{Y_\la}&\leq C_1 \norm{G}_{\calG_{\la}}.\\
  \end{align*}
  Moreover, these family of operators have a Lipschitz dependence on $\la$, so that in particular if $\la_1<\la_2<\frac{y_0}{10}$, then we have
  \begin{align*}
    \max_i \norm{T_{\la_2}(F)-T_{\la_1}(F)}_{Y_{\la_i}}&\leq C_1 \frac{\la_2-\la_1}{\la_1}\min_i \norm{F}_{\calF_{\la_i}},\\
    \max_i \norm{S_{\la_2}(G)-S_{\la_1}(G)}_{Z_{\la_i}}&\leq C_1 \frac{\la_2-\la_1}{\la_1}\min_i \norm{G}_{\calG_{\la_i}},\\
    \max_i \norm{K_{\la_2}(G)-K_{\la_1}(G)}_{Y_{\la_i}}&\leq C_1 \frac{\la_2-\la_1}{\la_1}\min_i \norm{G}_{\calG_{\la_i}}.
  \end{align*}
\end{lem}
For the proof see Subsection \ref{int_linear_proof}

The analog of the nonlinear bound is the following.
\begin{lem}
  \label{int_nonlinear_lemma}
  There exists $C_2>0$ (independent of $y_0,\la$) such that if $u,w\in Y, Z$ respectively and if $\norm{u}_{Y_\la}+\norm{w}_{Z_\la}\leq C$ then if $C\cdot y_0^{\frac{2}{2-\g}}\ll 1$, $y_0^{\frac{2}{2-\g}}C^2\leq 1$ and $\la\leq \frac{y_0}{10}$, then we have $F_\la(w,u),G_{\la}(w,u)$ are well-defined, lie in $\calF$ and $\calG$ respectively, and satisfy
  \begin{equation*}
    \norm{F_\la(w,u)}_{\calF_\la}+\norm{G_{\la}(w,u)}_{\calG_\la}\leq C_2.
  \end{equation*}
  Moreover we have the Lipschitz estimate that
  \begin{multline}
    \label{eq:int_lipschitz_ineq}
    \norm{F_\la(w_2,u_2)-F_{\la}(w_1,u_1)}_{\calF_\la}+\norm{G_{\la}(w_2,u_2)-G_{\la}(w_1,u_1)}_{\calG_\la}\leq C_2 \cdot (1+C) y_0^{\frac{2}{2-\g}}(\norm{u_2-u_1}_{Y_\la}+\norm{w_2-w_1}_{Z_\la}).
  \end{multline}
  Finally, we have the following estimate on how these operators depend on $\la$, for $0<\la_1<\la_2<\frac{y_0}{10}$,
  \begin{equation*}
    \max_i\norm{F_{\la_2}(w,u)-F_{\la_1}(w,u)}_{\calF_{\la_i}}+\max_i\norm{G_{\la_2}(w,u)-G_{\la_1}(w,u)}_{\calG_{\la_i}}\leq C_2\frac{\la_2-\la_1}{\la_1}.
  \end{equation*}
\end{lem}
For the proof see Subsection \ref{int_nonlinear_proof}.

\begin{lem}
  \label{int_contraction_lemma}
  There exists a constant $C_3>0$ (independent of $y_0,\la$) such that for $y_0^{\frac{2}{2-\g}}\ll 1$, $\la\leq \frac{y_0}{10}$, there exists a unique $C^2$ solution to \eqref{eq:enthalpyeq} with $\wt>0$ and the property that it can be decomposed as
  \begin{equation*}
    \wt(y)=Q_\la(y)+\la^{-\frac{2(\g-1)}{2-\g}+\frac{2}{2-\g}}w[\la](y),\qquad \ut=u_\la+\la^{1+\frac{2}{2-\g}}u[\la](y)
  \end{equation*}
  with $\norm{w}_{Z_\la}+\norm{u}_{Y_\la}\leq \frac{C_3}{100}$. Moreover, if $0<\la_1<\la_2<\frac{y_0}{10}$, then we have
  \begin{equation}
    \label{eq:contraction_lipschitz}
    \max_i\norm{w[\la_2]-w[\la_1]}_{Z_\la}+\norm{u[\la_2]-u[\la_1]}_{Y_\la}\leq \frac{\la_2-\la_1}{\la_1}C_3.
  \end{equation}
\end{lem}
For the proof see Subsection \ref{int_contraction_proof}

As a comment, the analog of Lemma 3.6 is then the following.
\begin{lem}
  \label{int_construction_lemma}
  If $y_0^{\frac{2}{2-\g}}\ll 1$ is fixed, then for $\la\leq \frac{y_0}{10}$ there exists $C^2$ functions $\rt_{int},\ut_{int}$ defined on $[0,y_0]$ which solve \eqref{selfsimsys} of the form
  \[
    \begin{cases}
      \rt_{int}[\lambda](y)=(\frac{\g-1}{\g}Q_{\la})^{\frac{1}{\g-1}}+\rho_{int}[\la](y)\\
      \ut_{int}[\lambda](y)=u_{\la}+\la^{1+\frac{2}{2-\g}} u_{int}[\la](y)\\
      \rho_{int}(0)=\rho_{int}'(0)=u_{int}(0)=u_{int}'(0)=0
    \end{cases}
  \]
  and they obey the following bounds on $[0,y_0]$,
  \begin{equation}
    \label{eq:int_construction_bound_1}
    \abs{\rho_{int}}\lesssim (\frac{y}{\la})^{2}\ev{\frac{y}{\la}}^{-2-\mu},\qquad \abs{u_{int}}\lesssim \frac{y^2}{\la^2} \ev{\frac{y}{\la}}^{\frac{2}{2-\g}-1-\mu}.
  \end{equation}
  \begin{equation}
    \label{eq:int_construction_bound_2}
    \abs{\rho_{int}'}\lesssim \frac{y}{\la^2}\ev{\frac{y}{\la}}^{-2-\mu},\qquad \abs{u_{int}'}\lesssim \frac{y}{\la^2} \ev{\frac{y}{\la}}^{\frac{2}{2-\g}-1-\mu}.
  \end{equation}
  Moreover, $\la\to 0$, $(\rho_{int}[\la](y_0),u_{int}[\la](y_0))$ satisfies the following on $(0,\frac{1}{10}y_0)$
  \begin{equation}
    \lim_{\la\to 0}\rho_{int}[\la]=0,\qquad \lim_{\la\to 0}(\frac{\la}{y_0})^{1+\frac{2}{2-\g}}u_{int}[\la]=0
    \label{intremainder1}
  \end{equation}
  as well as the following Lipchitz estimates with respect to $\la$ for $0<\la_1<\la_2<\frac{y_0}{10}$:
  \begin{equation}
    \label{eq:lambdalipschitz}
    \begin{cases}
      \abs{\rho_{int}[\la_2](y_0)-\rho_{int}[\la_1](y_0)}&\lesssim \frac{\la_2-\la_1}{\la_1}(\frac{\la_2}{y_0})^{\mu}\\
      \abs{(\frac{\la_2}{y_0})^{1+\frac{2}{2-\g}}u_{int}[\la_2](y_0)-(\frac{\la_1}{y_0})^{1+\frac{2}{2-\g}}u_{int}[\la_1](y_0)}&\lesssim \frac{\la_2-\la_1}{\la_1}(\frac{\la_2}{y_0})^{\mu}
    \end{cases}.
  \end{equation}
  We also have the following bounds on $[0,y_0]$, for some absolute constants independent of $y_0,\la$,
  \begin{equation}
    \label{eq:int_sonic_bounds}
    \abs{\ut+(2-\g)y}\lesssim y_0,\qquad \rt\gtrsim y_0^{-\frac{2}{\g-1}}.
  \end{equation}
\end{lem}
\begin{proof}
  We begin with Lemma \ref{int_contraction_lemma} and then define
  \begin{equation*}
    \rho_{int}[\la]:=(\frac{\g-1}{\g})^{\frac{1}{\g-1}}\frac{1}{\g-1}\la^{\frac{2}{2-\g}-\frac{2(\g-1)}{2-\g}}w[\la]\cdot\int_0^1 \left(Q_\la+t \la^{\frac{2}{2-\g}-\frac{2(\g-1)}{2-\g}} w[\la]\right)^{\frac{2-\g}{\g-1}}dt.
  \end{equation*}
  The estimates \eqref{eq:int_construction_bound_1}-\eqref{eq:lambdalipschitz} then follow directly. We may then use
  \begin{equation*}
    \abs{u_\la}\lesssim \abs{y}\leq y_0,\qquad \abs{Q_\la}\gtrsim \la^{-\frac{2(\g-1)}{2-\g}}\ev{\frac{y}{\la}}^{-\frac{2(\g-1)}{2-\g}}\geq 2^{\frac{-2(\g-1)}{2-\g}}y_0^{\frac{-2(\g-1)}{2-\g}}.
  \end{equation*}
  Moreover, from \eqref{eq:int_construction_bound_1} we have
  \begin{equation*}
    \abs{\la^{1+\frac{2}{2-\g}}u_{int}}\lesssim \abs{y}\cdot \ev{\frac{y}{\la}}^{-\mu}\leq \abs{y_0},\qquad \abs{\rho_{int}}\lesssim \ev{\frac{y}{\la}}^{-\mu}\leq 1.
  \end{equation*}
  As we impose $y_0^{\frac{2}{2-\g}}\ll 1$, applying the triangle inequality yields \eqref{eq:int_sonic_bounds}.
\end{proof}
\section{Proof of the theorem}
\label{sec:matching}
\subsection{Matching and existence of $C^1$ polytropic Hunter solutions}
\label{sec:matching_subsection}
In this section the following lemma is proved.
\begin{lem}
  \label{lem:C1_result}
  For some integer $N\gg 1$ there exists $C^1$ solutions to \eqref{polytropicselfsimeq} on $[0,\infty)$ subject to the boundary conditions \eqref{eq:selfsim_boundary_conds} and also $\rt_i>0$. In addition, $\rt_i$ intersects the far-field solution, $\rt_f$, exactly $i+1$ times. Moreover, the sonic point condition $(\ut+(2-\g)y)^2-\g\rt^{\g-1}=0$ is satisfied at exactly one point $y=y_*\in (0,2y_f)$, where $y_f$ is the sonic point for the far-field solution, \eqref{eq:yf_def}.
\end{lem}
\begin{proof}
  We follow the same eight steps as in \cite{sandine}. As a comment, it is easier technically to use a variant of the implicit function theorem which only requires (quantitative) Lipschitz regularity. This variant is non-standard, and so the statement and (elementary) proof are provided in Subsection \ref{implicit_proof}.

  \begin{enumerate}[label=\arabic*)]
  \item
    We pick $y_0\ll 1$ (dictated by the assumptions of Lemma \ref{int_construction_lemma}). Without loss of generality, we may require that \eqref{eq:wigglecond} is satisfied. Then, we pick $\ep_0\ll y_0^{\mu}$ (dictated by the assumptions of Lemma \ref{ext_construction_lemma}). For $\ep\in [0,\ep_0]$ we apply Lemma \ref{ext_construction_lemma} to obtain a family of solutions $(\rt_{ext}[\ep],\ut_{ext}[\ep])$ on $[y_0,\infty)$. Then, for $\la\leq \frac{y_0}{10}$, we apply Lemma \ref{int_construction_lemma} to obtain a family of solutions $(\rt_{int}[\la],\ut_{int}[\la])$ defined on $[0,y_0]$.
  \item
    We define
    \begin{align*}
      \calF[y_0](\ep,\la):&=\rt_{ext}[\ep](y_0)-\rt_{int}[\la](y_0)\\
                          &=\left(\frac{k}{y_0^{\frac{2}{2-\g}}}+\frac{\ep}{y_0^{\frac{2}{2-\g}}}p_{hom}(y_0)-(\frac{\g-1}{\g}Q_{\la})^{\frac{1}{\g-1}}\right)+(\ep \rho_{ext}[\ep](y_0)-\rho_{int}[\la](y_0)).
    \end{align*}
    We let $q(y)=(\frac{\g-1}{\g}Q)^{\frac{1}{\g-1}}-\frac{k}{y^{\frac{2}{2-\g}}}-c_2\frac{\sin(\nu \log(y)+d_2)}{y^{\frac{5}{2}}}$, and note that by \eqref{laneemdenasymptotics} we have
    \[
      q(y)=O(\frac{1}{y^{\frac{2}{2-\g}+2\mu}}).
    \]
    From \eqref{hom0asymptots}, \eqref{extremainder1}, \eqref{extremainder2}, \eqref{intremainder1}, \eqref{eq:lambdalipschitz}, we see that we may express
    \[
      \calF[y_0](\ep,\la)=\frac{\ep}{y_0^{\frac52}}c_1\sin(\nu\log(y_0)+d_1)-c_2\frac{\la^{\frac{6-5\g}{2-\g}}}{y_0^{\frac52}}\sin(\nu \log(y_0/\la)+d_2)+
      \frac{1}{y_0^{\frac52}}q_1(\ep)+\frac{1}{y_0^{\frac52}}q_2(\la)
    \]
    where
    \[
      q_{1}(0)=0,\qquad \frac{q_1(\ep_1)-q_{1}(\ep_2)}{\ep_1-\ep_2}=O(y_0^{\frac{2}{2-\g}})+O(\ep y_0^{-\mu}).
    \]
    \[
      \lim_{\la\to 0^+}q_2(\la)=0,\qquad \frac{q_2(\la_1)-q_2(\la_2)}{\la_1^{\frac{6-5\g}{2(2-\g)}}-\la_2^{\frac{6-5\g}{2(2-\g)}}}=O(y_0^{\frac{2}{2-\g}})+O(\la^{\mu}y_0^{-\mu}).
    \]
    We then recall \eqref{eq:wigglecond} which gives
    \begin{equation}
      \calF_{y_0}=y_0^{-\frac52}\left(\ep c_1-\la^{\mu}c_2 \cos(\nu \log(\la)+d_1-d_2)+q_1(\ep)+q_{2}(\la)\right),
      \label{firstmatchingasymptots}
    \end{equation}
    and note that if we consider the regime $\ep,\la^{\mu}\lesssim y_0^{\frac{2}{2-\g}+\mu}$ then
    \begin{equation}
      \abs{\frac{q_1(\ep_1)-q_2(\ep_2)}{\ep_1-\ep_2}}\lesssim y_0^{\frac{2}{2-\g}},\qquad \abs{\frac{q_2(\la_1^{\mu})-q_2(\la_2^{\mu})}{\la_1^{\mu}-\la_2^{\mu}}}\lesssim y_0^{\frac{2}{2-\g}}.
      \label{remainderlipschitz}
    \end{equation}
    We then note that if $y_0^{\frac{2}{2-\g}}\ll 1$, then by Lemma \ref{lem:implicit_function_theorem}, we have a well-defined function $\ep(\la)$, Lipschitz in $\la^{\mu}$, which satisfies $\lim_{\la\to 0^+}\ep(\la)=0$ is locally uniquely determined for $\la>0$ by
    \[
      \calF[y_0](\ep(\la),\la)=0.
    \]
  \item
    We observe that from \eqref{firstmatchingasymptots} and \eqref{remainderlipschitz} we may read off
    \[
      \ep(\la)=\frac{c_2}{c_1}\la^{\mu}\left(\cos(\nu \log(\la)+d_1-d_2)+O(y_0^{\frac{2}{2-\g}})\right).
    \]
  \item
    We now define the second matching function to be
    \[
      \calG[y_0](\la):=\frac{1}{c_2 y_0^{\frac{3\g-2}{2(2-\g)}} \la^{\mu}\sin(\th_0)}\frac{k}{(2-\g)^{\frac32}}\left(\ut_{ext}[\ep(\la)](y_0)-\ut_{int}[\la](y_0)\right)
    \]
    This will be a Lipschitz function of $\la^{\mu}$. 
    We then may compute that (by a similar computation as for the first matching function), using \eqref{hom0asymptots}, \eqref{extremainder1}, \eqref{extremainder2}, \eqref{intremainder1}, \eqref{eq:lambdalipschitz}, 
    \[
      \calG[y_0](\la)=-\sin(\nu \log(\la)+d_1-d_2)+O(y_0^{\frac{2}{2-\g}}).
    \]
  \item
    We then pick
    \begin{equation*}
      \la_{i,\pm}=\exp(\frac{1}{-\nu}(\pm \frac{1}{10}-d_1+d_2+k\pi))
    \end{equation*}
    and we have that if $i$ is sufficiently large so that $\la_{i,\pm}\leq \la_0$ $\la_{i,\pm}^{\mu}\lesssim y_0^{\frac{2}{2-\g}+\mu}$ then $\calG[y_0](\la_{i,\pm})$ is well-defined and
    \begin{equation*}
      \calG[y_0](\la_{i,\pm})=(-1)^k\sin(\pm \frac{1}{10})+O(y_0^{\frac{2}{2-\g}}).
    \end{equation*}
    We may choose $y_0^{\frac{2}{2-\g}}\ll 1$ if necessary to ensure that the remainder is bounded in magnitude by $\frac{1}{2}\sin(\frac{1}{10})$. Thus, for each $i>N$ for some $N$, we have that by the intermediate value theorem there exists $\la_i\in (\la_{i,+},\la_{i,-})$ such that $\calG[y_0](\la_i)=0$. Thus, we may glue together the interior and exterior solutions to get a continuous solution on $[0,\infty)$. Moreover, by \eqref{eq:int_sonic_bounds}, we see that the determinant
    \begin{equation*}
      D(y):=(u+(2-\g)y)^2-\g\rt^{\g-1}
    \end{equation*}
    is negative at $y_0$ (the negative contribution is $O(y_0^{\frac{-2(\g-1)}{2-\g}})$ while the positive contribution is $O(y_0)$). Thus, $D(y)$ does not vanish at the sonic point, so the ODE system is regular, and so the solutions $(\rho_{int},u_{int})$ and $(\rho_{ext},u_{ext})$ can be glued at $y_0$ to obtain a classical solution on $[0,\infty)$. We note that as previously mentioned, on $[0,y_0]$, we have that $D(y)$ is negative. Moreover, we see from \eqref{eq:ext_monotonicities} that $D(y)$ is increasing on $[y_0,\infty)$, and so there is exactly one sonic point. it remains to count the intersections with the far-field solution.
  \item
    We first show that $y^{\frac{2}{2-\g}}\rt-k$ and $p_{hom}$ have the same number of zeroes on $[y_0,\infty)$. To do this, we define
    \begin{equation*}
      \psi(t):=e^{\mu t}p_{hom}(e^t),\qquad \psit(t):=e^{(\mu+\frac{2}{2-\g})t}\rt_{ext}(e^t).
    \end{equation*}
    The problem then becomes showing that $\psi(t)$ and $\psi(t)+\psit(t)$ have the same number of zeroes on $[-\log(y_0),\infty)$. We then note that by Lemma \ref{ext_fun_matrix_lemma} we have that as $t\to \infty$,
    \begin{equation*}
      \psi(t)=\mu_3e^{\mu t}+O(e^{(\mu-\frac{1}{2-\g}) t})
    \end{equation*}
    and hence has no roots. Similarly, as $t\to -\infty$, we have
    \begin{equation*}
      \psi(t)=c_1\sin(\nu t+d_1)+O(e^{\frac{2}{2-\g} t}).
    \end{equation*}
    In particular, by \eqref{eq:wigglecond} we have $\psi(\log(y_0))=c_1+O(y_0^{\frac{2}{2-\g}})$ which is bounded away from zero. It follows that $\psi$ will have finitely many roots and they will all be contained in an interval $[\log(y_0),\log(Y_0)]$ for some large $Y_0$. Moreover, there exists $\de>0$ so that if $\{t\}_l$ are these roots then $\abs{\psi'}>\de$ on $\cup_l[t_l-\frac{1}{100},t_l+\frac{1}{100}]$ and $\abs{\psi}>\de$ on the complement of this set. It follows from \eqref{eq:ext_construction_bound_1} and \eqref{eq:ext_construction_bound_2} that
    \begin{equation*}
      \abs{\psit},\abs{\psit'}\lesssim \ep y_0^{-\mu}\ev{e^t}^{\mu}.
    \end{equation*}
    Since the roots occured in a compact set, by choosing $\ep y_0^{-\mu}\ll 1$, $\abs{\psit}$, we may ensure $\abs{\psit'}\leq \frac{\de}{100}$ which will ensure that $\psi$ and $\psi+\psit$ have the same number of roots.
  \item
    We next claim that if $y_0^{\frac{2}{2-\g}}\ll 1$ and $\la^{\mu}\lesssim y_0^{\frac{2}{2-\g}+\mu}$, then we will have that $y^{\frac{2}{2-\g}}\rt-k$ and $y^{\frac{2}{2-\g}}(\frac{\g-1}{\g}Q_\la)^{\frac{1}{\g-1}}-k$ will intersect $0$ the same number of times on $[0,y_0]$. To see this, we define
    \begin{align*}
      \chi(t)&=\ev{e^t}^{\mu}\left((\la e^t)^{\frac{2}{2-\g}}(\frac{\g-1}{\g}Q_\la(\la e^t))^{\frac{1}{\g-1}}-k\right),\\
      \chit(t)&=\ev{e^t}^{\mu}\left((\la e^t)^{\frac{2}{2-\g}}\rt(\la e^t)\right)
    \end{align*}
    and note that as $t\to -\infty$ we have
    \begin{equation*}
      \chi(t)=-k+O(e^{\frac{2}{2-\g} t})
    \end{equation*}
    and as $t\to \infty$ we have, using \eqref{laneemdenasymptotics},
    \begin{equation*}
      \chi(t)=c_2\sin(\nu t+d_2)+O(e^{-\mu t}).
    \end{equation*}
    We note that we will have $y\lesssim y_0$, so we have that, using $\la^{\mu}\lesssim y_0^{\frac{2}{2-\g}+\mu}$
    \begin{equation*}
      e^{-\mu t}=(\frac{y}{\la})^{-\mu}\leq \frac{\la^{\mu}}{y_0^{\mu}}\lesssim y_0^{\frac{2}{2-\g}}.
    \end{equation*}
    Thus, if we switch variables to $x$, which is defined implicitly by
    \begin{equation*}
      t=\log(y_0)-\frac{1}{\nu}(x-d_1+d_2)
    \end{equation*}
    we will have be \eqref{eq:wigglecond} that
    \begin{equation*}
      \chi(t)=c_2\cos(x)+O(y_0^{\frac{2}{2-\g}}).
    \end{equation*}
    and $t_{k,\pm}=\log(y_0/\la_{k,\pm})$ corresponds to $x_{k,\pm}=-k \pi\pm \frac{1}{10}$. We can thus see that $\abs{\chi(t)}$ is bounded away from zero on a neighboorhood of $\log(y_0/\la_k)$. Similarly on neighboorhoods of the roots of $\chi(t)$ one can show that the derivative is bouned away from zero. Finally, it follows from \eqref{eq:int_construction_bound_1} and \eqref{eq:int_construction_bound_2} that on $(-\infty,\log(y_0/\la))$ we have
    \begin{equation*}
      \abs{\chi(t)},\abs{\chi'(t)}\lesssim \ev{e^t}^{\mu}(\la e^t)^{\frac{2}{2-\g}}\ev{e^t}^{-2-\mu}\leq (\la e^t)^{\frac{2}{2-\g}}\leq y_0^{\frac{2}{2-\g}}.
    \end{equation*}
    From this, similarly to Step 6, we may conclude that $\chi$ and $\chit$ have the same number of zeroes. It then follows that $y^{\frac{2}{2-\g}}\rt_{int}[\la_i]-k$ and $y^2e^{Q_{\la_i}}-k$ intersect $0$ the same number of times.
  \item
    We see that the number of zeroes of $y^{\frac{2}{2-\g}}\rt-k$ will be the sum of the number of zeroes of $p_{hom}$ on $[y_0,\infty)$ (a finite number, indepedent of $i$) along with the number of zeroes of $y^{\frac{2}{2-\g}}(\frac{\g-1}{\g}Q_{\la_i})^{\frac{1}{\g-1}}-k$ on $[0,y_0]$. We now show that this latter number increases by $1$ when $i\to i+1$. It suffices to show that the function $\chi(t)$ defined in the previous section has exactly one root in $[\log(y_0/\la_i,\log(y_0/\la_{i+1}))]$. We define
    \begin{equation*}
      \chib(t)=\chi(t)-c_2\sin(\nu t+d_2).
    \end{equation*}
    and note that by \eqref{laneemdenasymptotics} we have that as $t\to \infty$
    \begin{equation*}
      \abs{\chib(t)},\abs{\chib'(t)}\lesssim y_0^{\frac{2}{2-\g}}
    \end{equation*}
    and so we see that $\chi(t)$ has the same number of roots on this interval as $c_2\sin(\nu t+d_2)$. By a similar argument as in the previous step, we may see that $c_2\sin(\nu t+d_2)$ has precisely one root on this interval.

    Thus, for $i$ sufficiently large, the number of roots of $y^{\frac{2}{2-\g}}\rt-k$ increases by 1 as $i$ increases by $1$, so by re-enumerating if necessary we see that $y^{\frac{2}{2-\g}}\rt_i-k$ has precisely $i+1$ zeroes.
  \end{enumerate}
\end{proof}
\subsection{Proof that $C^1$ polytropic Hunter solutions are real analytic}
\label{sec:proof}
We recall that by Lemma \ref{lem:normalform} the Lane-Emden equation can be put into the desired normal form at the sonic point. We now state and proof an (easier) lemma regarding the singular ODE problem at the origin.
\begin{lem}
  \label{lem:origin_normal_form}
  Near $y=0$, assuming $\rt(0)=\rho_0$, $\ut(0)=0$ (and $\rho_0>0$) we can put \eqref{polytropicselfsimeq} in normal form. The parameters in this case are
  \begin{equation*}
    a=\rho_0,\qquad b=0,\qquad c=2\rho_0,\qquad d=2\rho_0.
  \end{equation*}
  In particular, $U=\ut'(0)=\frac{2}{3}$ is determined uniquely and we have the nondegeneracies
  \begin{equation*}
    a+bU=\rho_0,\qquad \kappa=2.
  \end{equation*}
\end{lem}
\begin{proof}
  We consider the ODE system \eqref{polytropicselfsimeq}. We assume that $\rt(0)=\rho_0>0$, and multiply the mass-conservation equation by $y$. We divide the $u$ equation by $\g\rho_0^{\g-2}$. Finally, we substitute $\rt=\rho_0+\rhor$. This gives the system
  \begin{equation*}
    \mqty{\frac{(\rho_0+\rhor)^{\g-2}}{\rho_0^{\g-2}} & \frac{\ut+(2-\g)y}{\g \rho_0^{\g-2}}\\ y(\ut+(2-\g)y) & y(\rho_0+\rhor)}\mqty{\rhor'\\ \ut'}+\mqty{\frac{(\g-1)\ut}{\g \rho_0^{\g-2}}+\frac{4\pi}{(4-3\g)\g \rho_0^{\g-2}}(\rho_0+\rhor)(\ut+(2-\g)y)\\ 2(\rho_0+\rhor)(\ut+y)}=0.
  \end{equation*}
  This equation is in the desired normal form. We may then read off
  \begin{equation*}
    a=\rho_0,\qquad b=0,\qquad c=2\rho_0,\qquad d=2\rho_0.
  \end{equation*}
  In this case, the characteristic quadratic, \eqref{eq:characteristic_quadratic}, is linear and the unique solution is $U=-\frac{2}{3}$. We thus have
  \begin{equation*}
    a+bU=\rho_0,\qquad \kappa=2.
  \end{equation*}
\end{proof}

We now prove Theorem \ref{thm:result}.
\begin{proof}
  By Lemma \ref{lem:C1_result} there exists a $C^1$ solution with these properties. The solution is real analytic for $y\in (0,\infty)\setminus y_*$ by standard ODE theory (since the ODE is non-singular). We comment that by Lemma \ref{lem:normalform}, in a neighborhood of the soinc point, the system \eqref{polytropicselfsimeq} can be put in the form of Definition \ref{def:normal_form}. We then comment that if we replace $\o_0=2-\g$ by $\o_{\ep}=\omgr_0[\ep]$, then the characteristic parameters will change contiunuously. Thus, by Lemma \ref{lem:normalform}, for $\ep$ sufficiently small we will have $a+bU\neq 0$, $\kappa>-\frac{1}{2}$. Thus, by the existence part of Lemma \ref{lem:main_analyticity_lemma}, by choosing the Larson-Penston-Hunter boundary condition there exists locally an analytic solution. By the uniqueness aspect of Lemma \ref{lem:local_C1_wellposedness}, this agrees with the previously constructed solution, establishing analyticity at $y_f$. Similarly, by Lemma \ref{lem:origin_normal_form}, on a neighborhood of the origin there is a unique analytic solution. Thus, the solutions provided by Lemma \ref{lem:C1_result} are in fact analytic on $[0,\infty)$ as desired.
\end{proof}
\section{Deferred proofs of Section \ref{sec:exterior}}
\label{sec:ext_proofs}
\subsection{Proof of Lemma \ref{ext_fun_matrix_lemma}}
\label{ext_fun_matrix_proof}
\begin{proof}
  In this case, the linearized ODE we are considering is
  \[
    \mqty{(2-\g)z & kz\\ \frac{2\pi(2-\g)^2}{4-3\g}z^{1-\frac{2}{2-\g}} & (2-\g)z}\mqty{p'\\ \o'}+\mqty{0 & \frac{4-3\g}{2-\g}k\\ \frac{4\pi(2-\g)^2}{(4-3\g)z^{\frac{2}{2-\g}}} & 1+\frac{4\pi k}{(4-3\g)z^{\frac{2}{2-\g}}}}\mqty{p\\ \o}=0.
  \]
  We comment that the coefficient matrix is singular when $kz^{\frac{-2}{2-\g}}=\frac{4-3\g}{2\pi}$. Assuming that this is not the case, we may invert the first matrix to obtain
  \[
    \mqty{p'\\ \o'}+\frac{1}{1-\frac{2\pi}{4-3\g}k z^{-\frac{2}{2-\g}}}\mqty{\frac{-4\pi k}{(4-3\g)z^{1+\frac{2}{2-\g}}} & -3\frac{(\g-1)k}{(2-\g)^2 z}-\frac{4\pi k^2}{(4-3\g)(2-\g)^2 z^{1+\frac{2}{2-\g}}}\\ \frac{4\pi(2-\g)}{(4-3\g)z^{1+\frac{2}{2-\g}}} & \frac{1}{(2-\g)z}+\frac{2\pi k(3\g-2)}{(2-\g)(4-3\g)z^{1+\frac{2}{2-\g}}}}\mqty{p\\ \o}=0.
  \]
  We first let
  \[
    \bar{p}=p+3\frac{(\g-1)k}{2-\g}\omega
  \]
  and then we have
  \[
    \mqty{\bar{p}'\\ \omega'}+\frac{1}{1-\frac{2\pi}{(4-3\g)}k z^{-\frac{2}{2-\g}}}\mqty{-\frac{4\pi k}{z^{1+\frac{2}{2-\g}}} & -\frac{2\pi(6\g^2-15\g+11)k^2}{(2-\g)^2z^{1+\frac{2}{2-\g}}}\\ \frac{4\pi(2-\g)}{(4-3\g)z^{1+\frac{2}{2-\g}}} & \frac{1}{(2-\g)z}+\frac{2\pi k(6\g^2-15\g+10)}{(2-\g)(4-3\g)z^{1+\frac{2}{2-\g}}}}\mqty{\bar{p}\\ \omega}=0
  \]
  We then let
  \[
    q=\bar{p}+\frac{(6\g^2-15\g+11)k}{2(2-\g)^2}\omega
  \]
  and in terms of $(\bar{p},q)$ we have
  \[
    \mqty{\bar{p}'\\ q'}+\frac{1}{1-\frac{2\pi}{4-3\g}kz^{-\frac{2}{2-\g}}}\mqty{0 & -\frac{4\pi k}{z^{1+\frac{2}{2-\g}}}\\ -\frac{1}{(2-\g)z}+\frac{2\pi k}{(2-\g)(4-3\g)z^{1+\frac{2}{2-\g}}} & \frac{1}{(2-\g)z}-\frac{2\pi k(6-5\g)}{(4-3\g)(2-\g)z^{1+\frac{2}{2-\g}}}}\mqty{\bar{p}\\ q}=0.
  \]
  From the first equation we have that
  \[
    q=(\frac{1}{4\pi k})z^{1+\frac{2}{2-\g}}(1-\frac{2\pi}{4-3\g}k z^{-\frac{2}{2-\g}})\bar{p}'
  \]
  and we may substitute this into the second line to obtain the following second order equation for $\bar{p}$ (note a factorization is used to simplify the coefficient of $\bar{p}$)
  \[
    \bar{p}''+\frac{1}{1-\frac{2\pi k}{(4-3\g)}z^{-\frac{2}{2-\g}}}(\frac{5-\g}{(2-\g)z}-\frac{4\pi k}{(2-\g)z^{1+\frac{2}{2-\g}}})\bar{p}'-\frac{1}{1-\frac{2\pi k}{(4-3\g)}z^{-\frac{2}{2-\g}}}\frac{4\pi k}{(2-\g)}z^{-2-\frac{2}{2-\g}}\bar{p}=0.
  \]
  We then let $\bar{p}(z)=f(\frac{4-3\g}{2\pi k}z^\frac{2}{2-\g})$. We let $x=\frac{4-3\g}{2\pi k}z^{\frac{2}{2-\g}}$ and obtain the following equation for $f(x)$:
  \[
    f''+\frac{1}{2x(x-1)}(5x-(8-5\g))f'-\frac{(2-\g)(4-3\g)}{2x^2(x-1)}f=0.
  \]
  We then let $f(x)=g(1-\frac{1}{x})$, $\xi=1-\frac{1}{x}$ to obtain the following equation
  \begin{equation}
    g''+\frac{(-4+5\g)\xi+3-5\g}{2\xi(\xi-1)}g'+\frac{(2-\g)(4-3\g)}{2\xi(\xi-1)}g=0.
    \label{hypergeomeq}
  \end{equation}
  We recall the definitions
  \[
    \mu:=\frac{6-5\g}{2(2-\g)},\qquad \nu:=\frac{\sqrt{-\g^2-20\g+28}}{2(2-\g)}.
  \]
  We note that this is the hypergeometric equation with
  \[
    a=\frac{2-\g}{2}(-\mu-i\nu),\qquad b=\frac{2-\g}{2}(-\mu+i\nu),\qquad c=\frac{5\g-3}{2}.
  \]
  The unique solution (up to a constant multiple) which is smooth on a neighborhood of $\xi=0$ is given by
  \[
    g_{1}(\xi)=\frac{6\g^2-15\g+11}{3\g-1}F(\frac{2-\g}{2}(-\mu-i\nu),\frac{2-\g}{2}(-\mu+i\nu),\frac{5\g-3}{2},\xi).
  \]
  We recall that as $\xi\to 0$, we have
  \[
    F(\frac{2-\g}{2}(-\mu-i\nu),\frac{2-\g}{2}(-\mu+i\nu),\frac{5\g-3}{2},\xi)=1+\frac{(4-3\g)(2-\g)}{5-3\g}\xi+o(\xi).
  \]
  where $F(a,b,c,\xi)$ denotes Gauss's hypergeometric function. By inverting the change of variables, we get the following homogeneous solution for $L$:
  \[
    \mqty{p_{hom,pre}(z)\\ \o_{hom,pre}(z)}=\mqty{\frac{3\g-1}{(6\g^2-15\g+11)} & -\frac{6(\g-1)}{(4-3\g)(6\g^2-15\g+11)}(1-\frac{2\pi k}{(4-3\g)z^{\frac{2}{2-\g}}})\\ -\frac{2(2-\g)^2}{(6\g^2-15\g+11)k} & \frac{2(2-\g)}{(4-3\g)(6\g^2-15\g+11)k}(1-\frac{2\pi k}{(4-3\g)z^{\frac{2}{2-\g}}})}\mqty{g_1(1-\frac{2\pi k}{(4-3\g) z^{\frac{2}{2-\g}}})\\ g_1'(1-\frac{2\pi k}{(4-3\g) z^{\frac{2}{2-\g}}})}.
  \]
  One can then check that as $z\to y_f$, we have
  \[
    \begin{cases}
      p_{hom,pre}&=1+\frac{-18\g^2+18\g+4}{(5\g-3)(3\g-1)}\frac{1}{y_f}(z-y_f)+o(z-y_f)\\
      \omega_{hom,pre}&=-\frac{2(2-\g)^2}{(3\g-1)k}+\frac{4(2-\g)(\g-1)(7-3\g)}{(3\g-1)(5\g-3)}\frac{1}{k y_f}(z-y_f)+o(z-y_f)
    \end{cases}.
  \]
  Finally, letting $(p_{hom},\o_{hom})=\frac{k(3\g-1)}{2(2-\g)}(p_{hom,pre},\o_{hom,pre})$ and this finishes proving part (a).

  For the next part, we start by considering
  \begin{equation*}
    \begin{cases}
      g_3(\xi)&:=F(\frac{2-\g}{2}(-\mu-i\nu),\frac{2-\g}{2}(-\mu+i\nu),-\frac{1}{2},1-\xi)\\
      g_{4}(\xi)&:=-\frac{2}{3}(1-\xi)^{\frac{3}{2}}F(\frac{15\g-12}{4}+i\frac{2-\g}{2}\nu,\frac{15\g-12}{4}-i\frac{2-\g}{2}\nu,\frac{5}{2},1-\xi)
    \end{cases}.
  \end{equation*}
  This gives the following fundamental matrix for \eqref{hypergeomeq}:
  \begin{equation*}
    V_{\infty}(\xi):=\mqty{g_{3}(\xi) & g_4(\xi)\\ g_3'(\xi) & g_4'(\xi)}.
  \end{equation*}
  In this case, the determinant satisfies the following ODE and asymptotic expansion as $\xi\to 1$:
  \[
    W_{\infty}'+\frac{(-4+5\g)\xi+3-5\g}{2\xi(\xi-1)}W_{\infty}=0,\qquad W_{\infty}=(1-\xi)^{\frac{1}{2}}+O((1-\xi)^{\frac{3}{2}}).
  \]
  Solving this ODE gives
  \[
    W_{\infty}(\xi)=\frac{\sqrt{1-\xi}}{\xi^{\frac{5\g-3}{2}}},\qquad V_{\infty}^{-1}(\xi)=\frac{\xi^{\frac{5\g-3}{2}}}{\sqrt{1-\xi}}\mqty{g_4'(\xi) & -g_4(\xi)\\ -g_3'(\xi) & g_3(\xi)}.
  \]
  This leads to the following fundamental matrix
  \[
    \tilde{U}_{\infty}(z):=\mqty{\frac{3\g-1}{(6\g^2-15\g+11)} & -\frac{6(\g-1)}{(4-3\g)(6\g^2-15\g+11)}(1-\frac{2\pi k}{(4-3\g)z^{\frac{2}{2-\g}}})\\ -\frac{2(2-\g)^2}{(6\g^2-15\g+11)k} & \frac{2(2-\g)}{(4-3\g)(6\g^2-15\g+11)k}(1-\frac{2\pi k}{(4-3\g)z^{\frac{2}{2-\g}}})}V_{\infty}(1-\frac{2\pi k}{(4-3\g)z^{\frac{2}{2-\g}}}).
  \]
  This will satisfy $L U_{\infty}=0$. Taking the limit as $z\to \infty$ gives 
  \[
    \tilde{U}_{\infty}(z)=
    \mqty{1& -\frac{6(\g-1)}{(4-3\g)(6\g^2-15\g+11)}(\frac{2\pi k}{4-3\g})^{\frac{1}{2}}\frac{1}{z^{\frac{1}{2-\g}}}\\
      0 & \frac{2(2-\g)}{(4-3\g)(6\g^2-15\g+11)k}(\frac{2\pi k}{4-3\g})^{\frac{1}{2}}\frac{1}{z^{\frac{1}{2-\g}}}}+\mqty{O(\frac{1}{z^{\frac{2}{2-\g}}}) & O(\frac{1}{z^{\frac{3}{2-\g}}})\\ O(\frac{1}{z^{\frac{2}{2-\g}}}) & O(\frac{1}{z^{\frac{3}{2-\g}}})}.
  \]
  Finally, for the fundamental matrix of the lemma, we define
  \[
    U_{\infty}(z)=\tilde{U}_{\infty}\cdot \mqty{1 & 0\\ 0 & \frac{(4-3\g)(6\g^2-15\g+11)k}{2(2-\g)}(\frac{4-3\g}{2\pi k})^{\frac{1}{2}}}.
  \]
  To consider the asymptotics of the inverse, we first compute
  \begin{multline*}
    \det \mqty{\frac{3\g-1}{(6\g^2-15\g+11)} & -\frac{6(\g-1)}{(4-3\g)(6\g^2-15\g+11)}(1-\frac{2\pi k}{(4-3\g)z^{\frac{2}{2-\g}}})\\ -\frac{2(2-\g)^2}{(6\g^2-15\g+11)k} & \frac{2(2-\g)}{(4-3\g)(6\g^2-15\g+11)k}(1-\frac{2\pi k}{(4-3\g)z^{\frac{2}{2-\g}}})}\\
    =\frac{2(2-\g)}{(4-3\g)(6\g^2-15\g+11)k}(1-\frac{2\pi k}{(4-3\g)z^{\frac{2}{2-\g}}}),
  \end{multline*}
  which in particular gives
  \begin{multline*}
    \mqty{\frac{3\g-1}{(6\g^2-15\g+11)} & -\frac{6(\g-1)}{(4-3\g)(6\g^2-15\g+11)}(1-\frac{2\pi k}{(4-3\g)z^{\frac{2}{2-\g}}})\\ -\frac{2(2-\g)^2}{(6\g^2-15\g+11)k} & \frac{2(2-\g)}{(4-3\g)(6\g^2-15\g+11)k}(1-\frac{2\pi k}{(4-3\g)z^{\frac{2}{2-\g}}})}^{-1}\\
    =\mqty{1 & \frac{3(\g-1)k}{2-\g}\\ (2-\g)(4-3\g)(1-\frac{2\pi k}{(4-3\g)z^{\frac{2}{2-\g}}})^{-1} & \frac{(4-3\g)(3\g-1)k}{2(2-\g)}(1-\frac{2\pi k}{(4-3\g)z^{\frac{2}{2-\g}}})^{-1}}.
  \end{multline*}
  This allows us to compute the asymptotics for the inverse of $\tilde{U}_{\infty}$ to be
  \[
    \tilde{U}_{\infty}^{-1}=\mqty{1 & \frac{3(\g-1)k}{2-\g}\\ 0 & \frac{(4-3\g)(6\g^2-15\g+11)k}{2(2-\g)}(\frac{4-3\g}{2\pi k})^{\frac{1}{2}} z^{\frac{1}{2-\g}}}+\mqty{O(\frac{1}{z^{\frac{2}{2-\g}}}) & O(\frac{1}{z^{\frac{2}{2-\g}}})\\ O(\frac{1}{z^{\frac{1}{2-\g}}}) & O(\frac{1}{z^{\frac{1}{2-\g}}})}.
  \]
  Finally, the asymptotics of $U_{\infty}^{-1}$ follow from
  \[
    U_{\infty}^{-1}=\mqty{1 & 0\\ 0 & \frac{(4-3\g)(6\g^2-15\g+11)k}{2(2-\g)}(\frac{4-3\g}{2\pi k})^{\frac{1}{2}}}^{-1}\tilde{U}_{\infty}^{-1}.
  \]
  To construct $U_0$, we start by defining
  \[
    \begin{cases}
      g_{a}:=(-\xi)^{\frac{2-\g}{2}(\mu+i\nu)}F(\frac{2-\g}{2}(-\mu-i\nu),\frac{2-\g}{2}(-\mu-i\nu)-\frac{5(\g-1)}{2},1-i(2-\g)\nu,\frac{1}{\xi})\\
      g_b:=(-\xi)^{\frac{2-\g}{2}(\mu-i\nu)}F(\frac{2-\g}{2}(-\mu+i\nu),\frac{2-\g}{2}(-\mu+i\nu)-\frac{5(\g-1)}{2},1+i(2-\g)\nu,\frac{1}{\xi})
    \end{cases}.
  \]
  We then define
  \[
    g_5=\frac{g_a+g_b}{2},\qquad g_6=-\frac{g_a-g_b}{2i}
  \]
  which are real-valued. We then define for $\xi\in (-\infty,0)$,
  \[
    V_0(\xi)=\mqty{g_5(\xi) & g_6(\xi)\\ g_5'(\xi)& g_6'(\xi)},\qquad W_0=\det V_0.
  \]
  In this case, the boundary condition is that as $\xi \to -\infty$,
  \[
    W_0=\frac{2-\g}{2}\nu (-\xi)^{\frac{4-5\g}{2}}+O((-\xi)^{\frac{2-5\g}{2}}).
  \]
  Thus, we get
  \[
    W_0=\frac{(2-\g)\nu}{2} \frac{\sqrt{1-\xi}}{(-\xi)^{\frac{5\g-3}{2}}},\qquad V_0^{-1}(\xi)=\frac{2}{(2-\g)\nu}\frac{(-\xi)^{\frac{5\g-3}{2}}}{\sqrt{1-\xi}}\mqty{g_6'(\xi)& -g_6(\xi)\\ -g_5'(\xi) & g_5(\xi)}.
  \]
  We then get the following fundamental matrix for $z\in (0,1)$:
  \[
    \tilde{U}_{0}(z):=\mqty{\frac{3\g-1}{(6\g^2-15\g+11)} & -\frac{6(\g-1)}{(4-3\g)(6\g^2-15\g+11)}(1-\frac{2\pi k}{(4-3\g)z^{\frac{2}{2-\g}}})\\ -\frac{2(2-\g)^2}{(6\g^2-15\g+11)k} & \frac{2(2-\g)}{(4-3\g)(6\g^2-15\g+11)k}(1-\frac{2\pi k}{(4-3\g)z^{\frac{2}{2-\g}}})}V_{0}(1-\frac{2\pi k}{(4-3\g)z^{\frac{2}{2-\g}}}).
  \]
  We then note that as $\xi\to \infty$, we have
  \[
    \tilde{U}_0((\frac{2\pi k}{(1-\xi)(4-3\g)})^{\frac{2-\g}{2}})=(-\xi)^{\frac{6-5\g}{4}}\mqty{T_1 & T_2\\ T_3 & T_4}+O((-\xi)^{\frac{2-5\g}{4}})
  \]
  where
  \begin{align*}
    T_1&=\frac{(4-3\g)(3\g-1)-3(\g-1)(2-\g)\mu}{(6\g^2-15\g+11)(4-3\g)}\cos(\frac{2-\g}{2}\nu\log(-\xi))\\
       &\qquad+\frac{3(\g-1)(2-\g)\nu}{(4-3\g)(6\g^2-15\g+11)} \sin(\frac{2-\g}{2}\nu \log(-\xi)),\\
    T_2&=\frac{-(4-3\g)(3\g-1)+3(\g-1)(2-\g)\mu}{(6\g^2-15\g+11)(4-3\g)}\sin(\frac{2-\g}{2}\nu \log(-\xi))\\
       &\qquad+\frac{3(\g-1)(2-\g)\nu}{(4-3\g)(6\g^2-15\g+11)}\cos(\frac{2-\g}{2}\nu \log(-\xi)),\\
    T_3&=\frac{-2(2-\g)^2(4-3\g)+(2-\g)^2\mu}{(4-3\g)(6\g^2-15\g+11)k}\cos(\frac{2-\g}{2}\nu \log(-\xi))\\
       &\qquad+\frac{-(2-\g)^2\nu}{(4-3\g)(6\g^2-15\g+11)k}\sin(\frac{2-\g}{2}\nu \log(-\xi)),\\
    T_4&=\frac{2(2-\g)^2(4-3\g)-(2-\g)^2\mu}{2(4-3\g)(6\g^2-15\g+11)k}\sin(\frac{2-\g}{2}\nu \log(-\xi))\\
       &\qquad+\frac{-(2-\g)^2\nu}{(4-3\g)(6\g^2-15\g+11)k}\cos(\frac{2-\g}{2}\nu \log(-\xi)).
  \end{align*}
  We then claim that we may factorize this as
  \[
    \mqty{T_1 & T_2\\ T_3&T_4}=\mqty{\cos(-\frac{2-\g}{2}\nu \log(-\xi)) & \sin(-\frac{2-\g}{2}\nu \log(-\xi))\\ \frac{(2-\g)^{\frac{3}{2}}}{k}\cos(-\frac{2-\g}{2}\nu \log(-\xi)+\theta_0) & \frac{(2-\g)^{\frac{3}{2}}}{k}\sin(-\frac{2-\g}{2}\nu \log(-\xi)+\theta_0)}A
  \]
  for some constant-coefficient matrix $A$. Indeed, by inspecting the upper two entries if this were to hold, we must have
  \[
    A=\mqty{\frac{(4-3\g)(3\g-1)-3(\g-1)(2-\g)\mu}{(6\g^2-15\g+11)(4-3\g)} & \frac{3(\g-1)(2-\g)\nu}{(4-3\g)(6\g^2-15\g+11)} \\ -\frac{3(\g-1)(2-\g)\nu}{(4-3\g)(6\g^2-15\g+11)} & \frac{(4-3\g)(3\g-1)-3(\g-1)(2-\g)\mu}{(6\g^2-15\g+11)(4-3\g)} }.
  \]
  From this ansatz, verifying this factorization for $T_3$ follows from showing the two equalities
  \begin{multline*}
    (2-\g)^{\frac{3}{2}}(((4-3\g)(3\g-1)-3(\g-1)(2-\g)\mu)\cos(\theta_0)-3(\g-1)(2-\g)\nu \sin(\theta_0))\\
    =-2(2-\g)^2(4-3\g)+(2-\g)^2\mu,
  \end{multline*}
  \begin{multline*}
    (2-\g)^{\frac{3}{2}}(((4-3\g)(3\g-1)-3(\g-1)(2-\g)\mu)\sin(\theta_0)+3(\g-1)(2-\g)\nu \cos(\theta_0))=-(2-\g)^2\nu.
  \end{multline*}
  Indeed, to verify these, we recall
  \[
    \cos(\theta_0)=-\frac{2+\g}{4\sqrt{2-\g}},\qquad \sin(\theta_0)=-\frac{\sqrt{-\g^2-20\g-28}}{4\sqrt{2-\g}}=-\nu\frac{\sqrt{2-\g}}{2}
  \]
  We can then simplify the first equation to be
  \[
    (2+\g)(3\g^2+3\g-10)+3(\g-1)(-\g^2-20\g+28)=-16(2-\g)(4-3\g)+4(6-5\g).
  \]
  This holds by direct computation. Similarly, the second equation can be rewritten as
  \[
    -2(4-3\g)(3\g-1)+3(\g-1)(6-5\g)-3(\g-1)(2+\g)=-4
  \]
  Another direct computation verifies this identity. In terms of the variable $z$, we have that
  \begin{multline*}
    \tilde{U}_0(z)=(\frac{2\pi}{4-3\g})^{\frac{6-5\g}{4}}z^{-\mu}\mqty{\cos(\nu \log((\frac{4-3\g}{2\pi k})^{\frac{2-\g}{2}}z)) & \sin(\nu \log((\frac{4-3\g}{2\pi k})^{\frac{2-\g}{2}}z))\\ \cos(\nu \log((\frac{4-3\g}{2\pi k})^{\frac{2-\g}{2}}z)+\theta_0) &\sin(\nu \log((\frac{4-3\g}{2\pi k})^{\frac{2-\g}{2}}z)+\theta_0)}\cdot \\
    \mqty{\frac{(4-3\g)(3\g-1)-3(\g-1)(2-\g)\mu}{(6\g^2-15\g+11)(4-3\g)} & \frac{3(\g-1)(2-\g)\nu}{(4-3\g)(6\g^2-15\g+11)} \\ -\frac{3(\g-1)(2-\g)\nu}{(4-3\g)(6\g^2-15\g+11)} & \frac{(4-3\g)(3\g-1)-3(\g-1)(2-\g)\mu}{(6\g^2-15\g+11)(4-3\g)} }+O(z^{\frac{5\g-2}{2(2-\g)}}).
  \end{multline*}
  Finally, we define
  \begin{multline*}
    U_0(z)=(\frac{2\pi}{4-3\g})^{-\frac{6-5\g}{4}}\tilde{U}_0(z)\cdot \mqty{\frac{(4-3\g)(3\g-1)-3(\g-1)(2-\g)\mu}{(6\g^2-15\g+11)(4-3\g)} & \frac{3(\g-1)(2-\g)\nu}{(4-3\g)(6\g^2-15\g+11)} \\ -\frac{3(\g-1)(2-\g)\nu}{(4-3\g)(6\g^2-15\g+11)} & \frac{(4-3\g)(3\g-1)-3(\g-1)(2-\g)\mu}{(6\g^2-15\g+11)(4-3\g)} }^{-1}\\
    \cdot \mqty{\cos(\frac{2-\g}{2}\nu\log(\frac{4-3\g}{2\pi k})) & -\sin(\frac{2-\g}{2}\nu\log(\frac{4-3\g}{2\pi k}))\\ -\sin(\frac{2-\g}{2}\nu\log(\frac{4-3\g}{2\pi k})) & (\cos(\frac{2-\g}{2}\nu\log(\frac{4-3\g}{2\pi k})))}
  \end{multline*}
  and it has the asymptotics stated in the Lemma. One can similarly compute the asymptotics for $U_{0}^{-1}$ finishing the proof of part (b). 

  Part (c) follows from the standard existence and uniqueness for linear ODEs as there are no singularities on either of these intervals. In principle one could compute these constants exactly using connection formulae (as done for $\g=1$ in \cite{sandine}), but this is omitted in this case.

  For part (d), we start by defining the following on $(-\infty,0)\cup (0,1)$:
  \[
    g_2(\xi)=\abs{\xi}^{-\frac{5(\g-1)}{2}}F(\frac{4-5\g}{4}-i\frac{\sqrt{-\g^2-20\g+28}}{4},\frac{4-5\g}{4}+i\frac{\sqrt{-\g^2-20\g+28}}{4},\frac{7-5\g}{2},\xi).
  \]
  We comment that this is a solution to \eqref{hypergeomeq} defined on a neighborhood of $\xi=0$ which is linearly independent from $g_1$ (see section 5.10.1 of \cite{olver}). We now use (and $g_1$) this to construct a fundamental matrix for $L$ on a neighborhood of $z=y_f$. We recall that as $\xi\to 0$, we have
  \begin{multline*}
    F(\frac{4-5\g}{4}-i\frac{\sqrt{-\g^2-20\g+28}}{5},\frac{4-5\g}{4}+i\frac{\sqrt{-\g^2-20\g+28}}{5},\frac{7-5\g}{2},\xi)\\
    =1+\frac{6\g^2-15\g+11}{2(7-5\g)}\xi+o(\xi).
  \end{multline*}
  We note that we have (assuming $\xi>0$)
  \[
    g_2=\xi^{-\frac{5(\g-1)}{2}}+\frac{6\g^2-15\g+11}{2(7-5\g)}\xi^{\frac{7-5\g}{2}}+o(\xi^{\frac{7-5\g}{2}}),
  \]
  \[
    g_2'=-\frac{5(\g-1)}{2}\xi^{-\frac{5\g-3}{2}}+\frac{6\g^2-15\g+11}{4}\xi^{-\frac{5(\g-1)}{2}}+o(\xi^{\frac{5(\g-1)}{2}}).
  \]
  We then recall that
  \[
    1-\frac{2\pi k}{(4-3\g)z^{\frac{2}{2-\g}}}=\frac{2}{(2-\g)y_f}(z-y_f)-\frac{4-\g}{(2-\g)^2}\frac{(z-y_f)^2}{y_f^2}+o(z-y_f)^2.
  \]
  from which it follows that 
  \begin{align*}
    g_2(1-\frac{2\pi k}{(4-3\g)z^{\frac{2}{2-\g}}})&=\left(\frac{2(z-y_f)}{(2-\g)y_f}\right)^{-\frac{5(\g-1)}{2}}\\
                                                   &-\frac{\left(25 \gamma ^2-61 \gamma +32\right) (3-\gamma )}{8 (7-5 \gamma )}\left(\frac{2(z-y_f)}{(2-\g)y_f}\right)^{\frac{7-5\g}{2}}+o(z-y_f)^{\frac{7-5\g}{2}},\\
    g_2'(1-\frac{2\pi k}{(4-3\g)z^{\frac{2}{2-\g}}})&=-\frac{5(\g-1)}{2}\left(\frac{2(z-y_f)}{(2-\g)y_f}\right)^{-\frac{5\g-3}{2}}\\
                                                   &+\frac{1}{16} \left(25 \gamma ^3-116 \gamma ^2+115 \gamma -16\right)\left(\frac{2(z-y_f)}{(2-\g)y_f}\right)^{\frac{-5(\g-1)}{2}}+o(z-y_f)^{-\frac{5(\g-1)}{2}}.
  \end{align*}
  We then note that as $z\to y_f$, we have that
  \begin{multline*}
    \mqty{\frac{3\g-1}{(6\g^2-15\g+11)} & -\frac{6(\g-1)}{(4-3\g)(6\g^2-15\g+11)}(1-\frac{2\pi k}{(4-3\g)z^{\frac{2}{2-\g}}})\\ -\frac{2(2-\g)^2}{(6\g^2-15\g+11)k} & \frac{2(2-\g)}{(4-3\g)(6\g^2-15\g+11)k}(1-\frac{2\pi k}{(4-3\g)z^{\frac{2}{2-\g}}})}\mqty{g_2(1-\frac{2\pi k}{(4-3\g) z^{\frac{2}{2-\g}}})\\ g_2'(1-\frac{2\pi k}{(4-3\g) z^{\frac{2}{2-\g}}})}=\\
    =\mqty{\frac{1}{4-3\g}\\ -\frac{(2-\g)}{(4-3\g)k}}\abs{\frac{2(z-y_f)}{(2-\g)y_f}}^{-\frac{5(\g-1)}{2}}+\mqty{\frac{25 \gamma ^3-136 \gamma ^2+191 \gamma -72}{8 (4-3 \gamma ) (7-5 \gamma )}\\ -\frac{(2-\gamma ) (\gamma -1) \left(25 \gamma ^2-111 \gamma +104\right)}{8 (4-3 \gamma ) (7-5 \gamma ) k}}\frac{2(z-y_f)}{(2-\g)y_f}\abs{\frac{2(z-y_f)}{(2-\g)y_f}}^{-\frac{5(\g-1)}{2}}\\
    +o((z-y_f)\abs{z-y_f}^{-\frac{5(\g-1)}{2}}).
  \end{multline*}
  Thus, if we define the (prescaling) subordinate homogeneous solution by
  \begin{multline*}
    \mqty{p_{sub,pre}\\ \o_{sub,pre}}=(4-3\g)(\frac{2}{2-\g}y_f^{-1})^{\frac{5(\g-1)}{2}}\mqty{\frac{3\g-1}{(6\g^2-15\g+11)} & -\frac{6(\g-1)}{(4-3\g)(6\g^2-15\g+11)}(1-\frac{2\pi k}{(4-3\g)z^{\frac{2}{2-\g}}})\\ -\frac{2(2-\g)^2}{(6\g^2-15\g+11)k} & \frac{2(2-\g)}{(4-3\g)(6\g^2-15\g+11)k}(1-\frac{2\pi k}{(4-3\g)z^{\frac{2}{2-\g}}})}\\
    \mqty{g_2(1-\frac{2\pi k}{(4-3\g) z^{\frac{2}{2-\g}}})\\ g_2'(1-\frac{2\pi k}{(4-3\g) z^{\frac{2}{2-\g}}})}.
  \end{multline*}
  We then define $(p_{sub},\o_{sub})=\frac{k(3\g-1)}{2(2-\g)}(p_{sub,pre},\o_{sub,pre})$ then
  \begin{multline*}
    \mqty{p_{hom} & p_{sub}\\ \o_{hom} & \o_{sub}}=\frac{k(3\g-1)}{2(2-\g)}\mqty{1 & \abs{z-y_f}^{-\frac{5(\g-1)}{2}}\\ -\frac{2(2-\g)^2}{k(3\g-1)}& -\frac{(2-\g)}{k}\abs{z-y_f}^{-\frac{5(\g-1)}{2}}}\\
    +\frac{k(3\g-1)}{2(2-\g)}\mqty{\frac{2(-9\g^2+9\g+2)}{(3\g-1)(5\g-3)y_f}(z-y_f) & \frac{25 \gamma ^3-136 \gamma ^2+191 \gamma -72}{4 (2-\gamma) (7-5 \gamma) y_f} (z-y_f)\abs{z-y_f}^{-\frac{5(\g-1)}{2}}\\ \frac{4(2-\g)(\g-1)(7-3\g)}{k y_f(3\g-1)(5\g-3)}(z-y_f) & -\frac{(\gamma -1) \left(25 \gamma ^2-111 \gamma +104\right)}{4 (7-5 \gamma) k y_f}(z-y_f)\abs{z-y_f}^{-\frac{5(\g-1)}{2}}}\\
    +\mqty{o(z-y_f) & o(z-y_f)\abs{z-y_f}^{-\frac{5(\g-1)}{2}}\\ o(z-y_f) & o(z-y_f)\abs{z-y_f}^{-\frac{5(\g-1)}{2}}}.
  \end{multline*}
  Then, as in the previous section existence and uniqueness yield existence of $M_{\infty},M_{0}$ such that
  \[
    U_{\infty}M_{\infty}=U_0M_0=\mqty{p_{hom} & p_{sub}\\ \o_{hom} & \o_{sub}}
  \]
  yielding the first asymptotics stated in part (d) of the Lemma. 

  For part (e), we note that we can factor
  \[
    \mqty{p_{hom} & p_{sub}\\ \o_{hom} & \o_{sub}}=\left(X_0+X_1(z-y_f)+o(z-y_f)\right)\mqty{1 & 0\\ 0 & \abs{z-y_f}^{-\frac{5(\g-1)}{2}}}
  \]
  where
  \[
    X_0=\mqty{\frac{(3\g-1)k}{2(2-\g)} &\frac{(3\g-1)k}{2(2-\g)}\\ -(2-\g)& -\frac{3\g-1}{2}},\qquad X_1=\mqty{\frac{(-9\g^2+9\g+2)}{(2-\g)(5\g-3)}\frac{k}{y_f} & -\frac{(3 \gamma -1) \left(25 \gamma ^3-136 \gamma ^2+191 \gamma -72\right) k}{8 (\gamma -2)^2 (5 \gamma -7) y_f} \\ \frac{2(\g-1)(7-3\g)}{y_f(5\g-3)} & -\frac{(\gamma -1) (3 \gamma -1) \left(25 \gamma ^2-111 \gamma +104\right)}{8 (\gamma -2) (5 \gamma -7) y_f}}.
  \]
  We note that 
  \[
    \mqty{p_{hom} & p_{sub}\\ \o_{hom} & \o_{sub}}=\left(X_0+X_1(z-y_f)+o(z-y_f)\right)\mqty{1 & 0\\ 0 & \abs{z-y_f}^{-\frac{5(\g-1)}{2}}}.
  \]
  We then write abstractly
  \[
    \mqty{(2-\g)z & k z\\ \frac{2\pi(2-\g)^2}{4-3\g}z^{1-\frac{2}{2-\g}} & (2-\g)z}^{-1}\mqty{P\\ \O}=I_0\frac{1}{z-y_f}+I_1+o(1).
  \]
  It follows that the Taylor expansion of the solution will be given by (here $\G$ can be chosen to be arbitrary, it will cancel out)
  \begin{multline}
    \mqty{p\\ \o}=X_0\mqty{\G & 0\\ 0 & \frac{2}{5(\g-1)}}X_0^{-1}I_0\left(X_1\mqty{\G & 0\\ 0 & \frac{2}{5(\g-1)}} X_0^{-1}I_0+X_0\mqty{1 & 0\\ 0 & \frac{2}{(5\g-3)}}X_0^{-1}I_1\right.\\
    \left.-X_0\mqty{1 & 0\\ 0 & \frac{2}{(5\g-3)}} X_0^{-1}X_1 X_0^{-1}I_0\right)(z-y_f)+o(z-y_f).
    \label{formalexpansion}
  \end{multline}
  We thus let
  \begin{align*}
    A_0&=X_0\mqty{\G & 0\\ 0 & \frac{2}{5(\g-1)}}X_0^{-1},\\
    A_1&=X_0\mqty{1 & 0\\ 0 & \frac{2}{5\g-3}}X_0^{-1},\\
    A_2&=X_1X_0^{-1},
  \end{align*}
  and \eqref{formalexpansion} can be simplified to
  \begin{equation}
    \mqty{p\\ \o}=A_0 I_0+\left(A_2 A_0 I_0+A_1 I_1-A_1 A_2 I_0\right)(z-y_f)+o(z-y_f).
    \label{formalexpansion2}
  \end{equation}
  We may then explicitly compute
  \begin{align*}
    A_0&=\mqty{\frac{5 (3 \gamma -1) (\gamma -1) \Gamma -4 (2-\gamma )}{25 (\gamma -1)^2}& \frac{k (5 (3 \gamma -1) (\gamma -1) \Gamma -2 (3 \gamma -1))}{25 (2-\gamma ) (\gamma -1)^2}\\ \frac{4 (2-\gamma )^2-10 (2-\gamma )^2 (\gamma -1) \Gamma }{25 (\gamma -1)^2 k} & \frac{2 (3 \gamma -1)-10 (2-\gamma ) (\gamma -1) \Gamma }{25 (\gamma -1)^2}},\\
    A_1&=\left(
         \begin{array}{cc}
           \frac{3 \gamma +1}{5 \gamma -3} & \frac{(3 \gamma -1) k}{(2-\gamma ) (5 \gamma -3)} \\
           -\frac{2 (2-\gamma )^2}{(5 \gamma -3) k} & \frac{2 (\gamma -1)}{5 \gamma -3} \\
         \end{array}
         \right),\\
    A_2&=\mqty{-\frac{25 \gamma ^3-162 \gamma ^2+197 \gamma -32}{2 (7-5 \gamma ) (5 \gamma -3) y_f} & -\frac{\left(75 \gamma ^4-331 \gamma ^3+321 \gamma ^2-81 \gamma +88\right) k}{4 (2-\gamma )^2 (7-5 \gamma ) (5 \gamma -3) y_f}\\ \frac{(2-\gamma ) \left(25 \gamma ^3-102 \gamma ^2+81 \gamma +16\right)}{2 (7-5 \gamma ) (5 \gamma -3) k y_f} & \frac{75 \gamma ^4-451 \gamma ^3+913 \gamma ^2-873 \gamma +376}{4 (2-\gamma ) (7-5 \gamma ) (5 \gamma -3) y_f}},\\
    I_0&=\left(
         \begin{array}{c}
           \frac{1}{2} \left(\sigma _1-\frac{k \sigma _2}{2-\gamma }\right) \\
           -\frac{(2-\gamma ) \left(\sigma _1-\frac{k \sigma _2}{2-\gamma }\right)}{2 k} \\
         \end{array}
         \right),\\
    I_1&=\left(
         \begin{array}{c}
           -\frac{\sigma _2 (\gamma  k)}{4 (2-\gamma )^2 y_f}+\frac{\gamma  \sigma _1}{4 (2-\gamma ) y_f}+\frac{\sigma _3}{2} \\
           \frac{(4-\gamma ) \sigma _1}{4 k y_f}+\frac{\gamma  \sigma _2}{4 (2-\gamma ) y_f}-\frac{(2-\gamma ) \sigma _3}{2 k} \\
         \end{array}
         \right),
  \end{align*}
  and then simplifying \eqref{formalexpansion2} gives the formula stated in the lemma.
\end{proof}
\subsection{Proof of Lemma \ref{ext_linear_lemma}}
\label{ext_linear_proof}

We suppose that $(P,\O)\in N_{y_0}[\ep]$. It follows from Lemma \ref{ext_fun_matrix_lemma} that $(\pr,\omgr)$ is in $C^1[y_0,\infty]$. It follows the computations in section \ref{sec:ext_diffeq} that
\[
  p(y_f)=\pr_{0,e}[\ep],\qquad \o(y_f)=\omgr_{0,e}[\ep],\qquad \pr'(y_f)=\pr_{1,e}[\ep],\qquad \omgr'(y_f)=\omgr_{1,e}[\ep].
\]
It will also follow from examining the integral kernel of $R$ that
\begin{equation}
  \sup_{\frac{1}{2}\leq z\leq 2}\abs{p}+\abs{\o}+\abs{\frac{p+\frac{3\g-1}{2(2-\g)^2}k \o+\frac{\s_1}{2(2-\g)}-\frac{k\s_2}{2(2-\g^2)}}{z-y_f}}\lesssim \sup_{\frac{1}{2}\leq z\leq 2}\abs{P}+\abs{\O}+\abs{\frac{P-\frac{k}{2-\g}\O}{z-y_f}}+O(\ep).
  \label{intermed52}
\end{equation}
We may then rewrite $L(p,\o)=(P,\O)$ as
\[
  E_{lin}\mqty{p\\ \o}'=(P,\O)-F_{lin}\mqty{p\\ \o},
\]
We recall that as $z\to y_f$, we have that
\[
  E_{lin}^{-1}=\mqty{\frac{1}{2} & -\frac{k}{2(2-\g)}\\ -\frac{2-\g}{2 k} & \frac{1}{2}}\frac{1}{z-y_f}+O(1),\qquad E_{lin}^{-1}F_{lin}=\mqty{-(2-\g) & -\frac{(3\g-1)k}{2(2-\g)}\\ \frac{(2-\g)^2}{k} & \frac{3\g-1}{2}}\frac{1}{z-y_f}+O(1).
\]
and then if we invert $E_{lin}$ and use \eqref{intermed52} we obtain
\[
  \sup_{\frac{1}{2}\leq z\leq 2}\abs{p'}+\abs{\o'}\lesssim \sup_{\frac{1}{2}\leq z\leq 2}\abs{P}+\abs{\O}+\abs{\frac{P-\frac{k}{2-\g}\O}{z-y_f}}+O(\ep).
\]
For $z\geq 2$, we recall that as $z\to \infty$,
\[
  U_{\infty}(z)=\left(\mqty{1 & -\frac{3(\g-1)k}{2-\g}\\ 0 & 1}+O(\frac{1}{z^{\frac{2}{2-\g}}})\right)\cdot \mqty{1 & 0\\ 0 & \frac{1}{z^{\frac{1}{2-\g}}}},
\]
\[
  E_{lin}(z)=\mqty{2-\g & k\\ 0 & 2-\g}z+O(z^{1-\frac{2}{2-\g}}),
\]
from which it follows that
\[
  U_{\infty}^{-1}E_{lin}^{-1}=\mqty{1 & 0\\ 0 & z^{\frac{1}{2-\g}}}\cdot \left(\mqty{\frac{1}{2-\g} & -\frac{4-3\g}{(2-\g)^2}k\\ 0 & \frac{1}{2-\g}}z^{-1}+O(z^{-1-\frac{2}{2-\g}})\right)
\]
and so for $z\geq 2$
\[
  \abs{U_{\infty}^{-1}E_{lin}^{-1}\mqty{P\\ \O}}\lesssim \frac{1}{z^\frac{3-\g}{2-\g}}\left(\sup_{z\geq 2}\abs{z^{\frac{1}{2-\g}} P}+\abs{z^{\frac{2}{2-\g}}\O} \right).
\]
Integrating this from $2$ to $\infty$ gives the existence of the desired limits and bounds as $z\to \infty$.

We similarly compute that as $z\to 0$, we have that
\[
  E_{lin}(z)=\mqty{z & 0\\ 0 & z^{1-\frac{2}{2-\g}}}\left(\mqty{2-\g & k\\ \frac{(2-\g)^2 y_f^{\frac{2}{2-\g}}}{k} & 0}+O(z^{\frac{2}{2-\g}})\right)
\]
so that
\[
  E_{lin}^{-1}(z)=\mqty{0 & 0\\ \frac{1}{k} & 0}\frac{1}{z}+O(z^{\frac{2}{2-\g}-1}),
\]
\[
  U_0^{-1}(z)E_{lin}^{-1}=\mqty{-\frac{\sin(\nu\log(z))}{k \sin(\theta_0)}&0\\\frac{\cos(\nu\log(z))}{k \sin(\theta_0)}&0}z^{\mu-1}+O(z^{\mu-1+\frac{2}{2-\g}}).
\]
This gives that for $z\leq \frac{1}{2}$
\[
  \abs{U_0^{-1}(z)E_{lin}^{-1}\mqty{P\\ \O}}\lesssim \left(z^{-\mu-1}\abs{z^{2\mu} P}+\abs{z^{2\mu+\frac{2}{2-\g}} \O}\right).
\]
Integrating from $y_0$ to $\frac12$ gives the desired limits and \eqref{eq:ext_linear_bound} (note this is where one picks up the factor of $y_0^{-\mu}$). A similar argument gives \eqref{eq:ext_linear_lipschitz}.

\subsection{Proof of Lemma \ref{ext_nonlinear_lemma}}
\label{ext_nonlinear_proof}
We first check that $p,\o$ do not vanish. Indeed, we see that if $(\ph,\oh)=(p_{hom},\o_{hom})+(p,\o)$ then
\[
  \norm{(\ph,\oh)}_{X_{y_0}}\leq C+\norm{(p_{hom},\o_{hom})}_{X_{y_0}}.
\]
Examining the weights for $y\leq y_0,y\geq y_0$ respectively, we see that this is the same as (up to a universal constant)
\[
  \abs{\ph}\leq \frac{\ev{y}^{\mu}}{y^{\mu}}(C+\norm{(p_{hom},\o_{hom})}_{X_{y_0}}),\qquad
  \abs{\oh}\leq \frac{\ev{y}^{\mu-\frac{1}{2-\g}}}{y^{\mu}}(C+\norm{(p_{hom},\o_{hom})}_{X_{y_0}}).
\]
In particular, bounding the weights by $y_0^{-\mu}$, we have that
\[
  \abs{\ep \ph},\abs{\ep \oh}\leq \ep y_0^{-\mu}(C+\norm{(p_{hom},\o_{hom})}_{X_{y_0}}).
\]
and hence by choosing $C_0$ sufficiently small, we can ensure
\begin{equation}
  \label{eq:massnonvanishing}
  \abs{\ep \ph}\leq \frac{1}{2}k,\qquad
  \abs{\ep \oh}\leq \frac{1}{2}(2-\g).  
\end{equation}
In particular, we see that $\pt=k+\ep \ph$ does not vanish.

We first comment that by the computations in Section \ref{sec:ext_diffeq}, if $p,\o\in X_{y_0}[\ep]$, then we have that $\frac{P-\frac{k}{2-\g}\O}{z-y_f}\in C^0([y_0,\infty])$, and the boundary conditions
\[
  P(y_f)=\sr_1[\ep],\qquad \O(y_f)=\sr_2[\ep],\qquad (P-\frac{k}{2-\g}\O)'(y_f)=\sr_3[\ep]
\]
are satisfied. One also obtains the proper boundedness of
\[
  \sup_{\frac12\leq z\leq 2}\abs{\frac{P-\frac{k}{2-\g}\O-(P(0)-\frac{k}{2-\g}\O(0))}{z-y_f}}.
\]

We now consider the regime in which $z\to \infty$. It is helpful at this point to use the fundamental theorem of calculus to rewrite \eqref{nhdef} as
\[
  \Nh(\ep,\ph,\oh)=\ep\left(\Eh_{e}(z,\ph,\oh,y_0,\ep)\mqty{\ph'\\ \oh'}+\Fh_e(z,\ph,\oh,y_0,\ep)\right)
\]
where
\[
  \Eh_e(z,a,b,y_0,\ep)=-\int_0^1\pdv{\ep}(S^T \Eh)|_(z,a,b,y_0,t\ep)dt,\qquad \Fh_e=-\int_0^1\pdv{\ep}(S^T \Fh)|_(z,a,b,y_0,t\ep)dt.
\]

We note that in this case, since $(p,\o)\in X_{y_0}$, $\norm{\ph,\oh}_{X_{y_0}}<\infty$, so there exists constants $p_0,p_1,\o_0,\o_1$ such that
\begin{equation}
  \label{eq:hasymptot1}
  \ph(z)=\hat{p}_{\infty,0}+\frac{\hat{p}_{\infty,1}}{z^{\frac{1}{2-\g}}}+o(\frac{1}{z^{\frac{1}{2-\g}}}),\qquad \oh(z)=\frac{\oh_{\infty,0}}{z^{\frac{1}{2-\g}}}+\frac{\oh_{\infty,1}}{z^{\frac{2}{2-\g}}}+o(\frac{1}{z^{\frac{2}{2-\g}}}).
\end{equation}
\begin{equation}
  \label{eq:hasymptot2}
  \ph'(z)=-\frac{1}{2-\g}\frac{\ph_{\infty,1}}{z^{\frac{1}{2-\g}+1}}+o(\frac{1}{z^{\frac{1}{2-\g}+1}}),\qquad \oh'(z)=-\frac{1}{2-\g}\frac{\oh_{\infty,0}}{z^{\frac{1}{2-\g}+1}}-\frac{2}{2-\g}\frac{\o_{\infty,1}}{z^{\frac{2}{2-\g}+1}}+o(\frac{1}{z^{\frac{2}{2-\g}+1}}).
\end{equation}
To do the difference estimates more cleanly, we introduce two symbol classes. In particular, we define
\begin{multline*}
  S^k=\{f(y,a,b)\in C^{\infty}([y_0,\infty)\times[-\frac{k}{2},\frac{k}{2}]\times[-\frac{2-\g}{2},\frac{2-\g}{2}])\colon\\
  \abs{\p_{y}^{l_1}\p_a^{l_2}\p_b^{l_3}f}\leq C_{l_1,l_2,l_3}\abs{y}^{k-l_1},\qquad \lim_{y\to\infty}y^{l_1-k}\p_{y}^{l_1}f\text{\qquad exists and depends smoothly $a,b$}\},
\end{multline*}

\begin{multline*}
  \Sul^k=\{f(z,a,b,\ep)\in C^{\infty}([y_0,\infty)\times[\frac{k}{2},2k]\times[\frac{2-\g}{2},2(2-\g)]\times [0,1])\colon\\
  \abs{\p_{z}^{l_1}\p_a^{l_2}\p_b^{l_3}\p_{\ep}^{l_4}f}\leq C_{l_1,l_2,l_3,l_4}\abs{z}^{k-l_1}\},\\
  \qquad \lim_{z\to\infty}z^{l_1-k}\p_{z}^{l_1}\p_a^{l_2}\p_b^{l_3}\p_{\ep}^{l_4}f\text{\qquad exists and depends smoothly $a,b,\ep$}\},
\end{multline*}

\begin{multline*}
  \Sh^k=\{f(z,a,b,\ep)\in C^{\infty}([y_0,\infty)\times[-\frac{k}{2\ep},\frac{k}{2\ep}]\times[-\frac{2-\g}{2\ep},\frac{2-\g}{2\ep}]\times [0,1])\colon\\
  \abs{\p_{z}^{l_1}\p_a^{l_2}\p_b^{l_3}\p_{\ep}^{l_4}f}\leq C_{l_1,l_2,l_3,l_4}\ep^{l_2+l_3}(1+\abs{a}+\abs{b})^{l_4}\abs{z}^{k-l_1}\},\\
  \qquad \lim_{z\to\infty}z^{l_1-k}\p_{z}^{l_1}\p_a^{l_2}\p_b^{l_3}\p_{\ep}^{l_4}f\text{\qquad exists and depends smoothly $a,b,\ep$}\}.
\end{multline*}
We abbreviate $O_{S}(y^k)$ for an error in $S^k$ (and analogously for $\Sul,\Sh$).
We then note that if $a,b$ are restricted to $(-\frac{k}{2},\frac{k}{2}),(-\frac{2-\g}{2},\frac{2-\g}{2})$, which will always be the case by \eqref{eq:massnonvanishing}, then
\[
  \Et(y,k+a,2-\g+b)=\mqty{O_{S}(y) & O_{S}(y)\\ O_{S}(y^{1-\frac{2}{2-\g}}) & (2-\g)y+b O_{S}(y)}.
\]
It then follows that
\[
  \Eul(z,k+a,2-\g+b,y_0,\ep)=\mqty{O_{\Sul}(z) & O_{\Sul}(z)\\ O_{\Sul}(z^{1-\frac{2}{2-\g}}) & (2-\g)(z-\be)+bO_{\Sul}(z))}.
\]
It then follows that
\[
  \Eh(z,a,b,y_0,\ep)=\mqty{O_{\Sh}(z) & O_{\Sh}(z)\\ O_{\Sh}(z^{1-\frac{2}{2-\g}}) & (2-\g)(z-\be)+bO_{\Sh}(z)}
\]
It then follows that
\[
  E_e(z,a,b,y_0,\ep)=\mqty{O_{\Sh}(z) & O_{\Sh}(z)\\ O_{\Sh}(z^{1-\frac{2}{2-\g}}) & \fr[\ep](2-\g)(z-\be)+bO_{\Sh}(z))}\cdot(1+\abs{a}+\abs{b})
\]
where $\fr[\ep]=-\frac{1}{\ep}(\frac{2-\g}{k}\frac{\pr[\ep]}{\omgr[\ep]}-1)$
(in fact there is an asymptotic expansion).

Similarly, we note that in the same regimes,
\[
  \Ft(y,k+a,2-\g+b)=\mqty{b O_{S}(1)\\ b+a O_{S}(y^{-\frac{2}{2-\g}})+bO_{S}(y^{-\frac{2}{2-\g}})+b^2O_{S}(1)}
\]
\[
  \Ful(z,k+a,2-\g+b,y_0,\ep)=\mqty{b O_{\Sh}(1)\\ b+a O_{\Sh}(z^{-\frac{2}{2-\g}})+b O_{\Sh}(z^{-\frac{2}{2-\g}})+b^2O_{\Sh}(1)}
\]
\[
  \Fh(z,a,b,y_0,\ep)=\mqty{b O_{\Sh}(1)\\ b+a O_{\Sh}(z^{-\frac{2}{2-\g}})+b O_{\Sh}(z^{-\frac{2}{2-\g}})+b^2O_{\Sh}(1)}.
\]
\[
  F_e(z,a,b,y_0,\ep)=\mqty{b O_{\Sh}(1)\\ \fr[\ep]b+a O_{\Sh}(z^{-\frac{2}{2-\g}})+b O_{\Sh}(z^{-\frac{2}{2-\g}})+b^2O_{\Sh}(1)}\cdot(1+\abs{a}+\abs{b})
\]
As a comment, from this we see that as $z\to \infty$, substituting \eqref{eq:hasymptot1}, \eqref{eq:hasymptot2}, since the contributions of $(2-\g)z\p_z \oh+\oh$ will cancel, we have that
\[
  \Nh=\ep \mqty{O(z^{-\frac{1}{2-\g}})\\ O(z^{-\frac{2}{2-\g}})}.
\]
From this we get the desired boundedness
\[
  \sup_{z\geq y_f}\abs{z^\frac{1}{2-\g}P}+\abs{z^\frac{2}{2-\g}\O}\leq C_2\abs{\ep}
\]
By considering the asymptotic expansions more carefully, one sees that the desired limits as $z\to \infty$ exist.

Similarly, we note that as $\abs{\ph},\abs{\oh}\lesssim y_0^{-\mu},\abs{\ph'},\abs{\oh'}\lesssim y_0^{-\mu-1}$ we have that
\[
  \sup_{y_0\leq z\leq y_f}\abs{z^{2\mu}P}+\abs{z^{2\mu+\frac{2}{2-\g}}\O}\leq C_2\abs{\ep}.
\]
This finishes the proof of \eqref{eq:ext_nonlinear_bound}. A straightforward application of the fundamental theorem of calculus and the symbol properties shows \eqref{eq:ext_nonlinear_lipschitz}.
\subsection{Proof of Lemma \ref{ext_contraction_lemma}}
\label{ext_contraction_proof}
\begin{proof}
  We begin by supposing
  \[
    \abs{\ep}\leq \min(\frac{y_0^{\mu}}{C_2(\frac{1}{2})\cdot C_1\cdot 10},y_0^{\mu}C_0(\frac{1}{2})).
  \]
  We then pick $(p_0,\o_0)\in X_{y_0}[\ep]$ so that $\norm{(p_0,\o_0)}_{X_{y_0}}\leq y_0^{-\mu}\cdot \ep\cdot C_2(\frac{1}{2})\cdot C_1$. We then define $(P_n,\O_n)=N(\ep,p_{n-1},\o_{n-1})$, and then
  \[
    (p_n,\o_n)=R(P_n,\O_n)+\left(\frac{\frac{2-\g}{k}\sr_1[\ep]-\sr_2[\ep]}{5(\g-1)}+\omgr_{0,e}[\ep]\right)\frac{1}{-(2-\g)}(p_{hom},\o_{hom}).
  \]
  We note that $(P_j,\O_j)\in N_{y_0}[\ep]$ and $(p_j,\o_j)\in X_{y_0}[\ep]$. We assume inductively that $\norm{(p_k,\o_k)}_{X_{y_0}}\leq y_0^{-\mu}\cdot \ep\cdot C_2(\frac{1}{2})\cdot C_1$. it then follows by choice of smallness of $\ep$ that $\norm{(p_k,\o_k)}_{X_{y_0}}\leq \frac{1}{2}$. It then follows from Lemma \ref{ext_nonlinear_lemma} that $\norm{(P_{k+1},\O_{k+1})}_{N_{y_0}}\leq \abs{\ep}C_2(\frac{1}{2})$. From Lemma \ref{ext_linear_lemma}, we obtain $\norm{(p_{k+1},\o_{k+1})}_{X_{y_0}}\leq  y_0^{-\mu}\cdot \abs{\ep}\cdot C_1 C_2(\frac{1}{2})$, completing the induction.

  In addition, for $k\geq 1$, we have that
  \begin{align*}
    \norm{(p_{k+1},\o_{k+1})-(p_{k},\o_{k})}_{X_{y_0}}&\leq \frac{1}{2}C_1 y_0^{-\mu}\abs{N(\ep,p_{k},\o_{k})-N(\ep,p_{k-1},\o_{k-1})}\\
                                                      &\leq \frac{1}{2}C_1 y_0^{-\mu} C_2(\frac{1}{2})\abs{\ep}\norm{(p_{k},\o_k)-(p_{k-1},\o_{k-1})}_{X_{y_0}}\\
                                                      &\leq \frac{1}{2}\norm{(p_{k},\o_{k})-(p_{k-1},\o_{k-1})}_{X_{y_0}}.
  \end{align*}
  Thus, we have a contraction, giving us a unique limit point, concluding the proof of existence and uniqueness with $C_3=C_1 C_2(\frac{1}{2})$.

  To prove \eqref{ext_lipschitz_1}, we suppose $(p[\ep_1],\o[\ep_1])$ and $(p[\ep_2],\o[\ep_2])$ each solve \eqref{ext_diffeq}. We then use \eqref{ext_integral_diffeq} and Lemma \ref{ext_linear_lemma} to estimate
  \[
    \norm{(p[\ep_1],\o[\ep_1])-(p[\ep_2],\o[\ep_2])}_{X_{y_0}}\leq \frac{1}{2}C_1 y_0^{-\mu}\norm{N(\ep_1,p[\ep_{1}],\o[\ep_1]])-N(\ep_2,p[\ep_{2}],\o[\ep_2]])}+\frac{1}{2}C_1 \abs{\ep_1-\ep_2}.
  \]
  We then use Lemma \eqref{ext_nonlinear_lemma} and smallness of $\ep$ to get
  \begin{align*}
    \norm{(p[\ep_1],\o[\ep_1])-(p[\ep_2],\o[\ep_2])}_{X_{y_0}}&\leq \frac{1}{2}C_1\cdot C_2(\frac{1}{2}) \abs{\ep} y_0^{-\mu}\cdot \norm{(p[\ep_1],\o[\ep_1])-(p[\ep_2],\o[\ep_2])}_{X_{y_0}}\\
                                                              &\qquad+\frac{2}{3}C_1C_2(\frac{1}{2}) y_0^{-\mu} \abs{\ep_1-\ep_2}\\
                                                              &\leq \frac{1}{10}\norm{(p[\ep_1],\o[\ep_1])-(p[\ep_2],\o[\ep_2])}_{X_{y_0}}+\frac{2}{3}C_1C_2(\frac{1}{2}) y_0^{-\mu} \abs{\ep_1-\ep_2}.
  \end{align*}
  We may then absorb the first term on the RHS into the LHS, and when we do this we obtain the desired estimate.
\end{proof}
\section{Deferred proofs of Section \ref{section:int}}
\label{sec:interior_proofs}
\subsection{Proof of Lemma \ref{int_asymptotics_lemma}}
\label{int_asymptotics_proof}
\begin{proof}
  We proceed analogously as in the proof of the corresponding lemma of \cite{sandine}. We note that in this case, we may rewrite $Hv=0$ as
  \[
    \Ht v=f
  \]
  where
  \[
    f:=\left({\frac{4\pi}{\g-1}(\frac{\g-1}{\g})^{\frac{1}{\g-1}}Q^{\frac{2-\g}{\g-1}}-\frac{2(4-3\g)}{(2-\g)^2}\frac{1}{y^2}}\right)v.
  \]
  Our basis of homogeneous solutions is then chosen to be
  \begin{equation*}
    \vt_1:=\frac{\cos(\nu \log(y))}{y^{\frac{1}{2}}},\qquad \vt_2:=\frac{\sin(\nu \log(y))}{y^{\frac{1}{2}}}.
  \end{equation*}
  The Wronskian ($\vt_1 \vt_2'-\vt_1'\vt_2$) is found to be $\frac{\nu}{y^2}$. A right-inverse to $\Ht$ is found to be
  \begin{equation*}
    \St_{\infty}(f):=\vt_1\int_y^{\infty}\frac{(y')^2}{\nu}\vt_2(y')\cdot f(y')dy'-\vt_2\int_{y}^{\infty}\frac{(y')^2}{\nu}\vt_1(y')\cdot f(y')dy'.
  \end{equation*}
  We then recall that by \eqref{eq:potentialasymptot} and $v_i=O(y^{2-\frac{2}{2-\g}-\mu})$, we have that $f=O(y^{-\frac{2}{2-\g}-2\mu})$ as $y\to \infty$. Thus, we have
  \[
    \St_{\infty}(f)=O(y^{2-\frac{2}{2-\g}-2\mu}).
  \]
  Since $v-\St_{\infty}\circ f\in \ker \Ht$, we obtain
  \[
    v=a_1\vt_1+a_2\vt_2+O(y^{2-\frac{2}{2-\g}-2\mu}).
  \]
  This gives the desired asymptotic expansions. The non-vanishing of the Wronskian $W$ ensures that $c_3,c_4\neq 0$ (and so they can be can to be positive). Moreover, it is clear that $v_1$ is smooth up to the origin. We then note that the equation \eqref{eq:v2def} ensures that $y v_2$ is smooth up to the origin. Combining this regularity with the asymptotics as $y\to \infty$ yields the pointwise bounds stated in the lemma.
\end{proof}
\subsection{Proof of Lemma \ref{int_linear_lemma}}
\label{int_linear_proof}
\begin{proof}
  We first note that by scaling \eqref{eq:Qpointwise} and \eqref{eq:ustarpointwise} we have that
  \begin{equation}
    \label{eq:scalebounds1}
    Q_\la^{-\frac{1}{\g-1}}\lesssim \la^{\frac{2}{2-\g}}\ev{\frac{y}{\la}}^{\frac{2}{2-\g}},\qquad \abs{u_\la}\lesssim \abs{y}\ev{\frac{y}{\la}}^{-\mu}
  \end{equation}
  and (since we may differentiate the asymptotics) for $k\geq 0$
  \begin{equation}
    \label{eq:scalebounds2}
    \abs{\p_y^k Q_\la^{\frac{1}{\g-1}}}\lesssim \la^{-\frac{2}{2-\g}-k}\ev{\frac{y}{\la}}^{-\frac{2}{2-\g}-k},\qquad \abs{\p_y^{k+1}u_\la}\lesssim \la^{-k}\ev{\frac{y}{\la}}^{-\mu-k}.
  \end{equation}
  Similarly, as a consequence of Lemma \ref{int_asymptotics_lemma} and the rescaling, we have that for $k\geq 0$
  \begin{equation*}
    \abs{\p_y^k v_{1,\la}}\lesssim \la^{-k}\ev{\frac{y}{\la}}^{-\frac{1}{2}-k},\qquad \abs{\p_y^k v_{2,\la}}\lesssim \la^{-k}(\frac{y}{\la})^{-1-k}\ev{\frac{y}{\la}}^{\frac{1}{2}}.
  \end{equation*}
  We will use these bounds repeately in the following.
  We first consider $S_\la$, and note that from the regularity of $v_i$ at the origin (Lemma \ref{int_asymptotics_lemma}) we see that if $f\in \calG$ then we have $S_\la f\in Z$. We will then use
  \begin{equation}
    \label{eq:zbound}
    \abs{f}\leq \ev{\frac{y}{\la}}^{-\mu}\norm{f}_{\calG}
  \end{equation}
  and the pointwise bounds in Lemma \ref{int_asymptotics_lemma} to estimate
  \begin{align*}
    \abs{f\cdot v_{2,\la}y^2}&\lesssim y^2\frac{\la}{y}\cdot \ev{\frac{y}{\la}}^{\frac{1}{2}}\abs{f}\\
                             &\lesssim y^2 \frac{\la}{y}\cdot \ev{\frac{y}{\la}}^{\frac{1}{2}-\mu}\norm{f}_{\calG}.
  \end{align*}
  Integrating from $0$ to $y$ then gives
  \begin{align*}
    \abs{\int_0^y f v_{2,\la}\cdot (y')^2 dy'}\lesssim y^3\frac{\la}{y}\cdot \ev{\frac{y}{\la}}^{\frac{1}{2}-\mu}\norm{f}_{\calG}.
  \end{align*}
  Combining this with the bound for $v_{1,\la}$ gives
  \begin{equation*}
    \abs{\la^{\frac{2(\g-1)}{2-\g}-1-\frac{2}{2-\g}}v_{1,\la}\int_0^y f v_{2,\la}\cdot (y')^2 dy'}\lesssim (\frac{y}{\la})^2\ev{\frac{y}{\la}}^{-\mu}\norm{f}_{\calG}.
  \end{equation*}
  Estimating the second term in the variation of constants formula similarly gives
  \begin{equation}
    \label{eq:Sbound1}
    \abs{S_\la f}\lesssim (\frac{y}{\la})^2\ev{\frac{y}{\la}}^{-\mu}\norm{f}_{\calG}.
  \end{equation}
  To bound the first derivative of $S_\la f$, we note that when we differentiate the variation of constants formulat there is a cancellation, and so we may express
  \begin{equation}
    \label{eq:Sbound2}
    (S_\la f)'=\frac{1}{\la^{1-\frac{2(\g-1)}{2-\g}+\frac{2}{2-\g}}}\left(\left(\int_0^y f v_{2,\la}\cdot (y')^2 dy'\right)v_{1,\la}'-\left(\int_0^y f v_{1,\la}\cdot (y')^2dy'\right)v_{2,\la}'\right)
  \end{equation}
  and then estimating the terms similarly yields
  \begin{equation*}
    \abs{(S_{\la}f)'}\lesssim \frac{y}{\la^2}\ev{\frac{y}{\la}}^{-\mu}\norm{f}_{\calG}.
  \end{equation*}
  We then use $H_\la\circ S_{\la}=\id$ to express
  \begin{equation*}
    (S_\la f)''=-\la^{-2}f-\frac{2}{y}(S_\la f)'-\frac{4\pi}{\g-1}(\frac{\g-1}{\g})^{\frac{1}{\g-1}}Q_\la^{\frac{2-\g}{\g-1}}(S_\la f).
  \end{equation*}
  Estimating these three terms using \eqref{eq:zbound}, \eqref{eq:Sbound2} and \eqref{eq:Sbound1} respectively gives the desired estimate
  \begin{equation*}
    \abs{(S_\la f)''}\lesssim \la^{-2}\ev{\frac{y}{\la}}^{-\mu}.
  \end{equation*}
  This yields the desired boundedness of $S_\la$. Using that, as a consequence of Lemma \ref{int_asymptotics_lemma}
  \begin{equation*}
    \abs{y\p_y v_1}\lesssim \ev{y}^{-\frac{1}{2}},\qquad \abs{y\p_y v_2}\lesssim \frac{1}{y}\ev{y}^{-\frac{1}{2}}
  \end{equation*}
  one can obtain the desired boundedness of $S_{\la_2}-S_{\la_1}$. We next consider $T_\la$. In this case, since $Q_\la$ is smooth and non-vanishing at the origin, it is clear that if $f\in \calF$ then $T_\la(f)\in Y$. We then estimate
  \begin{equation*}
    \abs{f}\lesssim \frac{y}{\la}\ev{\frac{y}{\la}}^{-1-\mu}\norm{f}_{\calF}
  \end{equation*}
  so that
  \begin{equation*}
    \abs{\int_0^y f(y')\cdot(y')^2 dy'}\lesssim y^3\frac{y}{\la}\ev{\frac{y}{\la}}^{-1-\mu}\norm{f}_{\calF}
  \end{equation*}
  We then use that
  \begin{equation*}
    \abs{Q_\la^{-\frac{1}{\g-1}}}\lesssim \la^{\frac{2}{2-\g}}\ev{\frac{y}{\la}}^{\frac{2}{2-\g}}
  \end{equation*}
  to obtain the desired bound
  \begin{equation*}
    \abs{T_\la f}\lesssim (\frac{y}{\la})^{2}\ev{\frac{y}{\la}}^{-1-\mu}\norm{f}_{\calF}.
  \end{equation*}
  Differentiating $T_\la$ allows one to similarly conclude
  \begin{equation*}
    \abs{(T_\la f)'}\lesssim \frac{y}{\la^2}\ev{\frac{y}{\la}}^{\frac{2}{2-\g}-1-\mu}\norm{f}_{\calF}.
  \end{equation*}
  One can then use
  \begin{equation*}
    \abs{f'}\lesssim \frac{1}{\la}\ev{\frac{y}{\la}}^{\frac{2}{2-\g}-1-\mu}\norm{f}_{\calF}
  \end{equation*}
  to conclude the desired estimate
  \begin{equation*}
    \abs{(T_\la f)''}\lesssim \frac{1}{\la^2}\ev{\frac{y}{\la}}^{\frac{2}{2-\g}-1-\mu}\norm{f}_{\calF}.
  \end{equation*}
  One then uses a straightforward argument, along with the estimate
  \begin{equation*}
    \abs{y\p_y Q}\lesssim \ev{y}^{-\frac{2(\g-1)}{2-\g}}
  \end{equation*}
  to obtain the desired boundedness of $T_{\la_2}-T_{\la_1}$.

  Finally, to show boundedness of the operator $K_{\la}$ we consider the operator $L_\la$ and note that it maps $Z$ to $\calF$. We then note that using
  \begin{equation*}
    \abs{f}\lesssim \frac{y^2}{\la^2}\ev{\frac{y}{\la}}^{-\mu}\norm{f}_{Z_\la},\qquad \abs{f'}\lesssim \frac{y}{\la^2}\ev{\frac{y}{\la}}^{-\mu}\norm{f}_{Z_\la}
  \end{equation*}
  along with the asymptotics for $Q_\la,Q_\la',u_\la,\p_y u_\la$, we obtain
  \begin{align*}
    \abs{L_\la f}&\lesssim (\frac{y}{\la})^2\ev{\frac{y}{\la}}^{-2-\mu}\norm{f}_{Z_\la}\\
                 &\lesssim (\frac{y}{\la})\ev{\frac{y}{\la}}^{-1-\mu}\norm{f}_{Z_\la}.
  \end{align*}
  Similarly, if one differentiates $L_\la f$ and uses
  \begin{equation*}
    \abs{f''}\lesssim \frac{1}{\la^2}\ev{\frac{y}{\la}}^{-\mu}\norm{f}_{Z_\la}
  \end{equation*}
  one obtains the desired estimate
  \begin{equation*}
    \abs{(L_\la f)'}\lesssim \frac{1}{\la}\ev{\frac{y}{\la}}^{-1-\mu}\norm{f}_{Z_\la}.
  \end{equation*}
  Once again, the bound for $L_{\la_2}-L_{\la_1}$ follows analogously, using that one may differentiate the asymptotics of $Q,u$. Combining the estimates for $S_\la$, $T_\la$ and $L_\la$ give the boundedness of $K_\la$ and $K_{\la_2}-K_{\la_1}$.
\end{proof}
\subsection{Proof of Lemma \ref{int_nonlinear_lemma}}
\label{int_nonlinear_proof}
\begin{proof}
  We first note that using the boundedness of $\norm{u}_{Y_\la}$ we have (using $\la\ev{\frac{y}{\la}}\leq 2\cdot y_0$ for the second inequality)
  \begin{equation*}
    \abs{\la u}\lesssim \abs{y}\ev{\frac{y}{\la}}^{\frac{2}{2-\g}-\mu}C,\qquad \abs{\la^{1+\frac{2}{2-\g}} u}\lesssim \abs{y}y_0^{\frac{2}{2-\g}}\ev{\frac{y}{\la}}^{-\mu}C.
  \end{equation*}
  Similarly, for $k\in \{0,1\}$ we have
  \begin{equation*}
    \abs{\p_y^{1+k}(\la u)}\lesssim \la^{-k}\ev{\frac{y}{\la}}^{\frac{2}{2-\g}-\mu-k}C,\qquad \abs{\p_y^{1+k}(\la^{1+\frac{2}{2-\g}} u)}\lesssim \la^{-k}y_0^{\frac{2}{2-\g}}\ev{\frac{y}{\la}}^{-\mu-k}C.
  \end{equation*}
  Similarly, using boundedness of $\norm{w}_{Z_\la}$, for $k\in \{0,1,2\}$ we have that
  \begin{equation*}
    \abs{\p_y^k w}\lesssim \la^{-k}(\frac{y}{\la})^{2-k}\ev{\frac{y}{\la}}^{-\mu}C,\qquad \abs{\p_y^k \la^{\frac{2}{2-\g}}w}\lesssim \la^{-k}(\frac{y}{\la})^{2-k}y_0^{\frac{2}{2-\g}}\ev{\frac{y}{\la}}^{-\frac{2}{2-\g}-\mu}C.
  \end{equation*}
  We then note that
  \begin{equation*}
    Q+t\la^{-\frac{2(\g-1)}{2-\g}+\frac{2}{2-\g}}w=\la^{-\frac{2(\g-1)}{2-\g}}\left(Q(\cdot/\la)+t \la^{\frac{2}{2-\g}}w\right).
  \end{equation*}
  To compare this to $Q$, we compute
  \begin{align*}
    \abs{\la^{\frac{2}{2-\g}} w}\lesssim (\frac{y}{\la})^2y_0^{\frac{2}{2-\g}}\ev{\frac{y}{\la}}^{-\frac{2}{2-\g}-\mu}\leq \ev{\frac{y}{\la}}^{\frac{-2(\g-1)}{2-\g}-\mu}y_0^{\frac{2}{2-\g}}C
  \end{align*}
  We note that since $y_0^{\frac{2}{2-\g}}C\ll 1$, we may use \eqref{eq:Qpointwise} and obtain that uniformly for $t\in [0,1]$,
  \begin{equation*}
    \la^{-\frac{2(\g-1)}{2-\g}}\ev{\frac{y}{\la}}^{-\frac{2(\g-1)}{2-\g}}\lesssim Q+t\la^{-\frac{2(\g-1)}{2-\g}+\frac{2}{2-\g}}w\lesssim \la^{-\frac{2(\g-1)}{2-\g}}\ev{\frac{y}{\la}}^{-\frac{2(\g-1)}{2-\g}}.
  \end{equation*}
  In particular, this quantity is always positive, so we may take powers of it, and differentiate, so  Z the nonlinearity is well-defined. We note that as a consequence we have that uniformly for $t\in [0,1]$,
  \begin{align*}
    \abs{(Q_\la+t\la^{-\frac{2(\g-1)}{2-\g}+\frac{2}{2-\g}}w)^{\frac{1}{\g-1}}}&\lesssim \la^{-\frac{2}{2-\g}}\ev{\frac{y}{\la}}^{-\frac{2}{2-\g}},\\
    \la^{-\frac{2(\g-1)}{2-\g}}\abs{(Q_\la+t\la^{-\frac{2(\g-1)}{2-\g}+\frac{2}{2-\g}}w)^{\frac{2-\g}{\g-1}}}&\lesssim \la^{-\frac{2}{2-\g}}\ev{\frac{y}{\la}}^{-2},\\
    \la^{-\frac{4(\g-1)}{2-\g}}\abs{(Q_\la+t\la^{-\frac{2(\g-1)}{2-\g}+\frac{2}{2-\g}}w)^{\frac{3-2\g}{\g-1}}}&\lesssim \la^{-\frac{2}{2-\g}}\ev{\frac{y}{\la}}^{-\frac{2(3-2\g)}{2-\g}}.
  \end{align*}
  From these bounds, along with \eqref{eq:scalebounds1} and \eqref{eq:scalebounds2} one sees directly that
  \begin{align*}
    \abs{G_{\la,1}}&\lesssim \ev{\frac{y}{\la}}^{-\mu}(1+y_0^{\frac{2}{2-\g}}C)^2,\\
    \abs{G_{\la,2}}&\lesssim y_0^{\frac{2}{2-\g}} \ev{\frac{y}{\la}}^{-2\mu}C^2,\\
    \abs{F_\la}&\lesssim y_0^{\frac{2}{2-\g}}(\frac{y}{\la})\ev{\frac{y}{\la}}^{-1-2\mu}C^2.
  \end{align*}
  We will then use $y_0^{\frac{2}{2-\g}}C^2\lesssim 1$, $\ev{\frac{y}{\la}}^{-2\mu}\leq \ev{\frac{y}{\la}}^{-\mu}$ to obtain the estimate
  \begin{equation*}
    \abs{G_{\la}}\lesssim \ev{\frac{y}{\la}}^{-\mu},\qquad \abs{F_\la}\lesssim \frac{y}{\la}\ev{\frac{y}{\la}}^{-1-\mu}.
  \end{equation*}
  One can similarly estimate
  \begin{equation*}
    \abs{F_\la'}\lesssim \frac{1}{\la}\ev{\frac{y}{\la}}^{-1-\mu}.
  \end{equation*}
  which shows the desired boundedness of the nonlinearities.
  We may similarly estimate
  \begin{align*}
    \abs{G_{\la,1}(w_2,u_2)-G_{\la,1}(w_1,u_1)}&\lesssim y_0^{\frac{2}{2-\g}}\ev{\frac{y}{\la}}^{-\mu}(\norm{u_2-u_1}_{Y_{\la}}+\norm{w_2-w_1}_{Z_\la}),\\
    \abs{G_{\la,2}(w_2,u_2)-G_{\la,2}(w_1,u_1)}&\lesssim C y_0^{\frac{2}{2-\g}}\ev{\frac{y}{\la}}^{-\mu}(\norm{u_2-u_1}_{Y_{\la}}+\norm{w_2-w_1}_{Z_\la}),\\
    \abs{F_{\la,2}(w_2,u_2)-F_{\la,2}(w_1,u_1)}&\lesssim C y_0^{\frac{2}{2-\g}}\frac{y}{\la}\ev{\frac{y}{\la}}^{-1-\mu}(\norm{u_2-u_1}_{Y_{\la}}+\norm{w_2-w_1}_{Z_\la}),\\
    \abs{F_{\la,2}(w_2,u_2)'-F_{\la,2}(w_1,u_1)'}&\lesssim C y_0^{\frac{2}{2-\g}}\frac{1}{\la}\ev{\frac{y}{\la}}^{-1-\mu}(\norm{u_2-u_1}_{Y_{\la}}+\norm{w_2-w_1}_{Z_\la}).
  \end{align*}
  Summing these gives the desired Lipschitz continuity with respect to $w,u$. Finally, the continuity with respect to $\la$ follows straightforwardly. 
\end{proof}
\subsection{Proof of Lemma \ref{int_contraction_lemma}}
\label{int_contraction_proof}
\begin{proof}
  We consider the integral formuation \eqref{eq:int_integral_formulation}. We let $(u_0,w_0)=(0,0)$. We then inductively define that for $n\geq 1$,
  \begin{align*}
    (W_n,U_n)&:=(G_{\la}(w_{n-1},u_{n-1}),F_{\la}(w_{n-1},u_{n-1})),\\
    (w_n,u_n)&:=(S_\la(W_n),T_\la(U_n)+K_\la(W_n)).
  \end{align*}
  We note that $\norm{w_1}_{Z_{\la}},\norm{u_1}_{Y_\la}\leq C_1\cdot C_2$. We then let $\Ct=2\max(C_1C_2,1)$. We shall make the assumption
  \begin{equation}
    \label{eq:y0small}
    y_0^{\frac{2}{2-\g}}\Ct^2 \ll 1.
  \end{equation}
  We then will inductively assume that for $k\geq 1$,
  \begin{equation*}
    \norm{w_k}_{Z_\la},\qquad \norm{w_{k-1}}_{Z_\la},\qquad \norm{u_k}_{Z_\la},\qquad \norm{u_{k-1}}_{Z_\la}\leq \Ct.
  \end{equation*}
  It then follows that
  \begin{align*}
    \norm{w_{k+1}-w_k}_{Z_\la}+\norm{u_{k+1}-u_k}_{Y_\la}&\leq 2 C_1\left(\norm{W_{k+1}-W_k}_{\calG_\la}+\norm{U_{k+1}-U_k}_{\calF_\la}\right)\\
                                                         &\leq 2 C_1C_2(1+\Ct)y_0^{\frac{2}{2-\g}}(\norm{w_{k}-w_{k-1}}_{Z_\la}+\norm{u_{k}-u_{k-1}}_{Y_\la}).
  \end{align*}
  By \eqref{eq:y0small}, we can conclude
  \begin{equation*}
    \norm{w_{k+1}-w_k}_{Z_\la}+\norm{u_{k+1}-u_k}_{Y_\la}\leq \frac{1}{10}\left(\norm{w_{k}-w_{k-1}}_{Z_\la}+\norm{u_{k}-u_{k-1}}_{Y_\la}\right)
  \end{equation*}
  and so we see that the sequence converges. We may choose $C_3=100 \Ct$. The uniqueness of the fixed point follows as we have constructed a contraction. For the Lipschitz dependence, we suppose that for $\in \{1,2\}$
  \begin{equation*}
    w[\la_i]=S_{\la_i}(G_{\la_i}(w[\la_i],u[\la_i])).
  \end{equation*}
  Subtracting these equations gives
  \begin{multline*}
    w[\la_2]-w[\la_1]=(S_{\la_2}-S_{\la_1}))G_{\la_2}(w[\la_2],u[\la_2])\\
    +S_{\la_1}(G_{\la_2}(w[\la_2],u[\la_2])-G_{\la_1}(w[\la_2,u[\la_2]]))+S_{\la_1}(G_{\la_1}(w[\la_2],u[\la_2])-G_{\la_1}(w[\la_1],u[\la_1])).
  \end{multline*}
  We then use the dependencies on $\la$ in the bounds of Lemmas \ref{int_linear_lemma} and \ref{int_nonlinear_lemma} to estimate the first two terms, and then difference estimate \eqref{eq:int_lipschitz_ineq} to bound the third term, which gives
  \begin{multline*}
    \max_i\norm{w[\la_2]-w[\la_1]}_{Z_{\la_i}}\leq \frac{\la_2-\la_1}{\la_1}C_1C_2+\frac{\la_2-\la_1}{\la_1}C_1C_2\\
    +C_1 C_2 y_0^{\frac{2}{2-\g}}\left(\min_i \norm{w[\la_2]-w[\la_1]}_{Z_{\la_i}}+\min_i \norm{u[\la_2]-u[\la_1]}_{Y_{\la_i}}\right).
  \end{multline*}
  A similar bound holds for $u[\la_2]-u[\la_1]$, and then one may use \eqref{eq:y0small} to absorb the last term into the left-hand side, which gives \eqref{eq:contraction_lipschitz}.
\end{proof}
\appendix
\section{}
\label{sec:appendix}
\subsection{Proof of Lemma \ref{lem:local_C1_wellposedness}}
\label{sec:local_C1_lemma}
\begin{proof}
  The key idea is that if $\kappa>-1$ then $\dv{t}+\frac{\kappa}{t}$ is ivertible as a mapping from $\{f\in C^1\colon f(0)=0\}$ into $C^0$.
  
  We consider the ODE system
  \begin{equation}
    \label{eq:system_C1_main_eq}
    \begin{cases}
      A(t,u_1,u_2)\mqty{u_1'\\u_2'}+B(t,u_1,u_2)=0\\
      u_1(0)=u_2(0)=u_1'(0)=0,\qquad u_2'(0)=U
    \end{cases},
  \end{equation}
  with $A,B$ as in the statement of the lemma.
  We first note that if we multiply the system on the left by
  \begin{equation*}
    \mqty{\frac{1}{A_{11}} & 0\\ -\frac{A_{21}}{A_{11}} & 1}
  \end{equation*}
  we obtain the system
  \begin{equation*}
    \mqty{1 & \frac{A_{12}}{A_{11}}\\ 0 & A_{22}-\frac{A_{12}A_{21}}{A_{11}}}\cdot \mqty{u_1'\\u_2'}+\mqty{\frac{B_1}{A_{11}}\\ B_2-\frac{A_{21} B_1}{A_{11}}}=0.
  \end{equation*}
  then we note that this system is of the same form (and $a,b,c,d$ do not change). Thus, without loss of generality, we may assume that
  \begin{equation*}
    A_{11}=1,\qquad A_{21}=0.
  \end{equation*}
  so that the system takes the form
  \begin{equation}
    \label{eq:simplifiedform}
    \mqty{1 & A_{12}\\0 & A_{22}}\mqty{u_1'\\u_2'}+\mqty{B_1\\B_2}=0.
  \end{equation}
  We will also define
  \begin{equation*}
    e:=\pdv[A_{22}]{z_1}(0),\qquad f:=\pdv[B_2]{z_1}(0).
  \end{equation*}
  In preparation for carrying out a contraction mapping argument, we will perform finite Taylor expansions of the coefficients in $z_0,z_1,z_2$. We will use multi-index notation, so that for $\alpha=(\alpha_0,\alpha_1,\alpha_2)$, $f(z_0,z_1,z_2)$,
  \begin{equation*}
    \abs{\alpha}=\sum_{i=0}^2 \alpha_i,\qquad z^{\alpha}=z_0^{\alpha_0}z_1^{\alpha_1}z_2^{\alpha_2},\qquad \p^{\alpha}f=\p_{z_0}^{\alpha_0}\p_{z_1}^{\alpha_1}\p_{\alpha_2}^{\alpha_2},\qquad \alpha!=\alpha_0!\alpha_1!\alpha_2!.
  \end{equation*}
  We will also adopt the convention that
  \begin{equation*}
    f_{\alpha}:=\frac{1}{\alpha!}(\p^{\alpha}f)(0).
  \end{equation*}
  We then recall that for any smooth $f(z_0,z_1,z_2)$, we may use the fundamental theorem of calculus to decompose the function to zero and first order as
  \begin{align*}
    f(z_0,z_1,z_2)&=f(0)+\sum_{\substack{\alpha\in \Z_{\geq 0}^3\\ \abs{\alpha}=1}}f_{\alpha,rem}(z)\cdot z^{\alpha}\\
                  &=\sum_{\substack{\alpha\in \Z_{\geq 0}^3\\ \abs{\alpha}\leq 1}}f_{\alpha} z^{\alpha}+\sum_{\substack{\alpha\in \Z_{\geq 0}^3\\ \abs{\alpha}=2}}f_{\alpha,rem}(z)\cdot z^{\alpha},
  \end{align*}
  where the remainder terms are smooth functions defined by
  \begin{equation*}
    f_{\alpha,rem}(z)=
    \begin{cases}
      \int_0^1\sum_{\abs{\alpha}=1}\p_{\alpha}f(t z)dt & \abs{\alpha}=1\\
      \int_0^12(1-t)\left(\sum_{\abs{\alpha}=2}\frac{\p_{\alpha}f(t z)}{\alpha!}\right)dt&\abs{\alpha}=2
    \end{cases}.
  \end{equation*}
  We will use the zeroth-order expansion for $A_{12},B_1$ to rewrite the line of the matrix equation \eqref{eq:simplifiedform} as the equation (using multi-index notation and $u_0=t$)
  \begin{equation*}
    \p_t u_1=-(\sum_{\abs{\alpha}=1}u^{\alpha}(A_{12})_{\alpha,rem})\p_t u_2-(\sum_{\abs{\alpha}=1}u^{\alpha}(B_{2})_{\alpha,rem}).
  \end{equation*}
  We will then use the second-order expansion for $A_{22},B_2$ to rewrite the second equation. We will also use the decompositions
  \begin{equation*}
    bu_2u_2'=(bUt)u_2'+bu_2 U+b(u_2-Ut)(u_2'-U)-bU^2t,\qquad eu_1 u_2'=e u_1(u_2'-U)+e u_1 U.
  \end{equation*}
  We may then write the second equation as
  \begin{multline*}
    (a+bU)tu_2'+(d+bU)u_2+(f+eU)u_1=(bU^2-c)t\\
    -b(u_2-Ut)(u_2'-U)-eu_1(u_2'-U)-\left(\sum_{\abs{\alpha}=2} u^{\alpha}(A_{22})_{\alpha,rem}\right)u_2'
    -\left(\sum_{\abs{\alpha}=2} u^{\alpha}(B_{2})_{\alpha,rem}\right).
  \end{multline*}
  We will then divide this equation by $a+bU$ and also use the algebraic identity
  \begin{equation*}
    \frac{bU^2-c}{a+bU}=(\alpha+1)U.
  \end{equation*}
  to rewrite the entire system as
  \begin{equation}
    \begin{cases}
      &u_1'=F_1(t,u)\\
      &tu_2'+\kappa u_2=-\frac{(f+eU)}{a+bU}u_1+(\alpha+1)Ut+F_2(t,u)\\
      &u_1(0)=u_1'(0)=u_2(0)=0,\qquad u_2'(0)=U.
    \end{cases}
  \end{equation}
  where
  \begin{align*}
    F_1(t,u_1,u_2)&=-(\sum_{\abs{\alpha}=1}u^{\alpha}(A_{12})_{\alpha,rem})\p_t u_2-(\sum_{\abs{\alpha}=1}u^{\alpha}(B_{2})_{\alpha,rem})\\
    F_2(t,u_1,u_2)&:=-\frac{b}{a+bU}(u_2-Ut)(u_2'-U)-\frac{e}{a+bU}u_1(u_2'-U)\\
                  &\qquad-\frac{1}{a+bU}\left(\sum_{\abs{\alpha}=2} u^{\alpha}(A_{22})_{\alpha,rem}\right)u_2'
                    -\frac{1}{a+bU}\left(\sum_{\abs{\alpha}=2} u^{\alpha}(B_{2})_{\alpha,rem}\right).
  \end{align*}
  We will now choose a large constant $D$ so that
  \begin{equation*}
    \abs{U}+1\leq D.
  \end{equation*}
  We next chose $R$ sufficiently larger than $D$ so that for $\g\in \Z_{\geq 0}^3$, $\abs{\g}\leq 3$, uniformly for $z$ in a neighbhorhood of zero,
  \begin{align*}
    \abs{(A_{12})_{\g}}&\leq \frac{1}{100}(\frac{R}{2})^{\abs{\g}}D^{-\g_2},\\
    \abs{(B_{1})_{\g}}&\leq \frac{1}{100}(\frac{R}{2})^{\abs{\g}}D^{1-\g_2},\\
    \abs{(A_{22})_{\g}}&\leq \frac{1}{100}\abs{a+bU}(\frac{R}{2})^{\abs{\g}}D^{-\g_2},\\
    \abs{(B_{2})_{\g}}&\leq \frac{1}{100}\abs{a+bU}(\frac{R}{2})^{\abs{\g}}D^{1-\g_2}.
  \end{align*}
  As a comment, this may be done since $A_{12}(0)=B_1(0)=A_{22}(0)=B_2(0)=0$, and so each term will have a positive power of $R$ or $\frac{R}{D}$. In particular, we have
  \begin{equation*}
    \abs{\frac{b}{a+bU}}\leq \frac{R}{D},\qquad \abs{\frac{e}{a+bU}}\leq R,\qquad \abs{\frac{f+eU}{a+bU}}\leq DR.
  \end{equation*}
  We will now define spaces (depending on $T$)
  \begin{equation*}
    \begin{cases}
      X_{U}:=\{f\in C^{1}(-T,T)\colon f(0)=0,f'(0)=U\},\qquad \norm{f}_X=\norm{f'}_{L^{\infty}}\\
      N_{(\alpha+1)U}:=\{f\in C^0(-T,T)\colon f(0)=(\alpha+1)U\},\qquad \norm{f}_{N}:=\norm{f}_{L^{\infty}}
    \end{cases}
  \end{equation*}
  We comment that if $\alpha>-1$, then $\dv{t}+\frac{\alpha}{t}$ is invertible $X_{U}\to X_{(\alpha+1)U}$ with inverse given by
  \begin{equation*}
    f\mapsto I_{\alpha}f:=t^{-\alpha}\int_{0}^t (t')^{\alpha}f(t')dt'.
  \end{equation*}
  Motivated by this, we will consider the integral formulation
  \begin{equation}
    \label{eq:integral_formulation}
    \begin{cases}
      u_1=\int_0^t F_1(t',u(t'))dt'\\
      u_2=I_{\alpha}\left(-\frac{f+eU}{a+bU}u_1+(\alpha+1)U t+F_2(t,u)\right)
    \end{cases}
  \end{equation}
  We will first describe boundedness of this iteration. Our goal will be to propogate
  \begin{equation*}
    \norm{u_2-Ut}_{X}\leq D\ep.
  \end{equation*}
  We first note that in orderto control the contribution of $\frac{-f+eU}{a+BU}u_1$, we will need the pointwise bound
  \begin{equation*}
    \norm{u_1'}_{L^{\infty}}\leq \frac{C_1}{R}\ep
  \end{equation*}
  for some universal small constant $C_1$ (depending on $\alpha$). Assuming these bounds on $u_1',u_2'-U$, we can bound
  \begin{equation*}
    \norm{F_1(t,u)}_{N}\lesssim TRD,\qquad \norm{\frac{F_2(t,u)}{t}}_{N}\lesssim RD\ep^2+TR^2D
  \end{equation*}
  and then choose $\ep$ sufficiently small so that 
  Thus, to close, we may choose $T$ sufficiently small so that
  \begin{equation*}
    TR^2D\ll \ep,
  \end{equation*}
  which allows us to close the bound on $u_1$ and then choose $\ep$ sufficiently small so that
  \begin{equation*}
    \ep R\ll 1
  \end{equation*}
  which allows us to close the bound on $u_2$. Similarly, if we have that for $u_a,u_b$ both in the ball $\norm{u_2-U}_{X}\leq D\ep$, $\norm{u_1}_X\leq \frac{C_1}{R}\ep$ and we have the difference estimate
  \begin{equation*}
    \norm{(u_a)_2-(u_b)_2}_X\leq \delta,\qquad \norm{(u_a)_1-(u_b)_1}_X\leq C_1 \frac{\delta}{D R},
  \end{equation*}
  we can estimate (replacing $u_a,u_b$ with the output of the the integral formulation \eqref{eq:integral_formulation})
  \begin{equation*}
    \norm{(u_a)_2-(u_b)_2}_X\leq \frac{\delta}{2},\qquad \norm{(u_a)_1-(u_b)_1}_X\leq \frac{1}{2}C_1 \frac{\delta}{D R}.
  \end{equation*}
  It then follows from a standard fixed point argument that the integral formulation \eqref{eq:integral_formulation} has a unique solution $(u_1,u_2)\in C^1(-T,T)$ such that
  \begin{equation}
    \label{eq:uniqueness_cond}
    \norm{u_1'}_{L^{\infty}}\leq \frac{\alpha+1}{100 R}\ep,\qquad \norm{u_2'-U t}_{L^{\infty}}\leq D \ep.
  \end{equation}
  We note that such a solution will be a $C^1$ solution to \eqref{eq:system_C1_main_eq} as desired, which gives the existence part of the lemma. Moreover, any solution to \eqref{eq:system_C1_main_eq} will solve the integral formulation \eqref{eq:integral_formulation}, and by continuity of the first derivative (by taking $T$ smaller if necessary) we can ensure \eqref{eq:uniqueness_cond}, which gives the uniqueness part of the theorem.
\end{proof}
\subsection{Proof of Lemma \ref{lem:main_analyticity_lemma}}
\label{sec:appendix_analyticity}
Before presenting the main proof of this section, we will first establish the following related lemma, which will be used in the proof. The key idea (for both this lemma and the main one of the subsection) is majorization, see for instance the treatment of the analytic implicit function theorem in \cite{krantz}. Moreover, dividing by $\kappa+i$ for successive integers $i$ will not affect the convergence, as long as this factor never vanishes. Our approach to dealing with this is inspired by the method of the proof of Theorem 4.5 in \cite{teschl}.
\begin{lem}{(Scalar, semilinear singular ODE)}
  \label{lem:scalar_singular_ode}
  Suppose $P(z_0,z_1)$ is an analytic function on a neighborhood of $0$ such that
  \begin{equation*}
    P(0)=\p_{z_1}P(0)=0.
  \end{equation*}
  Suppose $\kappa\notin \Z_{\leq -1}$. Then, for $T$ sufficiently small, the ODE
  \begin{equation*}
    \begin{cases}
      t w'+\kappa w=P(t,w)\\
      w(0)=0
    \end{cases}
  \end{equation*}
  has a unique real-analytic solution on $(-T,T)$.
\end{lem}
\begin{proof}
  We first show the existence of a formal Taylor series solution and then show that the Taylor series converges. We will use standard multi-index notation. In particular, if $\alpha=(\alpha_0,\alpha_1)\in \Z_{\geq 0}^{2}$ and $z=(z_0,z_1)\in \R^2$, then we define
  \begin{equation*}
    \alpha!:=\alpha_0!\alpha_1!,\qquad z^{\alpha}:=z_0^{\alpha_0}z_1^{\alpha_1},\qquad \p_z^{\alpha}:=\p_{z_0}^{\alpha_0}\p_{z_1}^{\alpha_1}.
  \end{equation*}
  We will also use the abbreviation that for $m\in \Z_{\geq 0}$ and multi-indices $\alpha\in \Z_{\geq 0}^2$,
  \begin{equation*}
    w_{m}:=\frac{w^{(m)}(0)}{m!},\qquad P_{\alpha}:=\frac{P^{(\alpha)}(0)}{\alpha!}.
  \end{equation*}
  Thus, we can write the (formal) Taylor expansion of $w$ and the (necessarily convergent) Taylor expansion of $P$ as
  \begin{equation*}
    w(t)=\sum_{m=1}^{\infty}w_m t^m,\qquad P(z_0,z_1)=\sum_{\substack{\alpha\in \Z_{\geq 0}^2\\\alpha_0+2\alpha_1\geq 2}} P_{\alpha}z^{\alpha}.
  \end{equation*}
  We then note that formally, by collecting terms with the same powers of $t$, we may express
  \begin{equation*}
    P(t,w(t))=\sum_{r=1}^{\infty}Q_r(P_{\alpha},w_m) t^r,
  \end{equation*}
  where for each fixed $r$, $Q_r$ can be written as, using the indexing $\beta=(\beta_1,\dots,\beta_M)$,
  \begin{equation*}
    Q_r(P_{\alpha},w_m)=\sum_{\substack{(\alpha_0,M)\in \Z_{\geq 0}^2\\\alpha_0+M=r}}\sum_{\alpha_1\leq M}P_{(\alpha_0,\alpha_1)}\sum_{\substack{\beta\in \Z_{\geq 0}^{M}\\ \abs{\beta}=\alpha_1\\ \sum_{i=1}^M i \beta_i=M}}c_{\alpha_0,\beta}w^{\beta},
  \end{equation*}
  where the coefficients $c_{\alpha_0,\beta}$ are positive combinatorial constants independent of $P,w$. We use here the notation that if $\beta\in \Z_{\geq 0}^{M}$,
  \begin{equation*}
    w^\beta:=w_1^{\beta_1}\cdots w_{M}^{\beta_M}.
  \end{equation*}
  We also comment, that since $P_{(0,1)}=0$, we have that $Q_r$ only depends on $w_j$ for $j\leq r-1$. We note that the equation specifies that for $m\geq 1$
  \begin{equation*}
    (\kappa+m)w_m=Q_m(P_{\alpha},w_m).
  \end{equation*}
  Since by assumption we have that for all $m\geq 0$, $\kappa+m\neq 0$, we can uniquely solve for each $w_m$. Thus, the formal power series is uniquely determined by the equation, and if it can be shown to converge on an open interval containing zero we have a unique real-analytic solution.

  To show convergence we will use the method of majorants. We shall let $R>0$ be a sufficiently large constant to be determined later. We now define
  \begin{equation*}
    \Pb(z_0,z_1)=\frac{1}{(1-R z_0)(1-R z_1)}-1-R z_1.
  \end{equation*}
  We note that the Taylor coefficients of $\Pb$ are all non-negative. We comment that as $P$ is analytic, by choosing $R$ sufficiently large, we can ensure it is majorized by $\Pb$, i.e.
  \begin{equation*}
    \abs{P_{\alpha}}\leq \Pb_\alpha.
  \end{equation*}
  We then note that the implicit equation
  \begin{equation}
    \label{eq:scalar_implicit_equation}
    V=\frac{1}{2}\Pb(t,V),\qquad V(0)=0
  \end{equation}
  is solved by
  \begin{equation*}
    \Vr(t):=\frac{2(1-R t)-\sqrt{4(1-R t)^2-4 R^2 t(2+R)(1-R t)}}{2R(2+R)(1-R t)}.
  \end{equation*}
  In particular, by Taylor expanding each side of \eqref{eq:scalar_implicit_equation}, we have that for $r\geq 1$ (from now on we notationally omit the ranges of the summation indices)
  \begin{equation}
    \label{eq:scalar_majorization_identity}
    \Vr_{r}=\frac{1}{2}Q_r(P_{\alpha},\Vr_m).
  \end{equation}
  We note that (by comparing the series solution with that of $V=\frac{1}{2}(\frac{1}{1-R t}-1)$)
  \begin{equation}
    \label{eq:recursion_lower_bound}
    \Vr_{r}\geq \frac{1}{2}R^r.
  \end{equation}
  We will now prove inductively that for $R$ sufficiently large, for $k\geq 1$,
  \begin{equation*}
    w_k\leq \frac{1}{k}\Vr_{k}.
  \end{equation*}
  Indeed, as a first step, we note that there exists $K$ sufficiently large such that for all $k\geq K$,
  \begin{equation}
    \label{eq:k_large}
    \frac{1}{\abs{1+\frac{\kappa}{k}}}\leq 2.
  \end{equation}
  We then note that by the lower bound \eqref{eq:recursion_lower_bound}, by choosing $R$ sufficiently large, we may assume without loss of generality that for all $k< K$, $\abs{w_k}\leq \frac{\Vr_{k}}{k}$.
  
  For $k\geq K$, we will work inductively. In particular, we consider the equation
  \begin{equation*}
    t w'+\kappa w=P(t,w).
  \end{equation*}
  and expand each side to order $k$. This gives
  \begin{equation*}
    (k+\kappa)w_k=Q_k(P_\alpha,w_m).
  \end{equation*}
  where the RHS only depends on $w_r$ for $r\leq k-1$. We next take the absolute value of each side. We estimate the RHS using the triangle inequality, the fact that the constants $c_{\alpha_0,\beta}$ are positive and also (the second following from the inductive hypothesis)
  \begin{equation*}
    \abs{P_{\alpha}}\leq \mathring{P}_{\alpha},\qquad \abs{w_m}\leq \Vr_m.
  \end{equation*}
  This allows us to estimate
  \begin{align*}
    \abs{Q_{k}(P_{\alpha},w_m)}&\leq Q_{k}(\abs{P_{\alpha}},\abs{w_m})\\
                               &\leq Q_{k}(\Pb_{\alpha},\Vr_m)\\
                               &=\frac{1}{2}\Vr_k,
  \end{align*}
  with the last line coming from \eqref{eq:scalar_majorization_identity}.
  This gives
  \begin{equation*}
    \abs{w_k}\leq \frac{1}{2}\frac{1}{\abs{1+\frac{\kappa}{k}}}\frac{\Vr_{k}}{k}\leq \frac{\Vr_{k}}{k},
  \end{equation*}
  with the second inequality following from \eqref{eq:k_large}.
  Thus, it follows that $w$ is majorized by $\Vr$, and since the latter is analytic on a neighborhood of zero, it follows that $w$ is analytic on a neighborhood of $0$. 
\end{proof}

We now prove Lemma \ref{lem:main_analyticity_lemma}.
\begin{proof}
  As a comment, as in the proof of Lemma \ref{lem:local_C1_wellposedness}, we may assume without loss of generality that
  \begin{equation*}
    A_{11}=1,\qquad A_{21}=0.
  \end{equation*}
  Thus, we have the system
  \begin{equation*}
    \begin{cases}
      u_1'+A_{12}(t,u_1,u_2)u_2'+B_{1}(t,u_1,u_2)=0\\
      A_{22}(t,u_1,u_2)u_2'+B_2(t,u_1,u_2)=0
    \end{cases}.
  \end{equation*}
  We then switch variables from $u_1,u_2$ to $v_1,v_2$ which are defined by
  \begin{equation*}
    u_1=v_1,\qquad u_2=v_2+t U.
  \end{equation*}
  We then get the system
  \begin{equation*}
    \begin{cases}
      v_1'+\At_{12}(t,v_1,v_2)v_2'+\Bt_2(t,v_1,v_2)=0\\
      \At_{22}(t,v_1,v_2)v_2'+\Bt_{2}(t,v_1,v_2)=0
    \end{cases}.
  \end{equation*}
  where (we also divide the second equation by $a+bU$)
  \begin{align*}
    \At_{12}(z_0,z_1,z_2)&:=A_{12}(z_0,z_1,z_2+U z_0),\\
    \Bt_{1}(z_0,z_1,z_2)&:=B_{1}(z_0,z_1,z_2+U z_0)+U\At_{12}(z_0,z_1,z_2),\\
    \At_{22}(z_0,z_1,z_2)&:=\frac{1}{a+bU}A_{22}(z_0,z_1,z_2+U z_0),\\
    \Bt_{2}(z_0,z_1,z_2)&:=\frac{1}{a+bU}B_2(z_0,z_1,z_2+U z_0)+U \At_{22}(z_0,z_1,z_2).
  \end{align*}
  We note that (using \eqref{eq:characteristic_quadratic} for the third equality) we have that
  \begin{equation*}
    \pdv[\At_{22}]{z_0}(0)=1,\qquad \pdv[\At_{22}]{z_2}=\frac{b}{a+bU},\qquad \pdv[\Bt_{2}]{z_0}=0,\qquad \pdv[\Bt_2]{z_2}=\kappa.
  \end{equation*}
  We may rewrite the second equation as
  \begin{align*}
    tv_2'+\kappa v_2&=-(\At_{22}-z_0)v_2'-(\Bt_2-\kappa z_2)\\
                    &=\frac{-(\At_{22}-z_0)|_{t,v}tv_2'-(z_0 \Bt_2-\kappa z_0z_2)|_{(t,v)}}{t}.
  \end{align*}
  In particular, if we define (introducing a new dummy variable $\zt_2$ for $t v_2'$)
  \begin{align*}
    H(z_0,z_1,z_2,\zt_2)&:=-(\At_{22}-z_0)\zt_2-(z_0 \Bt_{2}-\kappa z_0z_2),\\
    S(z_0,z_1,z_2):&=-\Bt_1(z_0,z_1,z_2),\\
    T(z_0,z_1,z_2):&=-\At_{12}(z_0,z_1,z_2),
  \end{align*}
  then we can write the system as
  \begin{equation}
    \label{eq:singular_sys_normal}
    \begin{cases}
      v_1'&=S(t,v_1,v_2)+T(t,v_1,v_2)v_2'\\
      tv_2'+\kappa v_2&=\frac{H(t,v_1,v_2,tv_2')}{t}
    \end{cases}.
  \end{equation}
  We first examine the first equation. We recall that $S(0)=T(0)=0$ and further that we can Taylor expand
  \begin{align*}
    S(t,v_1,v_2)(t,z_1,z_2)&=\sum_{\gamma\in \Z_{\geq 0}^{3}} S_{\gamma} t^{\g_0}(v_1)^{\g_1}(v_2)^{\g_2}\\
                           &=\sum_{r=1}^{\infty}P_r(S_\g,(v_1)_m,(v_2)_m) t^r,
  \end{align*}
  where
  \begin{multline*}
    P_r(S_{\g},(v_1)_m,(v_2)_m)\\
    =\sum_{\g_0+M_1+M_2=r}\sum_{\g_1\leq M_1}\sum_{\g_2\leq M_2}S_{\g}\sum_{\substack{\beta_a\in \Z_{\geq 0}^{M_1}\\\abs{\beta_a}=\g_1\\ \sum_{i=1}^{M_a} i(\beta_a)_i=M_1}}\sum_{\substack{\beta_b\in \Z_{\geq 0}^{M_2}\\\abs{\beta_b}=\g_2\\ \sum_{i=1}^{M_2} i(\beta_b)_i=M_2}} c_{\g_0,\beta_a,\beta_b}(v_1)^{\beta_a}(v_2)^{\beta_b}.
  \end{multline*}
  We use here the notation that, for instance, if $\beta_a=((\beta_a)_1,\dots,(\beta_a)_M)$ then $(v_1)^{\beta_a}=(v_1)_1^{(\beta_a)_1}\cdots (v_1)_M^{(\beta_a)_M}$. The constants $c_{\g_0,\beta_a,\beta_b}$ are combinatorial in nature and necessarily positive.
  We comment that $P_r$ only depends on $(v_1)_j$ and $(v_2)_j$ such that $j\leq r$. We may similarly expand the term $T(t,v_1,v_2)v_2'$, and we see that we may expand the RHS of the first equation as
  \begin{multline*}
    S(t,v_1,v_2)+T(t,v_1,v_2)v_2'\\
    =\sum_{r=1}^{\infty}\left(P_r(S_{\g},(v_1)_m,(v_2)_m)+\sum_{l=1}^{r}{r\choose l}(r-l+1)P_r(T_{\g},(v_1)_m,(v_2)_m)\cdot (v_1)_{r-l+1}\right)t^r
  \end{multline*}
  and we see that the coefficient of $t^r$ only depends on $(v_1)_j$ and $(v_2)_j$ for $j\leq r$. 
  
  We comment that the numerator of the second equation satisfies
  \begin{equation*}
    H(0)=\p_{z_0}H(0)=\p_{z_1}H(0)=\p_{z_2}H(0)=\p_{\zt_2}H(0)=0,
  \end{equation*}
  \begin{equation*}
    \p_{z_0}^2 H(0)=\p_{z_0}\p_{z_2}H(0)=\p_{z_0}\p_{\zt_2}H(0)=0.
  \end{equation*}
  We note that if we denote $\g=(\g_0,\g_1,\g_2,\g_3)$, we may Taylor expand
  \begin{align*}
    H(t,v_1,v_2,tv_2')&=\sum_{\g\in \Z_{\geq 0}^4}H_{\g}t^{\g_0}v_1^{\g_1}v_2^{\g_2}(t v_2')^{\g_3}\\
                      &=\sum_{r=3}^{\infty} Q_{r}(H_{\g},(v_1)_m,(v_2)_m) t^r,
  \end{align*}
  where
  \begin{multline*}
    Q_r(H_{\g},(v_1)_m,(v_2)_m)=
    \sum_{\g_0+M_1+M_2+M_3=r}\sum_{\g_1\leq M_1}\sum_{\g_2\leq M_2}\sum_{\g_3\leq M_3}H_{\g}\\\cdot\sum_{\substack{\beta_a\in \Z_{\geq 0}^{M_1}\\\abs{\beta_a}=\g_1\\ \sum_{i=1}^{M_a} i(\beta_a)_i=M_1}}\sum_{\substack{\beta_b\in \Z_{\geq 0}^{M_2}\\\abs{\beta_b}=\g_2\\ \sum_{i=1}^{M_2} i(\beta_b)_i=M_2}}\sum_{\substack{\beta_c\in \Z_{\geq 0}^{M_3}\\\abs{\beta_c}=\g_3\\ \sum_{i=1}^{M_3} i(\beta_c)_i=M_3}}c_{\g_0,\beta_a,\beta_b,\beta_c}(v_1)^{\beta_a}(v_2)^{\beta_b}(tv_2')^{\beta_c}.
  \end{multline*}
  We also note that we may conclude
  \begin{equation*}
    Q_0(0)=Q_1(0)=Q_2(0)=0
  \end{equation*}
  using the fact that $v_1(0)=v_1'(0)=v_2(0)=v_2'(0)=0$ and also
  \begin{equation*}
    H(0)=\p_{z_0}H(0)=\p_{z_1}H(0)=\p_{z_2}H(0)=\p_{\zt_2}H(0)=\p_{z_0}^2H(0)=0.
  \end{equation*}
  Moreover, using $\p_{z_1}H(0)=0$, we see that $Q_r$ only depends on $(v_1)_j$ for $j\leq r-1$. Similarly, using the fact that
  \begin{equation*}
    \p_{z_2}H(0)=\p_{z_0}\p_{z_2}H(0)=\p_{\zt_2}H(0)=\p_{z_0}\p_{\zt_2}H(0)=0
  \end{equation*}
  and the first order vanishing of $v_1,v_2$, we see that $Q_r$ only depends on $(v_2)_j$ for $j\leq r-2$.
  We now return our attention to the system \eqref{eq:singular_sys_normal}. We note that if we differentiate the first equation $k-1$ times and evaluate at zero, by the above analyses, we can solve for $(v_1)_{k}$ in terms of $(v_1)_{j},(v_{2})_j$ for $j\leq k-1$. Then, differentiating the second equation $k$ times, we ay solve for $(v_2)_k$ in terms of $(v_1)_i,(v_2)_j$ for $i\leq k$, $j\leq k-1$. Thus, the Taylor series is uniquely determined by the equation. It suffices to prove that the Taylor series constructed this way converges.
  
  As a first step, motivatived by the fact that $v_2$ and $tv_2'$ will have Taylor coefficients differing only by a factor of $k$, we define
  \begin{equation*}
    \Ht(z_0,z_1,z_2)=H(z_0,z_1,z_2,z_2).
  \end{equation*}
  We note that this will have an expansion
  \begin{equation*}
    \Ht(t,v_1,v_2)=\sum_{r=3}^{\infty}\Qt_r(\Ht_{\g},(v_1)_m,(v_2)_m)t^r,
  \end{equation*}
  where
  \begin{multline*}
    \Qt_r(\Ht_{\g},(v_1)_m,(v_2)_m)=\\
    \sum_{\g_0+M_1+M_2=r}\sum_{\g_1\leq M_1}\sum_{\g_2\leq M_2}\Ht_{\g}\sum_{\substack{\beta_a\in \Z_{\geq 0}^{M_1}\\\abs{\beta_a}=\g_1\\ \sum_{i=1}^{M_a} i(\beta_a)_i=M_1}}\sum_{\substack{\beta_b\in \Z_{\geq 0}^{M_2}\\\abs{\beta_b}=\g_2\\ \sum_{i=1}^{M_2} i(\beta_b)_i=M_2}} \ct_{\g_0,\beta_a,\beta_b}(v_1)^{\beta_a}(v_2)^{\beta_b}.
  \end{multline*}
  Once again we have that the coefficients are positive, $Q_r$ only depends on $(v_1)_j$ for $j\leq r-1$ and $Q_r$ only depends on $j$ for $j\leq r-1$. We note that we have the inequality that for any $H_\g,v_1,v_2$, since $\abs{(v_2)_m}\leq \abs{(t v_2')_m}$,
  \begin{equation}
    \label{eq:system_analyticity_convert_inequality}
    Q_r(\abs{H_\g},\abs{(v_1)_m},\abs{(v_2)_m})\leq \Qt_r(\abs{\Ht_\g},\abs{(v_1)_m},\abs{(t v_2')_m}).
  \end{equation}
  We will now perform a majorization argument. We will let $D,R$ be sufficiently large constants to be determined later (in particular we will first pick $D$ to be sufficiently large, then pick $R$ large depending on $D$).
  We will pick these constants sufficiently large so that for $\g=(\g_0,\g_1,\g_2)\in \Z_{\geq 0}^3$,
  \begin{align*}
    \abs{(A_{12})_{\g}}&\leq \frac{1}{100}(\frac{R}{2})^{\abs{\g}}D^{-\g_2},\\
    \abs{(B_{1})_{\g}}&\leq \frac{1}{100}(\frac{R}{2})^{\abs{\g}}D^{1-\g_2},\\
    \abs{(A_{22})_{\g}}&\leq \frac{1}{100}\abs{a+bU}(\frac{R}{2})^{\abs{\g}}D^{-\g_2},\\
    \abs{(B_{2})_{\g}}&\leq \frac{1}{100}\abs{a+bU}(\frac{R}{2})^{\abs{\g}}D^{1-\g_2}.
  \end{align*}
  When we change variables from $u$ to $v$, and assume $\abs{U}+1\leq D$, this means that the coefficient functions satisfy
  \begin{align*}
    \abs{(\At_{12})_\g}&\leq R^{\abs{\g}}D^{-\g_2},\\
    \abs{(\Bt_1)_\g}&\leq R^{\abs{\g}}D^{1-\g_2},\\
    \abs{(\At_{22})_\g}&\leq \frac{1}{2}R^{\abs{\g}}D^{-\g_2},\\
    \abs{(\Bt_2)_\g}&\leq \frac{1}{2}R^{\abs{\g}}D^{1-\g_2}.
  \end{align*}
  From these, bounds, we obtain
  \begin{equation*}
    \abs{\Ht_{\g}}\leq R^{\abs{\g}-1}D^{1-\g_2-\g_3},\qquad \abs{S_{\g}}, \abs{T_{\g}}\leq R^{\abs{\g}}D^{1-\g_2}.
  \end{equation*}
  We then define
  \begin{align*}
    \Hb(z_0,z_1,z_2)&:=\frac{D}{R}\left(\frac{1}{(1-R z_0)(1-R z_1)(1-\frac{R}{D}z_2)}-\left(1+R z_0+R z_1+\frac{R}{D}z_2+R^2z_0^2+\frac{R^2}{D}z_0z_2\right)\right),\\
    \Sb(z_0,z_1,z_2)&:=D\left(\frac{1}{(1-R z_0)(1-R z_1)(1-\frac{R}{D} z_2)}-1\right).
  \end{align*}
  We comment that these functions majorize the terms appearing in our equation in the sense that
  \begin{equation}
    \label{eq:system_analyticity_coefficient_bound}
    \abs{\Ht_{\g}}\leq \Hb_{\g},\qquad \abs{S_{\g}}\leq \Sb_{\g},\qquad \abs{T_{\g}}\leq \Sb.
  \end{equation}
  We now consider the ODE
  \begin{equation}
    \label{eq:system_analytic_majorizing}
    \begin{cases}
      v_1'&=\Sb(t,v_1,v_2)(1+v_2')\\
      v_2&=\frac{1}{2}\frac{\Hb(t,v_1,v_2)}{t}\\
      v_1(0)&=v_1'(0)=v_2(0)=v_2'(0)=0
    \end{cases}.
  \end{equation}
  We will use an explicit solution to the ODE system to majorize the system \eqref{eq:singular_sys_normal}. To start with, we comment that if $t,v_1$ are fixed, the second equation can be viewed as a quadratic for $v_2$. In particular, if we define
  \begin{align*}
    A_2&:=\frac{1}{D^2}\left(2 R^2t+R^2+R^3t \right)(1-Rt)(1-R v_1),\\
    A_1&:=\frac{1}{D}\left(-2 Rt+R(R v_1+R^2 t^2)\right)(1-R t)(1-R v_1))),\\
    A_0&:=1-(1+R t+R v_1+R^2 t^2)(1-R t)(1-R v_1).
  \end{align*}
  then the equation becomes
  \begin{equation*}
    A_2 v_2^2+A_1 v_2+A_0=0.
  \end{equation*}
  After simplifying, this quadratic can be rewritten as
  \begin{equation*}
    \At_2(\frac{R}{D t}v_2)^2+\At_1(\frac{R}{D t}v_2)+\At_0=0,
  \end{equation*}
  where
  \begin{align*}
    \At_2&:=(2 t+1+R t)(1-R t)(1-R v_1),\\
    \At_1&:=\left(-2+R\frac{v_1}{t}+R^2 t\right)(1-Rt)(1-R v_1),\\
    \At_0&:=R^2(\frac{v_1}{t})^2(1-R t)+\frac{v_1}{t}(R^2-R^4 t)+R^3 t.
  \end{align*}
  We comment that these are each smooth functions of $t,\frac{v_1}{t}$. Introducing the dummy variables $y_0,y_1$ for $t,\frac{v_1}{t}$, we define
  \begin{align*}
    \Ar_2(y_0,y_1)&:=(2 y_0+1+R y_0)(1-R y_0)(1-R y_0 y_1),\\
    \Ar_1(y_0,y_1)&:=(-2+R y_1+R^2 y_0)(1-R y_0)(1-R y_0 y_1),\\
    \Ar_0(y_0,y_1)&:=R^2y_1^2(1-R y_0)+y_1(R^2-R^4 y_0^2)+R^3 y_0.
  \end{align*}
  Thus, if we define
  \begin{equation*}
    \mathring{Q}(y_0,y_1)=\frac{D}{R}\left(\frac{-\Ar_1-\sqrt{\Ar_1^2-4\Ar_0\Ar_2}}{2 \Ar_1}\right).
  \end{equation*}
  Then we can write the solution to the quadratic as
  \begin{equation*}
    v_2=y_0 \mathring{Q}(y_0,y_1)|_{t,\frac{v_1}{t}}.
  \end{equation*}
  We now define $w=\frac{v_1}{t}$, and we have that the second equation becomes
  \begin{equation*}
    v_2=t \Qr(t,w).
  \end{equation*}
  From this, when we differentiate in time, we get
  \begin{equation*}
    v_2'=\Qr(t,w)+t\p_{y_0}\Qr(t,w)+t\pdv[\Qr]{y_1}(t,w)w'.
  \end{equation*}
  If we substitute this into the first equation of \eqref{eq:system_analytic_majorizing}, we get, where $\Sb$ is evaluated at $(t,t w,t \Qr)$ and $\Qr$ is evaluated at $(t,w)$,
  \begin{equation*}
    t w'+w=\Sb(1+(\Qr+t\p_{z_0}\Qr))+t\Sb \pdv[\Qr]{z_1}w'.
  \end{equation*}
  We note that we can rewrite this as
  \begin{equation*}
    (1-\Sb\pdv[\Qr]{y_1})tw'+w=\Sb(1+(\Qr+t\p_{y_0}\Qr)).
  \end{equation*}
  We may then divide by the principal term, and write the resulting equation as
  \begin{equation*}
    tw'+w=-\frac{\Sb}{(1-\Sb \pdv[\Qr]{z_1})}w+\frac{\Sb}{(1-\Sb\pdv[\Qr]{z_1})}(1+(\Qr+t\p_{z_0}\Qr)).
  \end{equation*}
  This ODE has an analytic solution by Lemma \ref{lem:scalar_singular_ode}. We call the analytic solution to this ODE system $\Wr(t)$. It follows that if we define
  \begin{align*}
    \Vr_1(t)&:=t\cdot \Wr(t),\\
    \Vr_2(t)&:=t\Qr(t,\Wr(t)),
  \end{align*}
  then $\Vr_1(t)$ and $\Vr_2(t)$ solve \eqref{eq:system_analytic_majorizing}. In particular, we have the combinatorial identities
  \begin{align}
    \label{eq:system_analyticity_combinatorial_identity_1}
    (\Vr_1)_r&=P_r(\Sb_{\g},(\Vr_1)_m,(\Vr_2)_m)+\sum_{l=1}^{r}{r\choose l}(r-l+1)P_r(\Sb_\g,(\Vr_1)_m,(\Vr_2)_m)\cdot (\Vr_1)_{r-l+1},\\
    \label{eq:system_analyticity_combinatorial_identity_2}
    (\Vr_2)_r&=\frac{1}{2}\Qt_{r+1}(\Hb_\g,(\Vr_1)_m,(\Vr_2)_m).
  \end{align}
  We comment that since the right-hand-sides of the equations have non-negative Taylor series it follows that the Taylor coefficients of $\Vr_1,\Vr_2$ are non-negative. In fact, we see that
  \begin{align*}
    (\dvn{k}{t}\Sb(t,\Vr_1,\Vr_2))(0)&\geq \left(\dvn{k}{t} D\left(\frac{1}{1-R t}-1\right)\right)(0),\\
    (\dvn{k}{t}\frac{\Hb(t,\Vr_1,\Vr_2)}{t})(0)&\geq \left(\dvn{k}{t}\frac{D}{Rt} \left(\frac{1}{1-R t}-1-R t-R^2 t\right)\right)(0).
  \end{align*}
  from which we conclude the lower bounds that for $m\geq 2$,
  \begin{equation}
    \label{eq:system_analytic_coefficient_lower_bounds}
    (\Vr_1)_m\geq \frac{DR^{(m-1)}}{m},\qquad (\Vr_2)_m\geq \frac{1}{2}D R^{m}.
  \end{equation}
  We will now prove that by choosing $D,R$ sufficiently large, we have that
  \begin{equation}
    \label{eq:system_analytic_induction_goal}
    (v_1)_k\leq (\Vr_{1})_k,\qquad (v_{2})_{k}\leq \frac{1}{k}(\Vr_{2})_k.
  \end{equation}
  As a first step, we note that we may choose $K$ sufficiently large so that for all $k\geq K$,
  \begin{equation}
    \label{eq:system_k_large}
    \frac{1}{\abs{1+\frac{\kappa}{K}}}\leq 2.
  \end{equation}
  We then note that using the lower bound \eqref{eq:system_analytic_coefficient_lower_bounds} for the coefficients, by choosing $D$, $R$ sufficiently large, we may assume that \eqref{eq:system_analytic_induction_goal} holds for $k\leq K-1$.
  We now suppose $k\geq K$ and proceed inductively. We differentiate the first equation of \eqref{eq:singular_sys_normal} $k-1$ times and evaluate at zero. We have inductively that
  \begin{align*}
    \abs{(v_1)_k}&\leq\abs{P_k(S_{\g},(v_1)_m,(v_2)_m)+\sum_{l=1}^{k}{k\choose l}(k-l+1)P_k(T_{\g},(v_1)_m,(v_2)_m)\cdot (v_1)_{k-l+1}}\\
                 &\leq P_k(\abs{S_{\g}},\abs{(v_1)_m},\abs{(v_2)_m})+\sum_{l=1}^{k}{k\choose l}(k-l+1)P_k(\abs{T_{\g}},\abs{(v_1)_m},\abs{(v_2)_m})\cdot \abs{(v_1)_{k-l+1}}\\
                 &\leq P_k(\Sb_{\g},(\Vr_1)_m,(\Vr_2)_m)+\sum_{l=1}^{k}{k\choose l}(k-l+1)P_k(\Sb_\g,(\Vr_1)_m,(\Vr_2)_m)\cdot (\Vr_1)_{k-l+1},
  \end{align*}
  where to obtain the third line we have used the inductive hypothesis and the upper bounds for $\abs{S_{\g}}$ and $\abs{T_\g}$ provided in \eqref{eq:system_analyticity_coefficient_bound}.
  We then use \eqref{eq:system_analyticity_combinatorial_identity_1} to obtain the desired estimate
  \begin{equation*}
    \abs{(v_1)_k}\leq (\Vr_1)_k.
  \end{equation*}
  Similarly, differentiate the second equation of \eqref{eq:singular_sys_normal} $k$ times, evaluate at zero and divide by $k+\kappa$. This gives
  \begin{align*}
    \abs{(v_2)_k}
    &\leq \frac{1}{\abs{1+\frac{\kappa}{k}}}\frac{1}{2k}\abs{Q_{k+1}(H_{\g},(v_1)_m,(v_2)_m)},\\
    &\leq \frac{1}{\abs{1+\frac{\kappa}{k}}}\frac{1}{2k}\left( Q_{k+1}(\abs{H_{\g}},\abs{(v_1)_m},\abs{(v_2)_m})\right),\\
    &\leq \frac{1}{\abs{1+\frac{\kappa}{k}}}\frac{1}{2k}\left( \Qt_{k+1}(\abs{\Ht_{\g}},\abs{(v_1)_m},\abs{(t v_2')_m})\right),\\
    &\leq \frac{1}{\abs{1+\frac{\kappa}{k}}}\frac{1}{2k}\left( \Qt_{k+1}(\Hb_{\g},\abs{(\Vr_1)_m},\abs{(\Vr_2)_m})\right),
  \end{align*}
  where to obtain the third line we use the inequality \eqref{eq:system_analyticity_convert_inequality}, and to obtain the last line we use the inductive hypothesis and the upper bound for $\abs{H_{\g}}$ provided by the inequality \eqref{eq:system_analyticity_coefficient_bound}.
  We then use \eqref{eq:system_analyticity_combinatorial_identity_2} and \eqref{eq:system_k_large} to obtain the desired estimate
  \begin{equation*}
    \leq \frac{1}{2}\frac{1}{\abs{1+\frac{\kappa}{k}}} \frac{(\Vr_2)_k}{k}\leq \frac{(\Vr_2)_k}{k}.
  \end{equation*}
  Thus, we have established \eqref{eq:system_analytic_induction_goal}, and since the majorizing functions are analytic, we have that the Taylor expansions of $v_1,v_2$ converge uniformly as desired.
\end{proof}

\subsection{Proof of Lemma \ref{lem:normalform}}
\label{sec:normal_form_appendix}
\begin{proof}
  We consider the polytropic Euler-Poisson system \eqref{polytropicselfsimeq}. We note that the equation \eqref{intermedstep0} is equivalent to
  \begin{equation}
    \label{eq:ystardef}
    y_*\o_0=\g^{\frac{1}{2}}\rho_0^{\frac{\g-1}{2}}.
  \end{equation}
  We also will define
  \begin{equation*}
    A(y,\rt,\ut):=\mqty{\ut+(2-\g)y & \rt\\ \g\rt^{\g-2} & \ut+(2-\g)y},\qquad B(y,\rt,\ut):=\mqty{\frac{2\rt(\ut+y)}{y}\\ (\g-1)\ut+\frac{4\pi}{4-3\g}\rt(\ut+(2-\g)y)}.
  \end{equation*}
  We note that
  \begin{align*}
    A_0&:=A(y_*,\rho_0,u_0)=\mqty{y_* \o_0 & \rho_0\\ \frac{y_*^2 \o_0^2}{\rho_0} & y_*\o_0},\\
    B_0&:=B(y_*,\rho_0,u_0)=\mqty{2\rho_0\o_0+2(\g-1)\rho_0\\ (\g-1)\o_0y_*-(2-\g)(\g-1)y_*+\frac{4\pi}{4-3\g}\rho_0 y_*\o_0}.
  \end{align*}
  We also note that assuming $\rt-\rho_0,\ut-u_0=O(y-y_*)$, we have that
  \begin{align*}
    A(y,\rt,\ut)=A_0+\mqty{(2-\g)(y-y_*)+(u-u_0) & \rt-\rho_0\\(\g-2)\frac{(y_* \o_0)^2}{\rho_0^2}(\rt-\rho_0) & (2-\g)(y-y_*)+(u-u_0)}+O(y-y_*)^2.
  \end{align*}
  Similarly,
  \begin{multline*}
    B(y,\rt,\ut)=B_0\\
    +\mqty{2(\o_0+(\g-1))(\rt-\rho_0)+\frac{2\rho_0}{y_*}(\ut-u_0)-\frac{2\rho_0}{y_*}(\o_0-(2-\g))(y-y_*)\\(\g-1)(\ut-u_0)+\frac{4\pi}{4-3\g}(\rt-\rho_0)y_* \o_0+\frac{4\pi}{4-3\g}\rho_0(\ut-u_0)+\frac{4\pi(2-\g)}{4-3\g}\rho_0(y-y_*)}+O((y-y_*)^2).
  \end{multline*}
  We will then define the matrices
  \begin{equation*}
    C:=\mqty{\frac{1}{4y_* \o_0} & \frac{\rho_0}{4(y_*\o_0)^2}\\ \frac{1}{4y_* \o_0} & \frac{-\rho_0}{4(y_* \o_0)^2}},\qquad D:=\mqty{1 & 1\\\frac{y_*\o_0}{\rho_0} & -\frac{y_* \o_0}{\rho_0}}.
  \end{equation*}
  We will now perform a first change of variables
  \begin{equation*}
    \mqty{\rt\\ \ut}=\mqty{\rho_0\\ u_0}+D\mqty{\rs\\ \us}.
  \end{equation*}
  We will implement this change of variables, and then multiply the equation by $C$. We define
  \begin{equation*}
    \As=C\cdot A(y,\rho_0+D_{11}\rs+D_{12}\us,u_0+D_{21}\rs+D_{22}\us),\qquad \Bs=C\cdot B(y,\rho_0+D_{11}\rs+D_{12}\us,u_0D_{21}\rs+D_{22}\us)
  \end{equation*}
  so that the system can be written as
  \begin{equation*}
    \As(y,\rs,\us)\mqty{\rs'\\\us'}+\Bs(y,\rs,\us)=0.
  \end{equation*}
  We will expand (assuming $\rho,u=O(y-y_*)$)
  \begin{equation*}
    \As=\As_0+\As_1+O((y-y_*)^2),\qquad \Bs=\Bs_0+\Bs_1+O((y-y_*)^2).
  \end{equation*}
  where
  \begin{equation*}
    \As_1=\mqty{\frac{2(2-\g)}{4y_*\o_0}(y-y_*)+\frac{(\g+1)}{4\rho_0}\rs+\frac{(\g-3)}{4\rho_0}\us& \frac{(\g-3)}{4\rho_0}\rs+\frac{(\g-3)}{4\rho_0}\us\\ -\frac{(\g-3)}{4\rho_0}\rs-\frac{(\g-3)}{4\rho_0}\us & \frac{2(2-\g)}{4y_* \o_0}(y-y_*)-\frac{(\g-3)}{4\rho_0}\rs-\frac{(\g+1)}{4\rho_0}\us}.
  \end{equation*}
  We note that we have
  \begin{equation*}
    \As_0=C A_0 D=\mqty{1 & 0\\ 0 & 0},\qquad \Bs_0=C B_0=\frac{\rho_0}{4 y_* \o_0^2}\mqty{2\o_0^2+3(\g-1)\o_0-(2-\g)(\g-1)+\frac{4\pi\rho_0\o_0}{4-3\g}\\ 2\o_0^2+(\g-1)\o_0+(2-\g)(\g-1)-\frac{4\pi \rho_0\o_0}{4-3\g}}.
  \end{equation*}
  We then simplify using \eqref{eq:rho0sonic} to write
  \begin{equation*}
    \Bs_0=\mqty{\frac{\rho_0}{y_*\o_0}(\o_0+(\g-1))\\ 0}.
  \end{equation*}
  We also have
  \begin{equation*}
    \Bs_1=\frac{1}{4y_* \o_0}\mqty{(4\o_0+3(\g-1))\rs+(\g-1)\us-\frac{2\rho_0}{y_*}(\o_0-(2-\g))(y-y_*)+\frac{4\pi \rho_0^2}{4-3\g}(\frac{2\rs}{\rho_0}+\frac{(2-\g)(y-y_*)}{y_*\o_0})\\(4\o_0+(\g-1))\rs+3(\g-1)\us-\frac{2\rho_0}{y_*}(\o_0-(2-\g))(y-y_*)-\frac{4\pi \rho_0^2}{4-3\g}(\frac{2\rs}{\rho_0}+\frac{(2-\g)(y-y_*)}{y_*\o_0})}.
  \end{equation*}
  We then comment that the first line, evaluated at $y=y_*$ gives
  \begin{equation*}
    \rs'(y_*)+\frac{\rho_0}{y_*\o_0}(\o_0+(\g-1))=0.
  \end{equation*}
  We then do an addition change of variables from $\rs,\us$ to $\rc,\uc$ where
  \begin{equation*}
    \rs=\rc(y-y_*)-\frac{\rho_0}{y_* \o_0}(\o_0+(\g-1))\cdot (y-y_*),\qquad \us=\uc(y-y_*).
  \end{equation*}
  We then define
  \begin{align*}
    \Ac(y,\rc,\uc)&=\As(y,\rc-\frac{\rho_0}{y_* \o_0}(\o_0+(\g-1))(y-y_*),\uc),\\
    \Bc(y,\rc,\uc)&=\Bs(y,\rc-\frac{\rho_0}{y_* \o_0}(\o_0+(\g-1))(y-y_*),\uc)+\Ac\cdot \mqty{-\frac{\rho_0}{y_* \o_0}(\o_0+(\g-1))\\0}
  \end{align*}
  and then note that we get
  \begin{equation*}
    \Ac\mqty{\rc'\\ \uc'}+\Bc=0.
  \end{equation*}
  We note that now
  \begin{equation*}
    \Bc_0=\mqty{0\\0}.
  \end{equation*}
  We now compute
  \begin{equation*}
    \Ac_1=\mqty{\frac{2(2-\g)-(\g+1)(\o_0+(\g-1))}{4y_*\o_0}y+\frac{(\g+1)}{4\rho_0}\rc+\frac{(\g-3)}{4\rho_0}\uc & \frac{(\g-3)}{4\rho_0}\rc+\frac{(\g-3)}{4\rho_0}\uc-\frac{(\g-3)(\o_0+(\g-1))}{4y_*\o_0}y\\ -\frac{(\g-3)}{4\rho_0}\rc-\frac{(\g-3)}{4\rho_0}\uc +\frac{(\g-3)(\o_0+(\g-1))}{4y_*\o_0}y& \frac{2(2-\g)+(\g-3)(\o_0+(\g-1))}{4y_* \o_0}y-\frac{(\g-3)}{4\rho_0}\rc-\frac{(\g+1)}{4\rho_0}\uc}.
  \end{equation*}
  We then compute
  \begin{multline*}
    \Bc_1=\frac{1}{4y_* \o_0}\mqty{(4\o_0+3(\g-1))\rc+(\g-1)\uc+\frac{\rho_0(-2\o_0(\o_0-(2-\g))-(4\o_0+3(\g-1))(\o_0+(\g-1)))}{\o_0y_*}y\\(4\o_0+(\g-1))\rc+3(\g-1)\uc+\frac{\rho_0(-2\o_0(\o_0-(2-\g))-(4\o_0+(\g-1))(\o_0+(\g-1)))}{\o_0y_*}y}\\
    +\frac{1}{4y_* \o_0}\mqty{\frac{4\pi \rho_0^2}{4-3\g}(\frac{2\rc}{\rho_0}+\frac{\left((2-\g)-2(\o_0+(\g-1))\right)y}{y_*\o_0})\\
      -\frac{4\pi \rho_0^2}{4-3\g}(\frac{2\rc}{\rho_0}+\frac{\left((2-\g)-2(\o_0+(\g-1))\right)y}{y_*\o_0})}\\
    +\frac{(\o_0+(\g-1))}{4y_* \o_0}\mqty{-\frac{\rho_0}{y_*\o_0}(2(2-\g)-(\g+1)(\o_0+(\g-1)))y-\frac{\g+1}{y_*\o_0}\rc-\frac{\g-3}{y_*\o_0}\uc\\ -\frac{\rho_0}{y_*\o_0}(\g-3)(\o_0+(\g-1))y+(\g-3)\rc+(\g-3)\uc}.
  \end{multline*}
  From these we can read off $a,b,c,d$. Then, examining the change of variables gives
  \begin{equation*}
    \begin{cases}
      \rt'(y_*)=\rc'(0)+\uc'(0)-\frac{\rho_0}{y_* \o_0}(\o_0+(\g-1))\\
      \ut'(y_*)=\frac{y_*\o_0}{\rho_0}(\rc'(0)-\uc'(0)-\frac{\rho_0}{y_* \o_0}(\o_0+(\g-1)))\\
    \end{cases},
  \end{equation*}
  which yields the stated equations for $R,W$ in terms of $U$. Finally, according to \cite{ghjs}, the Larson-Penston-Hunter branch corresponds to the left-most root of the quadratic (in $R$), which corresponds to the branch stated in Lemma \ref{lem:normalform} since $b<0$.
  We next confirm that the far-field solution has the Larson-Penston-Hunter taylor expansion. Indeed, if we substitute $\o_0=2-\g,\rho_0=\frac{4-3\g}{2\pi}$ then we have
  We also note that we will have
  \begin{equation*}
    a+d=\frac{3\g-5}{4y_f \o_0},\qquad b=\frac{-(\g+1)}{4\rho_0},\qquad c=\frac{4(\g-1)\rho_0}{4 y_*^2\o_0^2}.
  \end{equation*}
  if we substitute this into \eqref{eq:lph_branch}, we obtain
  \begin{equation*}
    U_{farfield}=-\frac{\rho_0}{y_* \o_0}.
  \end{equation*}
  which if we substitute back in terms of $R$, we obtain
  \begin{equation*}
    \frac{y_f \rt'(y_f))}{\rho_0}=-\frac{2}{2-\g}.
  \end{equation*}
  which is consistent with the derivative of the far-field solution. We can then compute
  \begin{equation*}
    a+bU=\frac{2}{4y_f\o_0},\qquad d+bU=\frac{5(\g-1)}{4 y_f \o_0}.
  \end{equation*}
\end{proof}

\subsection{On matching with Lipschitz dependence on parameters}
\label{implicit_proof}
We will first prove a variant of the Inverse Function Theorem (where we replace the usual assumption of $C^1$ differentiability with a quantitative lipschitz bound). This is then used to prove a corresponding variant of the Implicit function theorem which is used in Section \ref{sec:matching_subsection}.

\begin{lem}
  \label{lem:inverse_function_theorem}
  Let $B_{\ep}(0)\subset \R^d$ denote the closed ball of radius $\ep$. Consider $f\colon B_{\ep}(0) \to \R^d$ of the form
  \[
    f(x)=Ax+g(x)
  \]
  where $A$ is an invertible matrix and $g$ is a Lipschitz function satisfying
  \[
    g(0)=0,\qquad \abs{g(x_1)-g(x_2)}\leq L_1 \abs{x_1-x_2}.
  \]
  We define
  \[
    a=\norm{A^{-1}}^{-1}.
  \]
  Suppose that $L_1<a$. Let $\de=\ep a(1-\frac{L_1}{a})$. Then for each $y\in B_{\de}(0)$, there exists a unique $x$ in $B_{\ep}(0)$ such that $f(x)=y$. Moreover, we have that $f^{-1}$ is Lipschitz with
  \begin{equation}
    \abs{\frac{f^{-1}(y_1)-f^{-1}(y_2)}{\abs{y_1-y_2}}}\leq \frac{1}{a}\frac{1}{1-\frac{L_1}{a}}.
    \label{inverselipschitz}
  \end{equation}  
\end{lem}

\begin{proof}
  For a fixed $y_0\in B_{\de}(0)$, we define $T_{y_0}(x)\colon B_\ep(0)\to \R^d$ by
  \[
    T_{y_0}(x)=A^{-1}(y_0-g(x)).
  \]
  We comment that if $x_1,x_2\in B_{\ep}(0)$ then 
  \[
    \abs{T_{y_0}(x)}\leq \frac{\abs{y_0}}{a}+\frac{L_1}{a}\abs{x},\qquad \abs{T_{y_0}(x_1)-T_{y_0}(x_2)}\leq \frac{L_1}{a}\abs{x_1-x_2}
  \]
  In particular, we see that if $\de=\ep a(1-\frac{L_1}{a})$, then if $\abs{y_0}\leq \de$, $\{T_{y_0}^{n}(0)\}_n$ will remain in $B_{\ep}(0)$ and will converge. Thus, $f^{-1}$ is well-defined. We then note that
  \[
    y_1-y_2=A(x_1-x_2)+g(x_1)-g(x_2).
  \]
  We may then solve for $x_1-x_2$ as
  \[
    x_1-x_2=A^{-1}(y_1-y_2)-A^{-1}(g(x_1)-g(x_2))
  \]
  We thus have
  \[
    \abs{x_1-x_2}\leq \frac{1}{a}\abs{y_1-y_2}+\frac{L_1}{a}\abs{x_1-x_2}
  \]
  from which we obtain \eqref{inverselipschitz}.
\end{proof}
\begin{lem}
  \label{lem:implicit_function_theorem}
  Let $B_{\ep}(0)\subset \R^2$ denote the closed ball of radius $\ep$. Consider $f\colon B_{\ep}(0)\to \R$ of the form
  \begin{equation*}
    f(x,y)=a x+b y+g(x,y).
  \end{equation*}
  Suppose $a\neq 0$. Suppose that, on $B_{\ep}$, $g(x,y)$ satisfies the Lipschitz estimate
  \begin{equation*}
    \abs{g(x_1,y_1)-g(x_2,y_2)}\leq L_2\sqrt{(x_1-x_2)^2+(y_1-y_2)^2}.
  \end{equation*}
  Suppose that $L_2$ satisfies
  \begin{equation*}
    L_2<\frac{\abs{a}}{1+\abs{\frac{b}{a}}}.
  \end{equation*}
  Then, on a neighborhood of $0$, the equation $f(x,y)=0$ uniquely determines $x=\phi(y)$ for a Lipschitz function $\phi(y)$. 
\end{lem}
\begin{proof}
  We let
  \[
    \tilde{f}(x,y)=\mqty{a & b\\ 0& a}\mqty{x\\ y}+\mqty{g(x,y)\\ 0}.
  \]
  We recall that
  \[
    \norm{\mqty{a & b\\ 0 & a}}=\sqrt{a^2+\frac{b^2}{2}+\frac{\abs{b}}{2}\sqrt{4a^2+b^2}}.
  \]
  Similarly,
  \[
    \norm{\mqty{a & b\\ 0 & a}^{-1}}=\sqrt{a^{-2}+\frac{1}{2}b^2 a^{-4}+\frac{\abs{b}}{2a^2}\sqrt{\frac{4}{a^2}+\frac{b^2}{a^4}}}.
  \]
  We define
  \begin{align*}
    \la_1&=\frac{1}{\sqrt{a^{-2}+\frac{1}{2}b^2 a^{-4}+\frac{\abs{b}}{2a^2}\sqrt{\frac{4}{a^2}+\frac{b^2}{a^4}}}}\\
         &=\frac{\abs{a}}{\sqrt{1+\frac{1}{2}(\frac{b}{a})^2+\abs{\frac{b}{a}}\sqrt{4+(\frac{b}{a})^2}}}\\
         &\geq \frac{1}{1+\abs{\frac{b}{a}}}.
  \end{align*}
  It follows that if $L_2<\la_1$, then by Lemma \ref{lem:inverse_function_theorem} there is a Lipschitz inverse function
  \[
    \tilde{f}^{-1}(z_1,z_2)=\mqty{x(z_1,z_2)\\ y(z_1,z_2)}.
  \]
  By $f\circ \tilde{f}^{-1}=\id$, we have
  \[
    f(x(z_1,z_2),y(z_1,z_2))=z_1,\qquad a y(z_1,z_2)=z_2
  \]
  It follows that $y(z_1,z_2)=\frac{z_2}{a}$. We then have
  \[
    f(x(z_1,z_2),\frac{z_2}{a})=z_1.
  \]
  If we substitute $z_1=0$, we have
  \[
    f(x(0,z_2),\frac{z_2}{a})=0.
  \]
  Thus, if we define
  \[
    \phi(y):=x(0,a\cdot y),
  \]
  then we note that we have $f(\phi(y),y)=0$. Thus, we see that we may parameterize the curve with respect to $y$.
\end{proof}
\printbibliography
\end{document}